\documentclass[10pt]{article}   

\usepackage{graphicx}

\usepackage{latexsym,amsfonts,amsmath,theorem,amssymb}
\usepackage{stmaryrd,array,tabularx,bbm}
\usepackage{pstricks,graphicx}

\usepackage{a4wide}

\usepackage{color}
\usepackage{caption}
\usepackage{theorem}
\usepackage{cleveref}
\usepackage{framed}
\usepackage{blindtext}
\usepackage{float}
\usepackage{todonotes}
\usepackage{graphicx,caption}

\usepackage[caption=false]{subfig}
\usepackage{algorithm}
\usepackage{algorithmicx}
\usepackage[noend]{algpseudocode}
\usepackage{tikz}
\usepackage{caption}

\algrenewcommand\algorithmicrequire{\textbf{Input:}}
\algrenewcommand\algorithmicensure{\textbf{Output:}}
\newtheorem{theorem}{Theorem}[section]
\newtheorem{lemma}[theorem]{Lemma}
\newtheorem{proposition}[theorem]{Proposition}
\newtheorem{assumption}[theorem]{Assumption}
\newtheorem{definition}[theorem]{Definition}

\newenvironment{proof}{\begin{trivlist}
    \item[\hskip\labelsep{\bf Proof.}]}{$\hfill\Box$\end{trivlist}}

{\theoremstyle{plain} \theorembodyfont{\rmfamily}

\newtheorem{remark}[theorem]{Remark}}

\numberwithin{equation}{section}
\numberwithin{figure}{section}
\numberwithin{table}{section}

\newcommand{\bsi}{{\boldsymbol{i}}}

\newcommand{\Nsub}{{N_{\rm sub}}}
\newcommand{\Nens}{{N_{\rm ens}}}

\newcommand{\Nobs}{{N_{\rm obs}}}
\newcommand{\comment}[1]{}

\newcommand{\argmin}{{\mathrm{argmin}}}
\newcommand{\spann}{{\mathrm{span}}}
\newcommand{\hPsi}{{\overline{\mathbf{F}}}}

\renewcommand{\hat}{\widehat}
\newcommand{\vertiii}[1]{{\left\vert\kern-0.25ex\left\vert\kern-0.25ex\left\vert #1 
\right\vert\kern-0.25ex\right\vert\kern-0.25ex\right\vert}}

\graphicspath{{./figures/}}

\title{Subsampling in Ensemble Kalman Inversion}
\author{Matei Hanu\footnote{Fachbereich Mathematik und Informatik, Freie Universit\"at Berlin, Arnimallee 6, 14195 Berlin, Germany, $\{$matei.hanu, c.schillings$\}$@fu-berlin.de} \and Jonas Latz\footnote{Department of Mathematics, University of Manchester, United Kingdom, jonas.latz@manchester.ac.uk} \and Claudia Schillings\footnotemark[1]}
\date{}

\begin{document}
\maketitle
\abstract{We consider the Ensemble Kalman Inversion which has been recently introduced as an efficient, gradient-free optimisation method to estimate unknown parameters in an inverse setting. In the case of large data sets, the Ensemble Kalman Inversion becomes computationally infeasible as the data misfit needs to be evaluated for each particle in each iteration.
Here, randomised algorithms like stochastic gradient descent have been demonstrated to successfully overcome this issue by using only a random subset of the data in each iteration, so-called subsampling techniques.
Based on a recent analysis of a continuous-time representation of stochastic gradient methods, we propose, analyse, and apply subsampling-techniques within Ensemble Kalman Inversion. Indeed, we propose two different subsampling techniques: either every particle observes the same data subset (single subsampling) or every particle observes a different data subset (batch subsampling).}
\section{Introduction}
A large variety of physical, biological and social systems and processes have been described by mathematical models. Those models can be used to analyse and predict the behaviour of the associated processes and systems. In case, a model shall be employed to describe a particular system, the model needs to be calibrated with respect to observation of that particular system. This calibration process that fits the model to data is often called \emph{inversion}.
Inversion forms the basis for, e.g., numerical weather prediction, medical image processing, and many machine learning methods. Several inversion techniques have been proposed and studied: we often distinguish variational/optimisation-based approaches and Bayesian/statistical approaches. Throughout this work, we consider a class of methods that can be seen as being in between those two approaches: the \emph{Ensemble Kalman Inversion} (EKI) framework going back to \cite{Iglesias_2013,Schillings2016}. EKI is based on an Ensemble Kalman-Bucy Filter that is iteratively applied to solve an inverse problem. In the linear setting the resulting algorithm is usually given in the form of a preconditioned gradient flow, that is an ordinary differential equation describing the dynamics of an ensemble of particles.

EKI becomes computationally infeasible, if the considered amount of data is too large: the data cannot be stored in the memory at once and, thus, EKI is not applicable.
The same problem arises also in other traditional optimisation algorithms, like gradient descent or the Gauss--Newton method. In the past decades, randomised algorithms that optimise in each time step only with respect to a subsample of the data set have become popular. A subsample is a (often randomly chosen) subset of the considered data set. The foundation for all of these stochastic optimisation algorithms is the \emph{stochastic gradient descent} (SGD) algorithm going back to \cite{Robbins1951}. Stochastic gradient descent and its variants have become especially popular in the machine learning community. 

The idea of randomised subsampling in the EKI framework has been proposed in \cite{Kovachki2019}, where it is, indeed, applied to train a neural network. There, the subsampling has been introduced after discretising the preconditioned gradient flow, but no analysis has been presented.
A recent work by \cite{Latz2021} explains how subsampling can be represented in continuous-time settings and how these can be analysed. In the present work, we aim at using this theory in the context of EKI to analyse subsampling methodology at the ODE level. 
\subsection{Literature Overview}
Since its introduction in \cite{Evensen2003} the Ensemble Kalman Filter (EnKF) has been widely used in both inverse problems as well as data assimilation problems. The EnKF is very appealing for many applications due to its straightforward implementation and robustness w.r. to small ensemble sizes \cite{Bergemann2009,Bergemann2010,Iglesias2014,Iglesias2016,Iglesias_2013,Li2009}. Stability results can be found in \cite{Tong2015,Tong2016}. Convergence analysis based on the continuous time limit of the Ensemble Kalman Inversion(EKI) has been developed in \cite{Bloemker2019,Bloemker2021,Bungert2021,Schillings2017,Schillings2016}. However, to obtain convergence results in the parameter space some form of regularisation is usually needed. We mainly consider Tikhonov regularisation which was analysed for the EKI in \cite{Tong2020} for example. Recently there has been further analysis on Tikhonov regularisation for the stochastic EKI as well as adaptive Tikhonov strategies to improve the original variant \cite{Weissmann2022}. Considering large ensemble sizes, an analysis of the mean-field limit is presented in \cite{Stuart2022,Ding2020}.\\
A historical overview of the Kalman filter and some of its extensions can be found \cite{Stuart2022}.\\\
After their introduction by Robbins and Monro \cite{Robbins1951}, the stochastic gradient descent method has in the recent past been further analysed by, e.g., \cite{Bottou2016}. Stochastic gradient descent is often computationally advantageous compared to normal gradient descent \cite{Nocedal2006} due to computational efficiency as well as being able to escape local minimisers in non-convex optimisation problems \cite{Choromanska2014,Vidal2017}. As mentioned earlier, the theory employed in this work is based on the continuous-time analysis of stochastic gradient descent  by Latz  \cite{Latz2021} that was further generalised in \cite{Jin2021}, but is somewhat orthogonal to the diffusion-based continuous-time analysis of SGD of, e.g.,  \cite{LiTaiE}. 
\subsection{Contributions and outline}
In the following, we focus on the case, where the data misfit is computationally infeasible due to large data. Inspired by the success story of randomised gradient descent methods, we introduce subsampling strategies to EKI to ensure feasibility of the method also in the large data regime. We summarise our contribution below:
\begin{enumerate}
    \item We introduce two subsampling techniques for EKI: single subsampling and batch subsampling.
    \item We present an analysis of the subsampling techniques for linear forward operators, in particular we analyse stability of the subsampling schemes and give conditions under which we obtain asymptotic stability. Indeed, the resulting dynamical system approximates the EKI solution.
    \item We illustrate our results with two examples: estimation of the source term for an elliptic partial differential equation and estimation of the diffusion coefficient for a parabolic partial differential equation.
\end{enumerate}
Whilst analysing the subsampling EKI, we generalise some results from \cite{Latz2021} to more general flows. These generalisations may be of independent interest.

This work is structured as follows. We introduce problem setting and EKI methods in Section~\ref{sec:background}. We discuss stochastic approximations of certain flows in general and the subsampling in EKI in particular in Section~\ref{sec_subsampling_flows}; before analysing them in Section~\ref{sec_subsEKI}. We show numerical examples in Section~\ref{sec_NumExp} and conclude the work in Section~\ref{sec_conclu}.

\section{Problem setting and mathematical background} \label{sec:background}
Let $(\Omega, \mathcal{A}, \mathbb{P})$ be a probability space, $X$ be a separable Hilbert space and $Y := \mathbb{R}^\Nobs$, with $\Nobs \in \mathbb{N} = \{1,2,\ldots\}$. We will refer to $X$ as \emph{parameter space} and to $Y$ as \emph{data space}. Let now $n, m \in \mathbb{N}$. We sometimes associate finite-dimensional spaces $\mathbb{R}^n$ with the basic inner product $\langle \cdot, \cdot \rangle$ and the associated Euclidean norm $\| \cdot \|$ or the weighted inner product $\langle \cdot, \cdot \rangle_{\Gamma} := \langle \cdot, \Gamma^{-1} \cdot \rangle $ and its associated weighted norm $\| \cdot \|_\Gamma$, where $\Gamma \in \mathbb{R}^{n \times n}$ is symmetric positive definite. Further we denote by $\mathcal{B}X:=\mathcal{B}(X,\|\cdot\|)$ (or respectively $\mathcal{B}X:=\mathcal{B}(X,\|\cdot\|_\Gamma)$), the Borel-$\sigma$-algebra on X. Given an additional space $\mathbb{R}^m$ we define the tensor product of vectors in $x \in \mathbb{R}^n$ and $y \in \mathbb{R}^m$ by $x \otimes y := x y^\top$.

In the following, we focus on linear inverse problems of the form
\begin{align} \label{eq_IP}
    A\theta^\dagger + \eta^\dagger = y^\dagger,
\end{align}
where $\theta^\dagger \in X$ is the \emph{true parameter}, $y^\dagger \in Y$ is the \emph{observed data set}, $\eta^\dagger \in Y$ is \emph{observational noise}, and $A: X \rightarrow Y$ is a compact operator. In a Bayesian setting, we model $\theta^\dagger, \eta^\dagger$ as random variables $\theta: \Omega \rightarrow X$ and $\eta: \Omega \rightarrow Y$, where $\theta \perp \eta$. Assuming that the noise is normally distributed, i.e. $\eta \sim \mathrm{N}(0, \Gamma)$ and non-degenerate, the posterior $\mu^y$ is characterised through Bayes' formula:
\begin{equation}
    \mathrm d \mu^y(\theta)=\frac{1}{Z}\exp\left(-\frac 12 \|y-A\theta\|^2_{\Gamma}\right)\mathrm d \mu_0(\theta)\,,\notag
\end{equation}
where $\mu_0$ denotes the prior distribution on the unknown parameters and $Z=\mathbb E_{\mu_0} \exp(-\frac 12 \|y-A\theta\|_{\Gamma}) $ is the normalization constant. We will focus in the following on the computation of a point estimate of the unknown parameters, the maximum aposteriori (MAP)  estimate, which is a minimiser of a regularised version of the potential
\begin{equation}\label{eq_pot_full}
    \Phi(\theta)=\Phi(\theta;y)=\frac 12 \|y-A\theta\|^2_{\Gamma}\,.
\end{equation}
In order to handle large data settings, i.e. $\Nobs$ large, we introduce a subsampling strategy, i.e., we partition the data $y^\dagger$ into multiple subsets into $\Nsub$ subsets $y_1^\dagger,\ldots,y_\Nsub^\dagger$, such that $(y_1^\dagger,\ldots,y_\Nsub^\dagger) = y^\dagger$,  $\Nsub \in \mathbb{N}, \Nsub \geq 2$, and $I := \{1,\ldots,\Nsub\}$. To this end, we define \emph{data subspaces} $Y_1,\ldots, Y_\Nsub$, such that $Y := \prod_{i \in I} Y_i$. Moreover, we assume that the noise $\eta$ has independent entries with respect to this splitting of the data space $Y$. In particular, we assume that there are covariance matrices $\Gamma_i : Y_i \rightarrow Y_i$, $i \in I$, such that $\Gamma$ has the following block diagonal structure:
$$
\Gamma = \begin{pmatrix} \Gamma_1 & & & \\ & \Gamma_2 & & \\ & & \ddots & \\ & & & \Gamma_\Nsub \end{pmatrix}.
$$
Finally, we split the operator $A$ into a family of $(A_i)_{i \in I}$, where
$$
A = \begin{pmatrix} A_1 \\ \vdots \\ A_{\Nsub}\end{pmatrix}.
$$
Then, we can equivalently represent the inverse problem \eqref{eq_IP}
by the family of inverse problems
\begin{align*}
    A_1\theta^\dagger + \eta^\dagger_1 &= y^\dagger_1 \notag\\
    \vdots& \\
     A_\Nsub\theta^\dagger + \eta^\dagger_\Nsub &= y^\dagger_\Nsub \notag,
\end{align*}
where $\eta^\dagger_i$ is a realisation of $\eta_i \sim \mathrm{N}(0, \Gamma_i)$ for $i \in I$. 
\begin{remark}
Note that the diagonal structure of the noise can always be guaranteed by multiplying \eqref{eq_IP} with the inverse square root of $\Gamma$. To simplify notation we will, w.l.o.g., assume for the remaining discussion that $\Gamma=\mathrm{Id}_k$ and correspondingly $\Gamma_i=\mathrm{Id}_{k_i}$.
\end{remark}
We will then consider the potentials of the respective data subset 
\begin{align*}
 \Phi_i(\theta) &:= \frac{1}{2} \|A_i \theta - y^\dagger_i \|^2 \qquad (i \in I).
\end{align*}

To overcome the ill-posedness of the inverse problem, we often consider a regularised version of the potential in the form of
\begin{equation}
    \Phi^{\mbox{\rm reg}}(\theta):=\Phi^{\mbox{\rm reg}}(\theta;y):=\frac 12 \|y-A\theta\|^2+\frac{\alpha}{2} \|\theta\|^2_{C_0}\,,\notag
\end{equation}
where $C_0$ is a self-adjoint, trace-class operator and $\alpha>0$. This corresponds to the MAP estimate in case of a Gaussian prior with covariance $\alpha C_0$, cp. \cite{Law2018, Tong2020}.  Assuming a Gaussian prior distribution with mean equal to zero, we can incorporate the regularisation via 
\begin{equation}
\tilde A=\begin{pmatrix}A \\ \left(\alpha C_0\right)^{\frac12}\end{pmatrix},\quad  \tilde y=\begin{pmatrix}y^\dagger \\ 0\end{pmatrix}\notag
\end{equation}
allowing us to write
\begin{equation}\label{eq:minreg}
\Phi^{\mbox{\rm reg}}(\theta)=\frac 12 \|\tilde y-\tilde A\theta\|^2\,.
\end{equation}
Note that the forward operator $A$ is usually not injective, whereas the regularisation results in an injective operator $\tilde A$.\\
Similarly, when we split the forward operator we consider
\begin{equation}
\tilde A_i=\begin{pmatrix}A_i \\ \left(\frac{\alpha}{\Nsub} C_0\right)^{\frac12}\end{pmatrix},\quad  \tilde y_i=\begin{pmatrix}y_i^\dagger \\ 0\end{pmatrix}\notag
\end{equation}
leading to the family of potentials
\begin{equation}
\Phi_i^{\mbox{\rm reg}}(\theta)=\frac 12 \|\tilde y_i-\tilde A_i\theta\|^2,\quad (i \in I).\notag
\end{equation}
that satisfies
\begin{equation}
\Phi^{\mbox{\rm reg}}(\theta)=\sum_{i=1}^\Nsub\Phi_i^{\mbox{\rm reg}}(\theta).\notag
\end{equation}

\subsection{Ensemble Kalman inversion and its variants}
We aim to solve the inverse problem \eqref{eq_IP} using the Ensemble Kalman inversion (EKI) framework. We will focus in the following on the continuous-time limit of the Kalman inversion, cp. \cite{Schillings2017}.

We define the initial ensemble to be $\theta_0 = (\theta_0^{(j)})_{j \in J} \in X^{\Nens}$ assuming w.l.o.g. that $(\theta_0^{(j)}-\bar{\theta}_0)_{j \in J}$ is a linearly independent family, with $\Nens \in \mathbb{N}, \Nens \geq 2$, and $J := \{1,\ldots,\Nens\}$.
The \emph{basic Ensemble Kalman inversion} proceeds by moving the particles in the parameter space according to the following dynamical system
\begin{align} \label{EKI_basic}
    \frac{\mathrm{d} \theta^{(j)}(t)}{\mathrm{d}t} &= - \widehat{C^{\theta y}}_t (A\theta^{(j)}(t)-y^\dagger) \qquad (j \in J)\\ \theta(0) &= \theta_0, \notag
\end{align}
where $$\widehat{C^{\theta y}}_t := \frac{1}{\Nens-1} \sum_{j = 1}^\Nens (\theta^{(j)}(t) - \overline{\theta}(t)) \otimes (A\theta^{(j)}(t) - A\overline{\theta}(t)), \qquad \overline{\theta}(t) = \frac{1}{\Nens}\sum_{j=1}^{\Nens}\theta^{(j)} \qquad (t \geq 0).$$
The linearity of the forward model leads to the following equivalent reformulation of the dynamical system
\begin{align} \label{EKI_basic_lin}
    \frac{\mathrm{d} \theta^{(j)}(t)}{\mathrm{d}t} &= - \widehat{C}_t D_\theta \Phi(\theta^{(j)}(t)) \qquad (j \in J)\\ \theta(0) &= \theta_0, \notag
\end{align}
where $$\widehat{C}_t := \frac{1}{\Nens-1} \sum_{j = 1}^\Nens (\theta^{(j)}(t) - \overline{\theta}(t)) \otimes (\theta^{(j)}(t) - \overline{\theta}(t)).$$
Intuitively, the dynamic represents parallel gradient flows minimizing \eqref{eq_pot_full} which are coupled through the empirical covariance $\widehat{C}_t$. This empirical covariance can be viewed as a preconditioner.
The following reformulation of the right hand side 
\begin{align*} 
    \frac{\mathrm{d} \theta^{(j)}(t)}{\mathrm{d}t} &= - \sum_{k=1}^\Nens \langle A\theta^{(k)}(t)-A\overline{\theta(t)},A\theta^{(j)}(t)-y^\dagger\rangle (\theta^{(k)}(t)-\overline{\theta(t)})\qquad (j \in J)\\ \theta(0) &= \theta_0
\end{align*}
reveals the so-called \emph{subspace property}, i.e. the particles ${(\theta_t^{(j)})_{j \in J}}$ lie in the span of the initial ensemble ${(\theta_0^{(j)})_{j \in J}}$ for any time $t\geq 0$, cp. \cite{Iglesias_2013}. Hence, we can assume that the parameter space is finite dimensional with ${X:=\mathbb{R}^{\Nens}}$ and will, w.l.o.g., do so in the following.

We define the \emph{Tikhonov-regularised Ensemble Kalman inversion} (TEKI) by the solution of the following ODE
\begin{align}\label{eki_regul_lin}
    \frac{\mathrm{d} \theta^{(j)}(t)}{\mathrm{d}t} &= - \widehat{C}_t D_\theta \Phi^{\mbox{\rm reg}}(\theta^{(j)}(t)) \qquad (j \in J) \\ \theta(0) &= \theta_0 \notag.
\end{align}

The ensemble of particles converges to the empirical mean following the dynamics given by \eqref{eki_regul_lin}. This results in the so-called ensemble collapse, and thus the degeneracy of the ensemble covariance operator, which results in an algebraic convergence rate rather than an exponential rate (compared to gradient flows).
\emph{Variance inflation} is a technique mitigating this effect by adding an operator to the degenerate $\widehat{C}_t$. Let $C_{\rm vi}: X \rightarrow X$ be a covariance operator on $X$. We define the \emph{variance-inflated Ensemble Kalman inversion} as the solution of the ODE
\begin{align} \label{EKI_VI}
    \frac{\mathrm{d} \theta^{(j)}(t)}{\mathrm{d}t} &= - (\widehat{C}_t + \alpha_{\rm vi} C_{\rm vi}) D_\theta \Phi(\theta^{(j)}(t)) \qquad (j \in J) \\ \theta(0) &= \theta_0, \notag
\end{align}
for $\alpha_{\rm vi}>0$. 

By replacing $\Phi$ by $\Phi^{\mbox{\rm reg}}$ in \eqref{EKI_VI}, one obtains the \emph{variance-inflated Tikhonov-regularised Ensemble Kalman inversion}
\begin{align*} 
    \frac{\mathrm{d} \theta^{(j)}(t)}{\mathrm{d}t} &= - (\widehat{C}_t + \alpha_{\rm vi} C_{\rm vi}) D_\theta \Phi^{\mbox{\rm reg}}(\theta^{(j)}(t)) \qquad (j \in J) \\ \theta(0) &= \theta_0, \notag
\end{align*}
for $\alpha_{\rm vi}>0$.

\subsubsection{Well-posedness and convergence analysis}

We summarise in this section the main results on the well-posedness and convergence properties of EKI with regularisation and variance inflation.\\ 
\begin{theorem}[\protect{\cite[Theorem~3.1]{Schillings2017}},\protect{\cite[Theorem~3.2]{Tong2020}}]
Let $\theta^1(0),...,\theta^J(0)$ be a given
initial ensemble, we denote by $S=\spann\{u^{(j)},j\in\{1,...,J\}\}$ the span of the initial ensemble. Then the ODE systems \eqref{EKI_basic_lin}, \eqref{eki_regul_lin} and \eqref{EKI_VI} have unique global solutions $u^{(j)}(t)\in C^1([0,\infty);S)$ for all $j\in\{1,...,J\}$.
\end{theorem}

The convergence of the EKI estimate to the true parameter is restricted to the span of the initial ensemble $S$. More precisely, the particles stay in the affine space $\theta_0^\perp +\mathcal E$ for all $t \ge 0$, cp. \cite[Corollary 3.8]{Tong2020}, where
\begin{equation}
    \mathcal E=\spann\{e^{(1)}(0), \ldots, e^{(\Nens)}(0)\} \notag
\end{equation}
with 
\begin{equation}
e^{(j)}=\theta^{(j)}-\bar\theta\,, \qquad (j\in\{1,\ldots,\Nens\})\,\notag
\end{equation}
$
\theta_0^\perp=\bar\theta(0)-P_{\mathcal E} \bar\theta(0)$, and
$P_{\mathcal{E}}$ being the projection onto the subspace $\mathcal{E}$.

This implies that the accuracy of EKI is bounded below by the accuracy of the best approximation in $\theta_0^\perp +\mathcal E$. We summarise in the following the main convergence results for the various variants of EKI.

\begin{theorem}[\protect{\cite[Theorem~3.3]{Schillings2017}},\protect{\cite[Theorem~3.13]{Tong2020}}]
Assume that $$\spann\{Ae^{(1)}(0), \ldots, Ae^{(\Nens)}(0)\}=\mathbb R^{\Nobs}.$$ Then, the residuals of EKI mapped under the forward operator converge to $0$. It holds that
    \begin{itemize}
        \item the rate of convergence for EKI without variance inflation is 
        \begin{equation}\|A\theta^{(j)}-y\|_\Gamma^2\in \mathcal O(t^{-1}) \qquad \forall j\in\{1,...,J\}\,,\notag \end{equation}
        \item the rate of convergence for EKI with variance inflation is \begin{equation}\|A\theta^{(j)}-y\|_\Gamma^2\in \mathcal O(e^{-ct}) \qquad \forall j\in\{1,...,J\}, \notag\end{equation}
        for a constant $c>0$.
    \end{itemize}
    
Assume that $\mathcal E=X$. Then, the particles of TEKI converge to the minimiser of the regularised least-squares problem $\theta_{\mbox{\rm reg}}^\dagger$, which is given by $\theta_{\mbox{\rm reg}}^\dagger:=\left(A^T A\right)^{-1}A^Ty$ and respectively $\theta_{\mbox{\rm reg}}^\dagger:=\left(\tilde A^T \tilde A\right)^{-1}\tilde A^T \tilde y$ when considering regularisation. It holds that
    \begin{itemize}
        \item the rate of convergence for TEKI without variance inflation is 
        \begin{equation}\|\theta^{(j)}-\theta_{\mbox{\rm reg}}^\dagger\|_X^2\in \mathcal O(t^{-1}) \qquad \forall j\in\{1,...,J\}\,,\notag\end{equation}
        \item the rate of convergence for TEKI with variance inflation is \begin{equation}\|\theta^{(j)}-\theta_{\mbox{\rm reg}}^\dagger\|_X^2\in \mathcal O(e^{-ct}) \qquad \forall j\in\{1,...,J\}\notag\end{equation}
        for a constant $c>0$.
    \end{itemize}

\end{theorem}

The assumption on the affine space $\mathcal E=X$ is rather restrictive and usually not satisfied in practice. We discuss in the following the generalisation of the convergence result to the more general setting $\theta_0^\perp +\mathcal E\subset X$. 
The best approximation in this space is given by the solution $\theta_{\mathcal E}^{\dagger}$ of the following constrained optimisation problem
\begin{equation}
    \min_{\theta\in\mathcal E} \frac12  \|\tilde A (\theta+\theta_0^\perp) - \tilde y\|^2\,,\notag
\end{equation}
which can be equivalently formulated as the unconstrained optimisation problem
\begin{equation}\label{eqn:constrTEKI}
    \min_{c\in \mathbb R^{\Nens-1}} \frac12  \|\tilde A E c - (\tilde y- \tilde A\theta_0^\perp)\|^2\,,
\end{equation}
with $E$ denoting a basis of $\mathcal E$ (w.l.o.g. $\dim(\mathcal E)=\Nens-1)$. Then, TEKI can be formulated in the coordinate system of $E$ and convergence results can be straightforwardly generalised. Note that convergence then follows to the minimiser of \ref{eqn:constrTEKI}, i.e. the best approximation of \ref{eq:minreg} in the affine space $\theta_0^\perp +\mathcal E$. We refer to \cite{Tong2020} for more details on the derivation of the convergence result.We will denote in the following the optimiser of the constrained optimisation problem by $\theta^\star$, i.e.

\begin{equation}
    \theta^* \in \argmin_{\theta\in\mathcal E} \frac12  \|\tilde A (\theta+\theta_0^\perp) - \tilde y\|^2\,,\notag
\end{equation}

\section{Subsampling in continuous time} \label{sec_subsampling_flows}

In this work, we are interested in certain stochastic approximations of ODEs, indeed, we are studying EKIs in which we randomly replace the potential $\Phi^{\rm reg}$ by one of the $(\Phi_i^{\rm reg})_{i \in I}$. We now introduce and study a framework in which we are able to consider the subsampling of (actually, more general) flows, before then discussing the subsampling in EKI.

\subsection{A general framework and result} \label{subsec_ctstime}
Let $\mathbf{F}_i: X \times [0, \infty) \rightarrow X$ be Lipschitz continuous for $i \in I$. Moreover, we define $\hPsi = \sum_{i \in I}\mathbf{F}_i/{\Nsub}$.
We study the \emph{full dynamical system} $(\theta(t))_{t \geq 0}$, given by
\begin{align} \label{EQ_full_ode}
    \dot{\theta}(t) &=  - \hPsi(\theta(t),t) \qquad (t > 0) \\
    \theta(0) &= \theta_0 \in X. \notag
\end{align}
and also the \emph{subsampled dynamical system}
\begin{align} \label{EQ_subs_ode}
    \dot{\theta}(t) &=  - \mathbf{F}_i(\theta(t),t) \qquad (t > 0) \\
    \theta(0) &= \theta_0 \in X, \notag
\end{align}
 We denote the flows with respect to $(\mathbf{F}_i)_{i \in I}$ by $(\varphi_t^{(i)})_{i \in I, t \geq 0}$: Let $t > 0, i \in I,$ and $\theta_0 \in X$, then $$\dot{\varphi}_t^{(i)}(\theta_0) = - \mathbf{F}_i({\varphi_t^{(i)}}(\theta_0)), \qquad  {\varphi_0^{(i)}}(\theta_0) = \theta_0.$$ In the same way, we denote the flow with respect to $\hPsi$ by $\bar{\varphi}_t$.

Solving or approximating the full dynamical system \eqref{EQ_full_ode} can be computationally expensive, especially if $\Nsub$ is large. We will now discuss an approximation strategy for this full dynamical system that replaces the full system at any point in time by a randomly selected subsampled system. Hence, in any small time interval, we only need to evaluate \eqref{EQ_subs_ode} for some $i \in I$.

We define the subsampled system through a continuous-time Markov process (CTMP) on $I$, which we call $(\bsi(t))_{t \geq 0}$.
Let $\eta: [0, \infty) \rightarrow (0, \infty)$ be  continuously differentiable and bounded from above. We refer to $\eta(t)$ as \emph{learning rate} at time $t \geq 0$. Let $\bsi: [0, \infty) \times \Omega \rightarrow I$ be the CTMP with transition rate matrix
\begin{equation} 
    \label{eq_transition_rate_mat} A(t) := \frac{1}{(\Nsub-1)\eta(t)}\begin{pmatrix} 1 &  \cdots & 1 \\ \vdots & \ddots & \vdots \\ 1 & \cdots & 1 \end{pmatrix} - \frac{\Nsub}{(\Nsub-1)\eta(t)} \cdot \mathrm{id}_I \qquad (t \geq 0)
\end{equation}
and initial distribution $\bsi(0) \sim \mathrm{Unif}(I)$.
$(\bsi(t))_{t \geq 0}$ is the stochastic process characterised by Algorithm~\ref{algo}.

\begin{algorithm} 
\caption{Sampling $(\bsi(t))_{t \geq 0}$}\label{algo}
\begin{algorithmic}[1]
   \State initialise $\bsi(0) \sim \mathrm{Unif}(I)$ and $t_0 \leftarrow  0$
      \State sample $\Delta$ with survival function $$\mathbb{P}(\Delta \geq t|t_0)  :=  \mathbf{1}[t <0 ] + \exp\left( - \int_{0}^t\eta(u + t_0)^{-1}\mathrm{d}u \right) \qquad (t \in [-\infty, \infty])$$
      \State set $\bsi|_{(t_0, t_0+ \Delta)} \leftarrow \bsi(t_0)$ 
      \State sample $\bsi(t_0 + \Delta) \sim \mathrm{Unif}(I \backslash \{\bsi(t_0)\})$
   \State increment $t_0 \leftarrow t_0 + \Delta$ and go to 2
\end{algorithmic}
\end{algorithm}

Hence, the process $(\bsi(t))_{t \geq 0}$ is a piecewise constant process that jumps from one state to another after random waiting times.
There are several other characterisations of the process $(\bsi(t))_{t \geq 0}$, we refer the reader to \cite{Anderson1991} for general CTMPs on discrete spaces. The algorithmic procedure above goes back to Gillespie \cite{Gillespie}. Properties of this particular CTMP have been studied in  \cite{Latz2021}.
We can now define the stochastic approximation process for $(\mathbf{F}_i)_{i \in I}$ and $(\bsi(t))_{t \geq 0}$.

\begin{definition} 
We define the \emph{stochastic approximation process} given by the family of flows $(\mathbf{F}_i)_{i \in I}$ and the index process $(\bsi(t))_{t \geq 0})$ by the tuple $(\bsi(t), \theta(t))_{t \geq 0}$, with
\begin{align*}
    \dot{\theta}(t) &=  - \mathbf{F}_{\bsi(t)}(\theta(t),t) \qquad (t > 0) \\
    \theta(0) &= \theta_0 \in X,
\end{align*}
\end{definition}

In the following, we are interested in the long time behaviour of the stochastic approximation process. 
\begin{assumption} \label{Assum_conv} Let $K \in \mathbb{N}$ and $X := \mathbb{R}^K$. Let (i)-(ii) hold for any $i \in I$:
\begin{itemize}
    \item[(i)] $\mathbf{F}_i \in C^1(X\times [0, \infty),X)$,
    \item[(ii)] the flow $\varphi_t^{(i)}$ contracts quickly -- in particular, we have a measurable function $h: [0, \infty) \rightarrow \mathbb{R}$, with $\int_0^\infty h(t) \mathrm{d}t = \infty$ such that  $$
     \langle \mathbf{F}_{i}(\varphi_t^{(i)}(\theta_0),t)- \mathbf{F}_{i}(\varphi_t^{(i)}(\theta_1),t), \varphi_t^{(i)}(\theta_0) - \varphi_t^{(i)}(\theta_1) \rangle_X \leq  -h(t) \| \varphi_t^{(i)}(\theta_0) - \varphi_t^{(i)}(\theta_1)\|^2
    $$
    for any two initial values $\theta_0, \theta_1 \in X$.
    
\end{itemize}
\end{assumption}
Note that Assumption~\ref{Assum_conv}(ii) already implies that the flow of $-\hPsi$ is exponentially contracting. The Banach fixed-point theorem implies that the flow has a unique stationary point, which we denote by $\theta^* \in X$. We now generalise one of the main results of \cite{Latz2021} by showing that the stochastic process converges to the unique stationary point $\theta^*$ of the flow $(\overline{\varphi}_t)_{t \geq 0}$ if the learning rate goes to zero. Convergence is measured in terms of the Wasserstein distance
$$
\mathrm{d}_{\rm W}(\pi, \pi') = \inf_{H \in C(\pi, \pi')} \int_{X \times X} \min\{1, \|\theta - \theta'\|^q \} \mathrm{d}H(\theta, \theta'),
$$
where $q \in (0,1]$ and $C(\pi, \pi')$ is the set of couplings of the probability measures $\pi, \pi'$ on $(X, \mathcal{B}X)$.

\begin{theorem}\label{thm_gen_Latz21} Let Assumption~\ref{Assum_conv}(i)-(ii) hold for a constant $h$ and a stochastic approximation process $(\bsi(t), \theta(t))_{t \geq 0}$ with initial values $(i_0, \theta_0) \in I \times X$. Moreover, assume that $\lim_{t \rightarrow \infty}\eta(t) = 0$. Then,
$$
\lim_{t \rightarrow \infty}\mathrm{d}_{\rm W}\left(\delta(\cdot - \theta^*), \mathbb{P}(\theta(t) \in \cdot | \theta_0, i_0)\right) = 0
$$
\end{theorem}

\begin{proof}
Please see \ref{thm_gen_Latz21_app}.
\end{proof}
Hence, the process converges to the Dirac measure $\delta(\cdot - \theta^*)$.

\subsection{Ensemble Kalman inversion with subsampling}

We introduce two possibilities of subsampling: EKI with single subsampling and EKI with batch-subsampling. In the first case, we choose a single subsample $y_i^\dagger$ from $y^\dagger$ essentially replace the potential in \eqref{EKI_basic} by the subsampled potential $\Phi_i$. In the mini-batching case, we pick a total of $\Nens$ subsamples $\left(y_{i(1)}^\dagger, \ldots, y_{i(\Nens)}^\dagger\right)$, i.e. one subsample for each particle in the ensemble. Then, we evolve each of the ensemble members with respect to their data subsample, i.e. $\theta^{(j)}(t)$ is evolved with respect to $\Phi_{i(j)}$, for $j = 1,\ldots,\Nens$.

After one remark, we continue by defining the single subsampling.
\begin{remark} When defining the subsampling algorithms, we will only refer to the basic Ensemble Kalman inversion.
Of course, it is possible to combine subsampling with Tikhonov regularisation. In this case, we replace the potential $\Phi_i$ in \eqref{EKI_subsampl} with
\begin{align*}
 \Phi_i'(\theta) &:= \Phi_i(\theta) + \frac{\alpha}{2\Nsub} \|\theta \|^2_X 
\end{align*}
for $i \in I$. In the same way, one can combine subsampling with variance inflation, by replacing the empirical covariance $\widehat{C}_t$ with the inflated covariance $(\widehat{C}_t + \alpha' C_{\rm vi})$.
\end{remark}
\subsubsection*{Single subsampling}
The essential idea is now the following: at every time step, we follow the EKI flow of the potential $\Phi_i$, for one random $i \in I$, as determined by $(\bsi(t))_{t \geq 0}$.
Indeed, the \emph{Ensemble Kalman inversion with single subsampling} is defined via the dynamical system
\begin{align} \label{EKI_subsampl}
    \frac{\mathrm{d} \theta^{(j)}(t)}{\mathrm{d}t} &= - \widehat{C}_t D_\theta \Phi_{\bsi(t)}(\theta^{(j)}(t))\qquad (j \in J) \\ \theta(0) &= \theta_0 \notag.
\end{align}
We illustrate this single subsampling strategy in Figure~\ref{fig:cartoon_single}.

\subsubsection*{Batch-subsampling}
In batch-subsampling, we define a set $\hat{I} \subseteq I^\Nens$ such that for all $i \in \{1,\ldots,\Nsub\}, j \in J$, the number of elements in $\{\hat{i} \in \hat{I}: \hat{i}_j = i \}$ is identical.
Moreover, we define a stochastic process $\bsi:  [0, \infty) \times \Omega \rightarrow \hat{I}$.
Here, the coordinate processes $(\bsi(t;j))_{t \geq 0}$, for $j = 1,\ldots,\Nens$, are stochastically independent CTMPs with transition rate matrix $(A(t))_{t \geq 0}$, as given in \eqref{eq_transition_rate_mat}.
The process $(\bsi(t;j))_{t \geq 0}$ now represents the data subset with which the particle $\theta^{(j)}$ is evolved, for $j = 1,\ldots,\Nens$. Hence, the \emph{Ensemble Kalman inversion with batch-subsampling} is given by
\begin{align*} 
    \frac{\mathrm{d} \theta^{(j)}(t)}{\mathrm{d}t} &= - \widehat{C}_t D_\theta \Phi_{\bsi(t;j)}(\theta^{(j)}(t))\qquad (j \in J) \\ \theta(0) &= \theta_0. \notag
\end{align*}
We illustrate the batch subsampling strategy in Figure~\ref{fig:cartoon_batch}.

\begin{figure}
\centering
\input{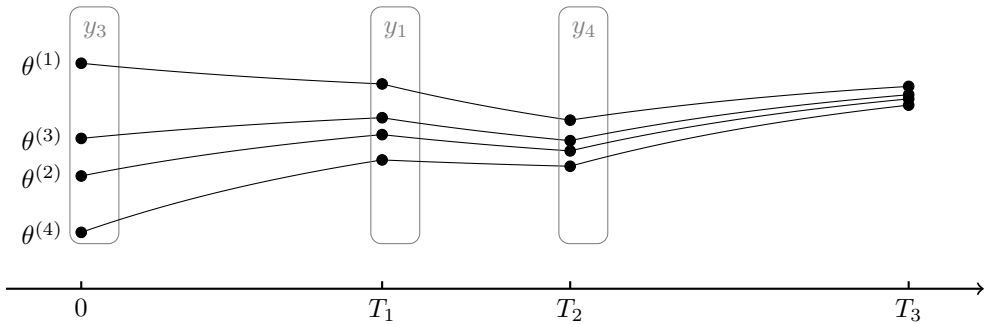}
\caption{Cartoon of EKI with single subsampling: same data subset $y_1,y_2,...$ for each ensemble member.}
\label{fig:cartoon_single}
\end{figure}

\begin{figure}
\centering
\input{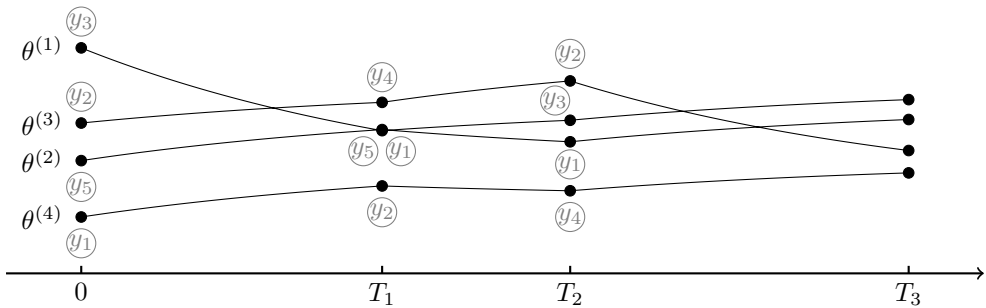}
\caption{Cartoon of EKI with batch subsampling: different data subsets $y_1,y_2,...$ for each ensemble member.}
\label{fig:cartoon_batch}
\end{figure}

\section{Analysis of the ensemble Kalman inversion with subsampling in the linear setting} \label{sec_subsEKI}

We present in the following a convergence theory for the two subsampling strategies in the linear setting. Our goal is to verify the Assumptions \ref{Assum_conv}, in particular the condition
\begin{align*}
-\langle \theta_1 -  \theta_2, \mathbf{F}_i(\theta_1,t) - \mathbf{F}_i(\theta_2,t) \rangle \leq -h(t) \| \theta_1 - \theta_2\|^2,
\end{align*}
for $t$ large enough, with $- \mathbf{F}_i(\theta(t),t)$ denoting the right hand side of the dynamical system and $h: [0, \infty) \rightarrow \mathbb{R}$ being a measurable function. We will see that the analysis of the ensemble collapse will play a central role for the construction of the function $h$. In order to derive convergence results in the parameter space, we will focus in the following on the regularised setting, i.e. we consider the potential $\Phi^{\mbox{\rm reg}}$.

\subsection{TEKI with variance inflation}

The variance inflation allows to explicitly control the ensemble collapse by controlling the preconditioner in the following way:

\begin{theorem}[Single Subsampling TEKI with variance inflation]\label{thm:ssTEKI}
Let $(\theta^{(j)}(t))_{t \geq 0, j \in J}$ satisfy 
\begin{align*} 
    \frac{\mathrm{d} \theta^{(j)}(t)}{\mathrm{d}t} &= - (\widehat{C}_t + \alpha_{\rm vi} C_{\rm vi})  D_\theta \Phi^{\mbox{\rm reg}}_{\bsi(t)}(\theta^{(j)}(t))\qquad (j \in J) \\ \theta(0) &= \theta_0 \notag.
\end{align*}
with index process $(\bsi(t))_{t \geq 0})$ and 
\begin{align*}
\Phi^{\mbox{\rm reg}}_{\bsi(t)}(\theta^{(j)}(t))=\frac 12 \|\tilde y_{\bsi(t)}-\tilde A_{\bsi(t)}\theta\|^2, \quad \tilde A_{\bsi(t)}=\begin{pmatrix}A_{\bsi(t)} \\ \left(\frac{\alpha}{\Nsub}C_0\right)^{\frac12}\end{pmatrix}, \quad \tilde y=\begin{pmatrix}y_{\bsi(t)}^\dagger \\ 0\end{pmatrix},    
\end{align*}
with $\alpha, \alpha_{vi}>0, \widehat{C}_t$ denoting the empirical covariance matrix of the particles $\theta^{(j)}(t)$, and $C_0, C_{vi}$ being symmetric positive definite matrices. The tuple $(\bsi(t), \theta(t))_{t \geq 0}$ denotes the single subsampling TEKI solution with variance inflation.
 Then, $$
\lim_{t \rightarrow \infty}\mathrm{d}_{\rm W}\left(\delta(\cdot - \theta^*), \mathbb{P}(\theta^{(j)}(t) \in \cdot | \theta_0, i_0)\right) = 0 \qquad (j \in J).$$
\end{theorem}

\begin{proof}
By Theorem \ref{thm_gen_Latz21}, we know that convergence of $(\theta^{(j)}(t))$, for $j \in J$, follows if Assumptions \ref{Assum_conv} are satisfied.
The continuous differentiability of the right hand side follows straightforwardly from the definition of the TEKI.  
The inequality holds with $$h(t)=\alpha_{vi}\lambda_{\min}(C_{vi})\left(\min_{1,...,\Nsub}\lambda_{\min}(\tilde A_i^T\tilde A_i)\right),$$ where obviously $\int_0^\infty h(t)\mathrm{d}t=\infty$ is fulfilled. The details on the construction of $h$ are given in Lemma \ref{lemma:convex_vi_ss_app}.
\end{proof}

\begin{theorem}[Batch subsampling TEKI with variance inflation]
Let $(\theta^{(j)}(t))_{t \geq 0, j \in J}$ satisfy 
\begin{align*} 
    \frac{\mathrm{d} \theta^{(j)}(t)}{\mathrm{d}t} &= - (\widehat{C}_t + \alpha_{\rm vi} C_{\rm vi}) D_\theta \Phi^{\mbox{\rm reg}}_{\bsi(t;j)}(\theta^{(j)}(t))\qquad (j \in J) \\ \theta(0) &= \theta_0. \notag
\end{align*}
with index process $(\bsi(t;j))_{t \geq 0, j\in{1,\ldots,\Nens}})$ and 
\begin{align*}
\Phi^{\mbox{\rm reg}}_{\bsi(t;j)}(\theta^{(j)}(t))=\frac 12 \|\tilde y_{\bsi(t;j)}-\tilde A_{\bsi(t;j)}\theta\|^2, \quad \tilde A_{\bsi(t;j)}=\begin{pmatrix}A_{\bsi(t;j)} \\ \left(\frac{\alpha}{\Nsub} C_0\right)^{\frac12}\end{pmatrix}, \quad \tilde y_{\bsi(t;j)}=\begin{pmatrix}y_{\bsi(t;j)}^\dagger \\ 0\end{pmatrix},    
\end{align*}
with $\alpha,\alpha_{vi}>0, \widehat{C}_t$ denoting the empirical covariance matrix of the particles $\theta^{(j)}(t)$, and $C_0, C_{vi}$ being symmetric positive definite matrices. The tuple $(\bsi(t;j), \theta(t))_{t \geq 0}$ denotes the batch subsampling TEKI solution with variance inflation.
 Then, $$
\lim_{t \rightarrow \infty}\mathrm{d}_{\rm W}\left(\delta(\cdot - \theta^*), \mathbb{P}(\theta^{(j)}(t) \in \cdot | \theta_0, i_0)\right) = 0,$$
for $j \in J$.
\end{theorem}

\begin{proof}
The proof follows the same lines as the proof of Theorem \ref{thm:ssTEKI}.
\end{proof}

\subsection{TEKI without variance inflation}

We have seen that the control on the smallest eigenvalue of the preconditioners, i.e. the empirical covariances, is crucial in order to prove convergence. We will consider in the following the more general case, where the smallest eigenvalue converges to $0$ with a rate such that $\int_0^\infty h(t)\mathrm{d}t=\infty$ still holds true and  convergence follows by Theorem \ref{thm_gen_Latz21}.

\begin{theorem}[Single Subsampling TEKI without variance inflation]\label{tmh:ssTEKI}
Let $(\theta^{(j)}(t))_{t \geq 0, j \in J}$ satisfy 
\begin{align} \label{EKI_subsampl_reg}
    \frac{\mathrm{d} \theta^{(j)}(t)}{\mathrm{d}t} &= - \widehat{C}_t   D_\theta \Phi^{\mbox{\rm reg}}_{\bsi(t)}(\theta^{(j)}(t))\qquad (j \in J) \\ \theta(0) &= \theta_0 \notag.
\end{align}
with index process $(\bsi(t))_{t \geq 0})$ and 
\begin{align*}
\Phi^{\mbox{\rm reg}}_{\bsi(t)}(\theta^{(j)}(t))=\frac 12 \|\tilde y_{\bsi(t)}-\tilde A_{\bsi(t)}\theta\|^2, \quad \tilde A_{\bsi(t)}=\begin{pmatrix}A_{\bsi(t)} \\ \left(\frac{\alpha}{\Nsub}C_0\right)^{\frac12}\end{pmatrix}, \quad \tilde y=\begin{pmatrix}y_{\bsi(t)}^\dagger \\ 0\end{pmatrix},    
\end{align*}
with $\alpha>0, \widehat{C}_t$ denoting the empirical covariance matrix of the particles $\theta^{(j)}(t)$, and $C_0$ being a symmetric positive definite matrix. The tuple $(\bsi(t), \theta(t))_{t \geq 0}$ denotes the single subsampling TEKI solution without variance inflation.
 Then, $$
\lim_{t \rightarrow \infty}\mathrm{d}_{\rm W}\left(\delta(\cdot - \theta^*), \mathbb{P}(\theta^{(j)}(t) \in \cdot | \theta_0, i_0)\right) = 0,$$
for $j \in J$.
\end{theorem}

\begin{proof}
The continuous differentiability of the right hand side follows with the same argument as in Theorem \ref{tmh:ssTEKI}.
The inequality holds with $h(t)=\lambda_{\min}(\widehat{C}^1_t)\min_{i\in\{1,...,\Nsub\}}\lambda_{\min}(\tilde A_i^T\tilde A_i)$, where $\int_0^\infty h(t)\mathrm{d}t=\infty$ is fulfilled due to $\lambda_{\min}(\widehat{C}^1_t)\in \mathcal O(t^{-1})$. The details on the minimal eigenvalue are given in the Appendix \ref{lemma:eig_ss_app}.\\
\end{proof}

In the case of batch subsampling, the rate of convergence of each particle is exponential for a fixed data set, as the ensemble collapse is prevented (under suitable assumptions of the data). This is contrast to the single subsampling case, where the rate of convergence is algebraic due to the ensemble collapse. This can be shown as follows: 
\begin{lemma}\label{thm:Expconv}
Given the $\Nsub$ subsamples $\left(\tilde y_{(1)}^\dagger, \ldots, \tilde y_{(\Nsub)}^\dagger\right)$, assume that the centered initial ensemble is a generator of the full space $X$, i.e. $\spann{\{e_0^{(j)},j \in J\}}=X$. We further assume that $\sum_{j=1}^\Nens (\theta^\dagger_j-\bar \theta^\dagger)(\theta^\dagger_j-\bar \theta^\dagger)^\top$ has full rank $d$. Then the particles converge to the true solution $\theta^\dagger$ exponentially fast, i.e. $\theta^{(j)}\to \theta^\dagger_i$ with $\theta^\dagger_i$ denoting the minimiser of $\Phi_i(\theta)=\frac12 \|\tilde A_i \theta - \tilde y_i\|^2$.
\end{lemma}
\begin{proof} 
Please see \ref{thm:Expconv_app}.

\end{proof}

The assumption on the ensemble spread being a generating set of the parameter space $X$ is rather restrictive and usually not satisfied in practice. Generalisation of the result is straightforward when working in the coordinate system of the linear subspace spanned by the initial ensemble using the projection $P_{\mathcal{E}}$.

The batch subsampling case leads to exponential convergence rates for fixed data, however, as the minimum eigenvalue depends on the inital data, the convergence result cannot be readily applied. We modify the process such that we can control the minimum eigenvalue similar to the variance inflation technique above. However, instead of considering a static lower bound, we now allow the minimum eigenvalue to decrease at a certain rate.

\begin{theorem}[Batch subsampling TEKI with diminishing variance inflation]
 Let \\ $(\theta^{(j)}(t))_{t \geq 0, j \in J}$ satisfy 
\begin{align} \label{EKI_subsampl_batch_dimvi}
    \frac{\mathrm{d} \theta^{(j)}(t)}{\mathrm{d}t} &= - (\widehat{C}_t + \frac{\alpha_{\rm vi}}{1+t} C_{\rm vi}) D_\theta \Phi^{\mbox{\rm reg}}_{\bsi(t;j)}(\theta^{(j)}(t))\qquad (j \in J) \\ \theta(0) &= \theta_0. \notag
\end{align}
with index process $(\bsi(t;j))_{t \geq 0, j\in{1,\ldots,\Nens}})$ and
\begin{align*}
\Phi^{\mbox{\rm reg}}_{\bsi(t;j)}(\theta^{(j)}(t))=\frac 12 \|\tilde y_{\bsi(t;j)}-\tilde A_{\bsi(t;j)}\theta\|^2, \quad \tilde A_{\bsi(t;j)}=\begin{pmatrix}A_{\bsi(t;j)} \\ \left(\frac{\alpha}{\Nsub} C_0\right)^{\frac12}\end{pmatrix}, \quad \tilde y=\begin{pmatrix}y_{\bsi(t;j)}^\dagger \\ 0\end{pmatrix},    
\end{align*}
with $\alpha, \alpha_{vi}>0, \widehat{C}_t$ denoting the empirical covariance matrix of the particles $\theta^{(j)}(t)$, and $C_0,C_{\rm vi}$ being symmetric positive definite matrices. The tuple $(\bsi(t;j), \theta(t))_{t \geq 0}$ denotes the bacth subsampling TEKI solution with variance inflation.
 Then, $$
\lim_{t \rightarrow \infty}\mathrm{d}_{\rm W}\left(\delta(\cdot - \theta^*), \mathbb{P}(\theta^{(j)}(t) \in \cdot | \theta_0, i_0)\right) = 0 \qquad (j \in J).$$
\end{theorem}

\begin{proof}
The proof follows the same lines as the proof of Theorem \ref{thm:ssTEKI}. The only difference being that we have here $h(t)=\frac{\alpha_{vi}}{t}\lambda_{\min}(C_{vi})\min_{i \in\{1,...,\Nsub\}}\lambda_{\min}(A_{\bsi}^TA_{\bsi})$, The diminishing rate however is slow enough to obtain $\int_0^\infty h(t)\mathrm{d}t=\infty$. 
\end{proof}

We now move on to proving the convergence of the subsampled EKI. We start by defining an auxiliary EKI subsampling process. Let $\varepsilon \in (0,1)$ and

\begin{align*} 
    \frac{\mathrm{d} \theta^{(j,\varepsilon)}(t)}{\mathrm{d}t} &= - \widehat{C}_t D_\theta \Phi_{{\bsi'}(t,\varepsilon;j)}(\theta^{(j,\varepsilon)}(t))\qquad (j \in J) 
    \\ \theta^{(j,\varepsilon)}(0) &= \theta_0^{(j)}\notag \qquad (j \in J), 
\end{align*}
where
$({\bsi'}(t, \varepsilon))_{t \geq 0}$ is the CTMP on $\hat{I}$ with transition rate matrix $B(t) := A(t \mathbf{1}[t \in [0,1/\varepsilon]] + 1/\varepsilon \mathbf{1}[t \in (1/\varepsilon, \infty)]).$

\begin{proposition} \label{prop:ergodicity_batch}
Let $\varepsilon > 0$. Then, the process $(\theta^{(\cdot,\varepsilon)}(t),\xi^\varepsilon(t), {\bsi'}(t, \xi))_{t \geq 0}$ has a unique stationary measure $\mu_\varepsilon$. Moreover, for every $\theta_0 \in X^{\Nens}$ and $i_0 \in \hat{I}$ there are $c, c' > 0$, with 
$$
\mathrm{d}_{\rm W}\left(\mu_\varepsilon, \mathbb{P}((\theta^{(\cdot,\varepsilon)}(t),\xi^\varepsilon(t), {\bsi'}(t, \xi)) \in \cdot | \theta(0) = \theta_0, \bsi'(t) = i_0) \right) \rightarrow 0 \qquad (t \rightarrow \infty).
$$
\end{proposition}
\begin{proof}
1. We note that for any initial value $\theta_0 \in X^{\Nens}$  there is a compact set $M \subseteq X^{\Nens}, M \ni \theta_0$ from which the process $(\theta^{(j,\varepsilon)}(t), {\bsi'}(t, \varepsilon))_{t \geq 0}$ cannot escape. See \cite{Benaim} for details.  Moreover, we note that under the hypothesis of Theorem~\ref{thm:Expconv}, we can assure that $\theta^{(\cdot,\varepsilon)}(t) \not\in X_{\rm diag} := \{(\theta_1,...,\theta_{\Nens}) \in X^{\Nens} : \exists j, j' \in  J, \theta_j = \theta_{j'}\}$

2. We now show that two coupled processes $(\theta^{(j,\varepsilon)}(t), \bsi'(t, \varepsilon))_{t \geq 0}$ and $(\theta_*^{(j,\varepsilon)}(t), \bsi'(t, \varepsilon))_{t \geq 0}$ starting at different initial points $\theta_0, \theta_{0,*}$ contract in the Wasserstein-2 distance.
Let $(i'(t,\varepsilon; \cdot))_{t \geq 0}$ be a realisation of ${\bsi'}(t, \xi; \cdot))_{t \geq 0}$. Then, again we try to find  a continuous function $h_{i'}: [0, \infty) \rightarrow (0, \infty)$ with
$$
-\langle \theta -  \theta_*, \widehat{C}_t D_\theta \Phi_{{i'}(t,\varepsilon;\cdot)}( \theta) - \widehat{C}_t D_\theta \Phi_{{i'}(t,\varepsilon;\cdot)}( \theta_*) \rangle \leq -h_{i'}(t) \| \theta - \theta_*\|^2.$$
Lemma \ref{lemma:convex_vi_ss_app} gives us 
$h_{i'}(t)=\frac{\alpha_{vi}}{1+t}\lambda_{\min}(C_{vi})\min_{i \in\{1,...,\Nsub\}}\lambda_{\min}(A_{\bsi}^TA_{\bsi}).$\\
Note that we are using variance inflation, however one that is diminishing at rate $t^{-1}$. This is slow enough so that we obtain $\int_0^\infty h_{i'}(t) \mathrm{d}t=\infty$.

\end{proof}

In Proposition \ref{prop:ergodicity_batch}, we showed that the auxiliary process $(\theta^{\varepsilon}(t)))_{t \geq 0}$ is ergodic  ($\varepsilon > 0$) and converges to a stationary measure. In the following, we will show that $(\theta^{\varepsilon}(t)))_{t \geq 0} \rightarrow (\theta(t)))_{t \geq 0}$ as $\varepsilon \rightarrow 0$ and then also that $\theta_t \rightarrow P_Y\theta^\dagger$ as $t \rightarrow \infty$.
\begin{theorem}
Under the assumptions of Proposition~\ref{prop:ergodicity_batch}, we have
$$\mathrm{d}_{\rm W}(\delta(\cdot-P_Y\theta^\dagger), \mathbb{P}(\theta(t) \in \cdot)) \rightarrow 0 \qquad (t \rightarrow \infty).$$
\end{theorem}
\begin{proof}
The result follows from \cite{Latz2021}.
\end{proof}

\section{Numerical Experiments} \label{sec_NumExp}
We now test our methodology in two numerical experiments: first, we aim to estimate  the source term in a 1D parabolic PDE using measurements from its solution. Then, we estimate the log-diffusion coefficient in a 2D elliptic PDE, again using measurements of the solution.

\subsection{1D-Heat equation} \label{sec:firstexample}

In this experiment we consider a one-dimensional Heat equation, given by the following differential equation

\begin{alignat*}{2}
            \frac{\partial u(x,t)}{\partial t} - \frac{\partial^2 u(t,x)}{\partial x^2}  &= f(x) \qquad  &&(t > 0, x \in (0,1))\\
            u(0,x) &= 0   \qquad     &&(x \in (0,1))\\
   u(t,0),u(t,1) &= 0     \qquad     &&(t\geq 0).
\end{alignat*}

Our goal in this inverse problem is to estimate the unknown forcing $f$ from perturbed measurement data that we have obtained from the solution. 
Indeed, we define 
$A=\mathcal{O} \circ L^{-1}$, where $L=\frac{\mathrm{d}}{\mathrm{d} t}-\frac{\mathrm{d}^2}{\mathrm{d}^2 x}$, and $\mathcal{O}:H_0^1([0,1],\mathbb{R})\rightarrow \mathbb{R}^K, (p(\cdot)) \mapsto \mathcal{O}(p(\cdot))=\left(p(x_1),...,p(x_K)\right)^T$ is an equidistant observation operator on $[0,1]\times\mathbb{R}_{\geq 0}$ and $p\in H_0^1([0,1],\mathbb{R})$ is a solution operator of the PDE. The inverse problem is given by:
        $$u=Af+\eta,$$
where $\eta\sim\mathcal{N}(0,\Gamma)$, with $\Gamma=0.1^2 \mathrm{Id}_K$.\\
The forcing is assumed to be a Gaussian random field with zero mean and covariance $C(s,t)=\sigma^2\exp{(-\frac{|s-t|^2}{L_{sc}})}$, where $(s,t)\in[0,1]\times[0,1],\sigma^2=10$ and $L_{sc}=0.1$. We simulate the random field using a KL-expansion, which is truncated after $8$ terms, i.e. $f(x,\omega)=\sum_{i=1}^8 \lambda_i^{1/2}e_i(x)\xi_i(\omega)$, where $\lambda_i$ are the largest eigenvalues of $C(s,t), e_i(x)$ the corresponding eigenfunctions and $\xi_i$ standard normal distributed random variables. Furthermore, we use a spatial step size of $h = 0.01$, a time step size of
$\Delta = 0.05$ and a time horizon of $T=0.3$. Hence, we have $6$ time steps. The PDE-solution is then computed through the Crank-Nicolson method.\\

To solve the inverse problem we will consider the numerical computed solutions of $u(t,x)$ at each time time step as one subsample. Therefore we have $\Nsub=6$ many subsamples. Additionally the choice of the step size $h$, yields $N_{dim(X)}=99$ interior points that we seek to estimate. Furthermore, we have $\Nobs=6 \cdot 99=594$ many observations and use $\Nens=5$ particles. Our initial ensemble is assumed to have the same distribution of the Gaussian random field that we choose for the forcing $f$. We sample from the random field by making $\Nens$ independent draws from $f(x,\omega)$. Therefore, the generated solutions will lie in the subspace $f_0^\perp +\mathcal E$, where $\mathcal E$ is the linear span of the centered initial ensemble and $f_0^\perp=\bar f(0)-P_{\mathcal E} \bar f(0)$.
The best approximation in this space is given by the solution $f_{\mathcal E}^{\dagger}$ of the constrained optimisation problem \eqref{eqn:constrTEKI}, i.e. 
$$    \min_{c\in \mathbb R^{\Nens-1}} \frac12  \|\tilde A E c - (\tilde y- \tilde Af_0^\perp)\|^2,
$$
with $E$ denoting a basis of $\mathcal E$ and $\tilde A$ and $\tilde y$ correspond to the regularised versions of the respective variable. This solution can be computed analytically and is given by
$$c^\dagger_{\mathcal{E}}=\left((\tilde A E)^T(\tilde A E)\right)^{-1}(\tilde A E)^T(\tilde y- \tilde A\theta_0^\perp).$$

Thus the reference solution in the parameter space is given by
$$f_{\mathcal E}^{\dagger}=Ec^\dagger_{\mathcal{E}}+\theta_0^\perp.$$
\\

We simulate $N=32$ many runs and illustrate the mean absolute error of the runs in the parameter space as well as in the observation space for single-subsampling, batch-subsampling and compare it to the EKI. Lastly, we also illustrate the mean ensemble collapse for each particle. The ODE solutions of the EKI as well as our subsampling methods are computed using MATLABs \verb+ode45+ ODE solver.

\subsubsection{TEKI with variance inflation}
The first conducted experiment uses constant variance inflation of the magnitude $\alpha_{\rm vi}=0.01$ and regularisation of $\beta=10$. Furthermore, the learning rate decays at exponential speed, i.e. $\eta(t)=a \exp(-bt)$, with $a=0.01$ and $b=10$. We compute the solution up until $T=1$. With those parameter choices we obtain approx. $2 \cdot 1e5$ many data changes. We illustrate our results in semi-log plots, i.e. we only apply the log functions on to the $y$-axis.

\begin{figure}[h]
\centering
\captionsetup{width=.9\linewidth}
\includegraphics[scale=0.42, trim = 1cm 0cm 0cm 1.1cm,clip]{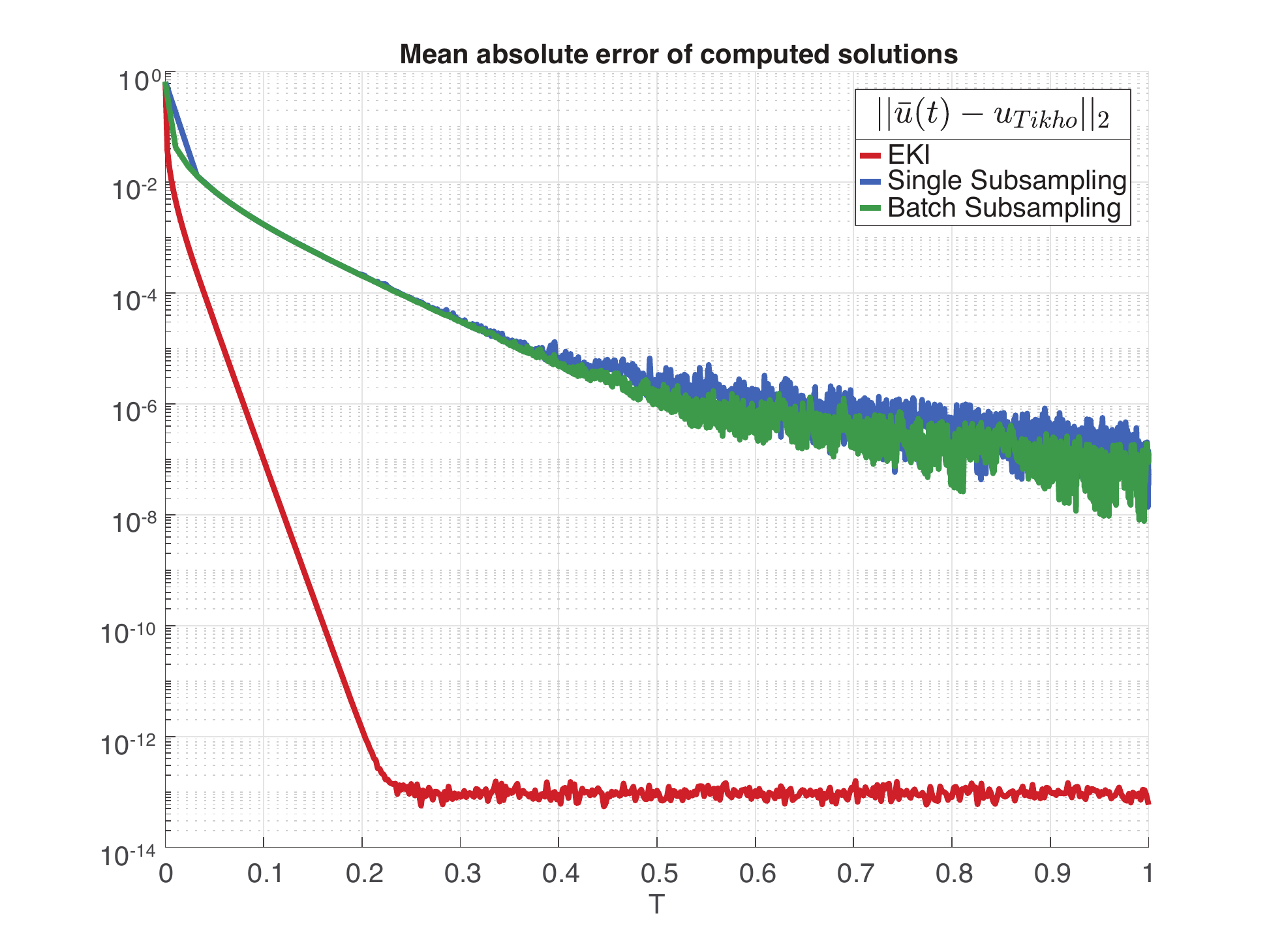}
\hspace*{-0.8cm}%
\includegraphics[scale=0.42, trim = 1.5cm 0cm 0cm 1cm,clip]{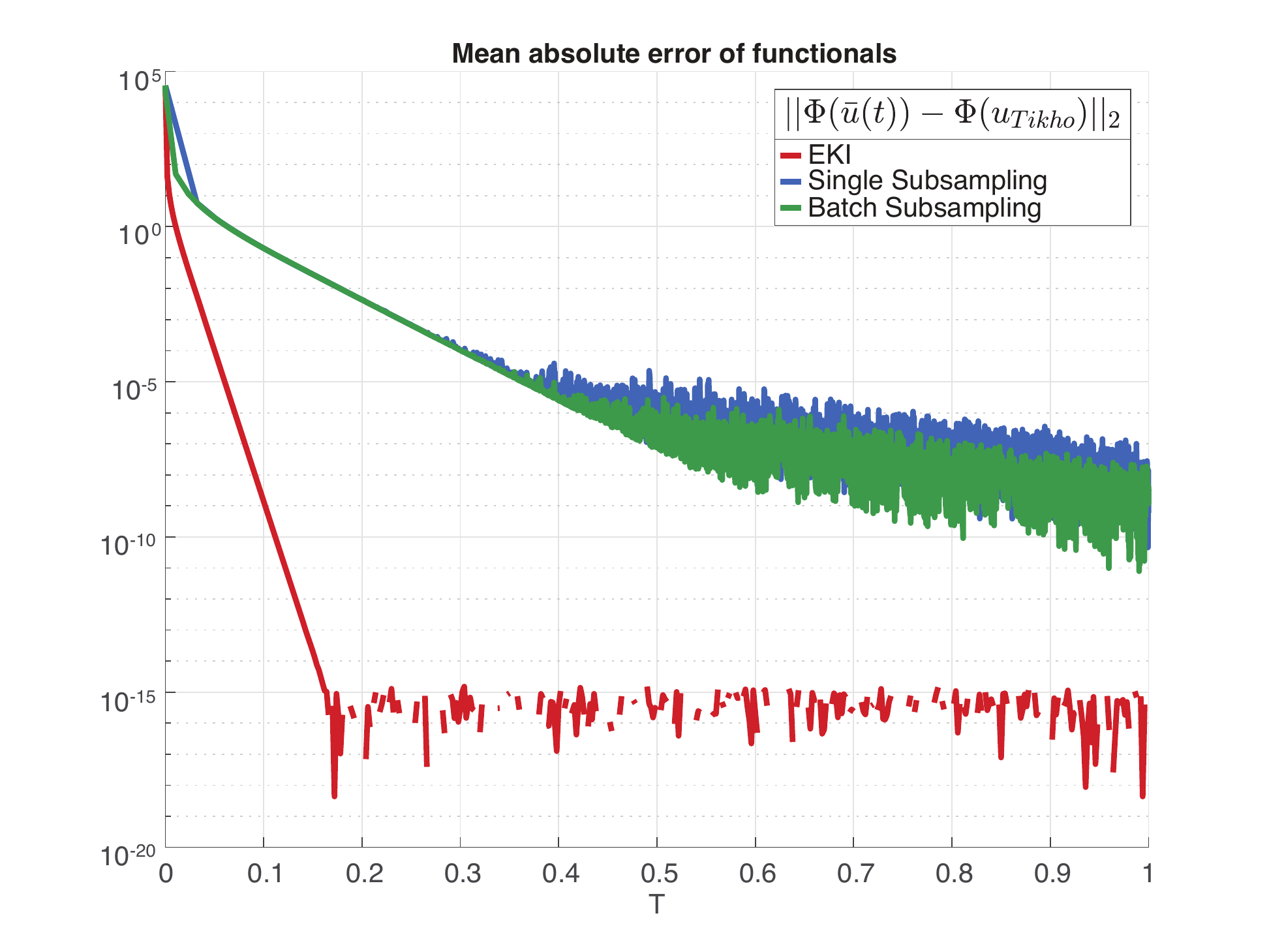}
\caption{Mean absolute errors of computed solutions in the parameter (left) and observation space (right). The red line illustrates the EKI, the blue line single subsampling and the green line batch subsampling.}
\label{fig:mean_err_vi_exp_decay}
\end{figure}

Figure \ref{fig:mean_err_vi_exp_decay} shows the mean relative error over all $N=32$ runs, w.r.t the Tikhonov solution. The left figure shows the mean error in the parameter space and the right figure in the observation space. We can see that both methods converge towards the true solution at exponential rate, due to the linear decay in the semi-log plots. They only differentiate themselves in the constants.

\begin{figure}[hp]
\centering
\captionsetup{width=.9\linewidth}
        \includegraphics[scale=0.29]{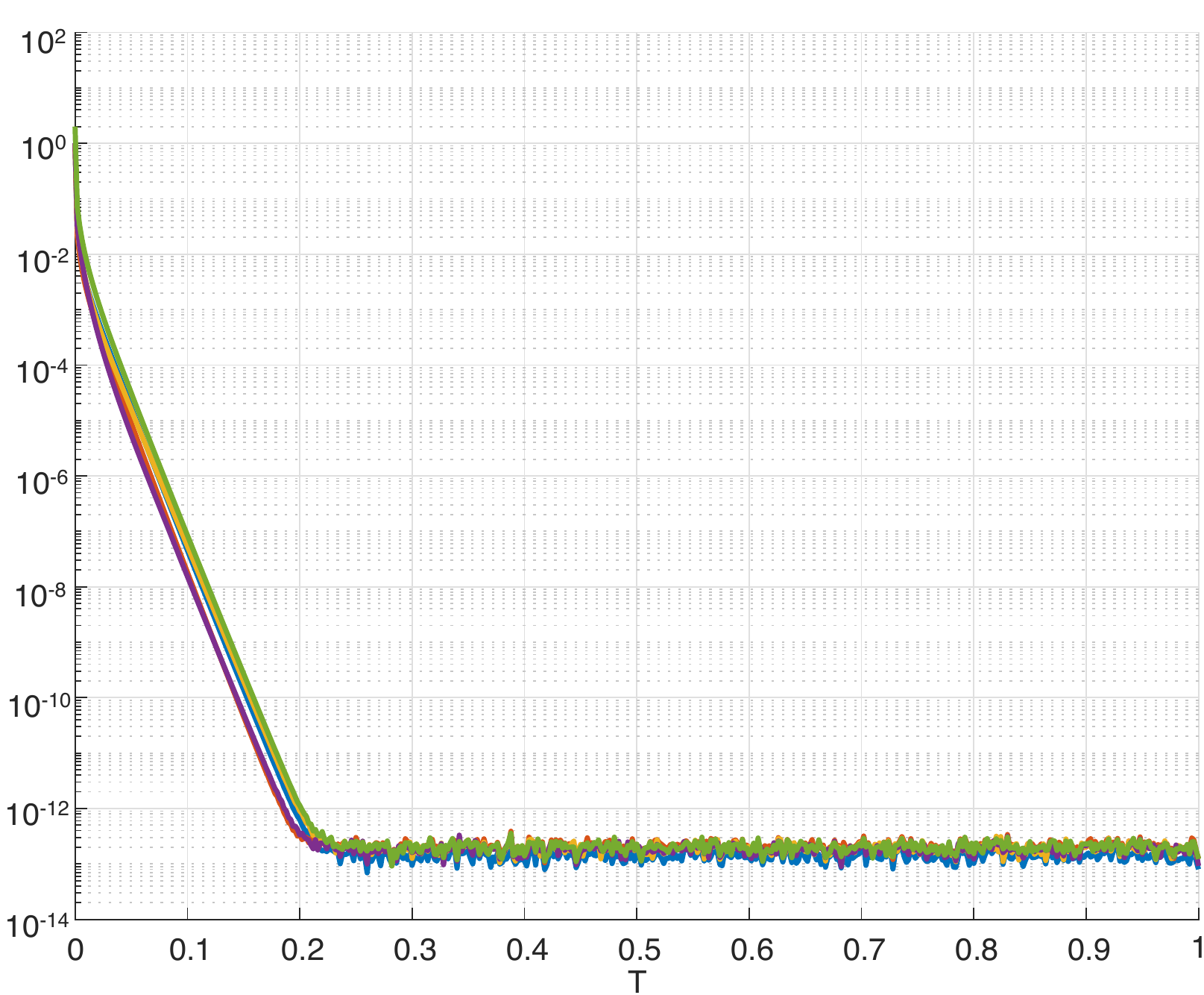}
        \includegraphics[scale=0.29]{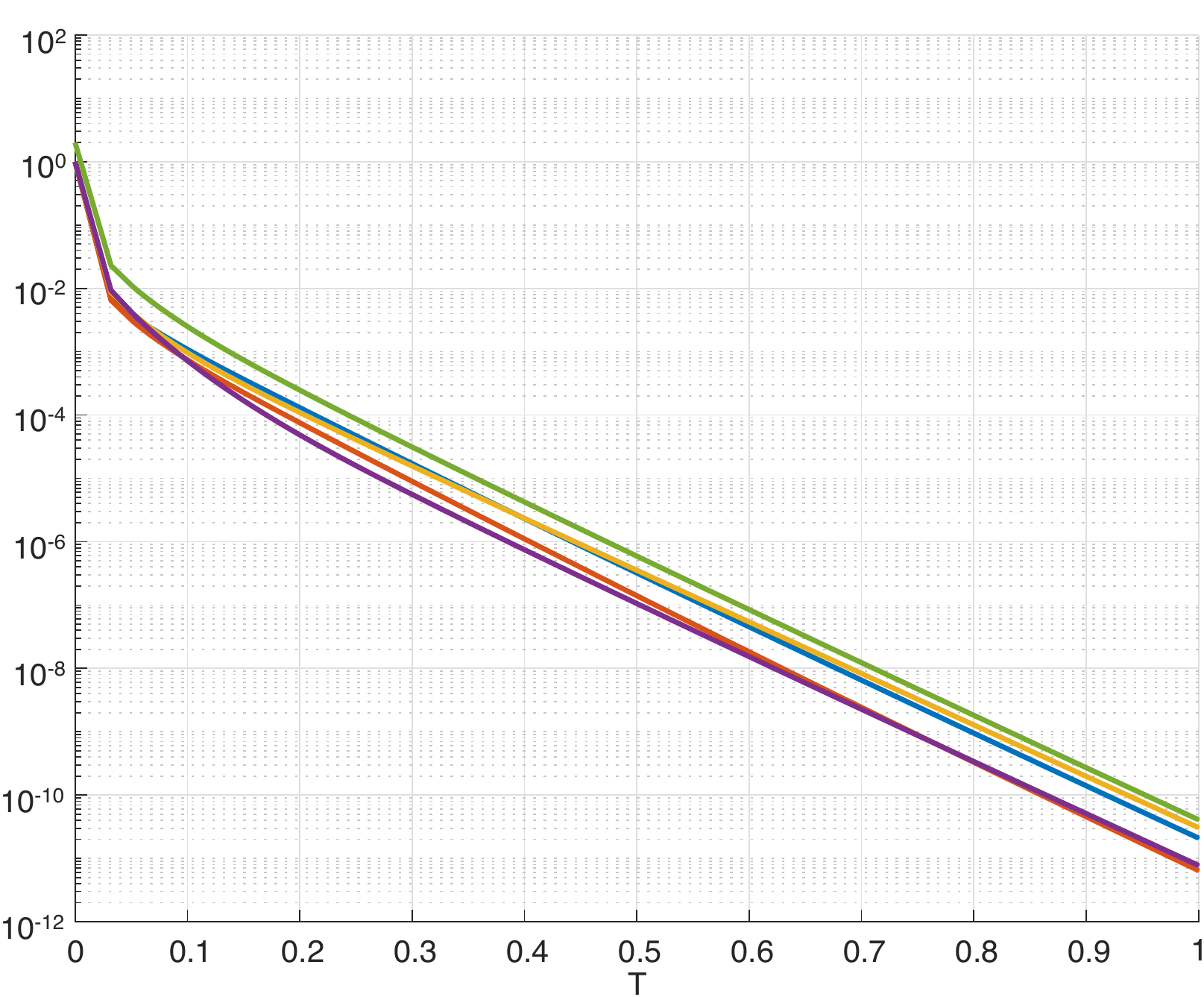}
        \includegraphics[scale=0.29]{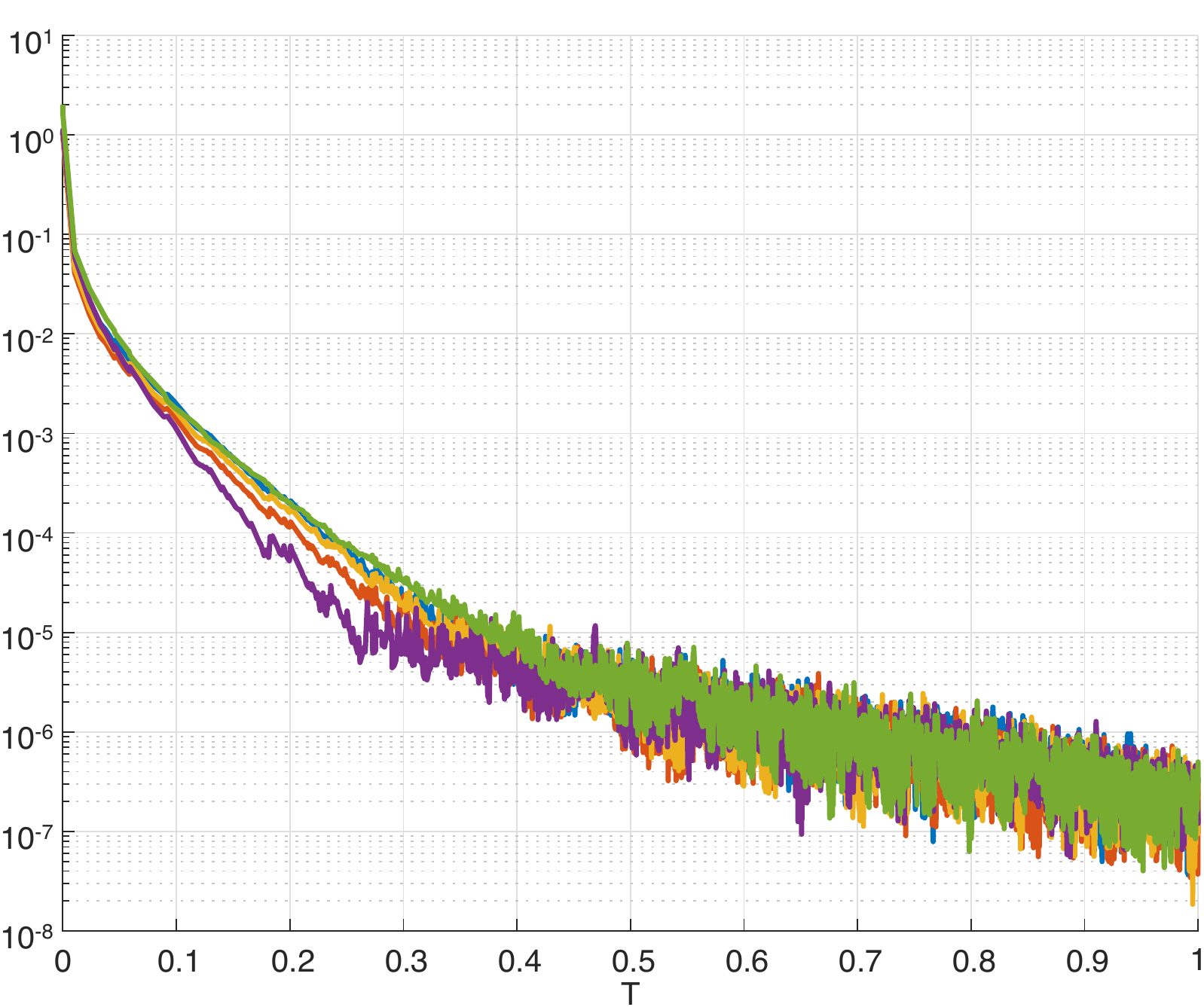}
\caption{Mean ensemble collapse of all $\Nens=5$ particles for the three methods. The different colors represent the ensemble collapse of one ensemble member respectively. Left figure: EKI; middle figure: single subsampling; right: batch subsampling.}.
\label{fig:ek_lin_vi_exp_decay}
\end{figure}

In Figure \ref{fig:ek_lin_vi_exp_decay} the mean ensemble collapse of the particles is illustrated. Again we can see that the collapse happens at an exponential rate. Interesting to note is that there seem to be more fluctuations in batch-subsampling than in single-subsampling. And that batch-subsampling is also a bit slower than single-subsampling(see different scaling on y-axis).

\subsubsection{TEKI with diminishing variance inflation}
In this simulation we consider diminishing variance inflation. We illustrate the results using the same parameters as chosen in the experiment with constant variance inflation. We let variance inflation vanish at a linear rate, i.e. the gradient flow is given by \eqref{EKI_subsampl_batch_dimvi} with $\alpha_{\rm vi}=0.01$. We use again an exponential decaying learning rate and we illustrate the results for EKI, single-subsampling and batch-subsampling.

\begin{figure}[h]
\centering
\captionsetup{width=.9\linewidth}
\includegraphics[scale=0.42, trim = 1cm 0cm 0cm 1.1cm,clip]{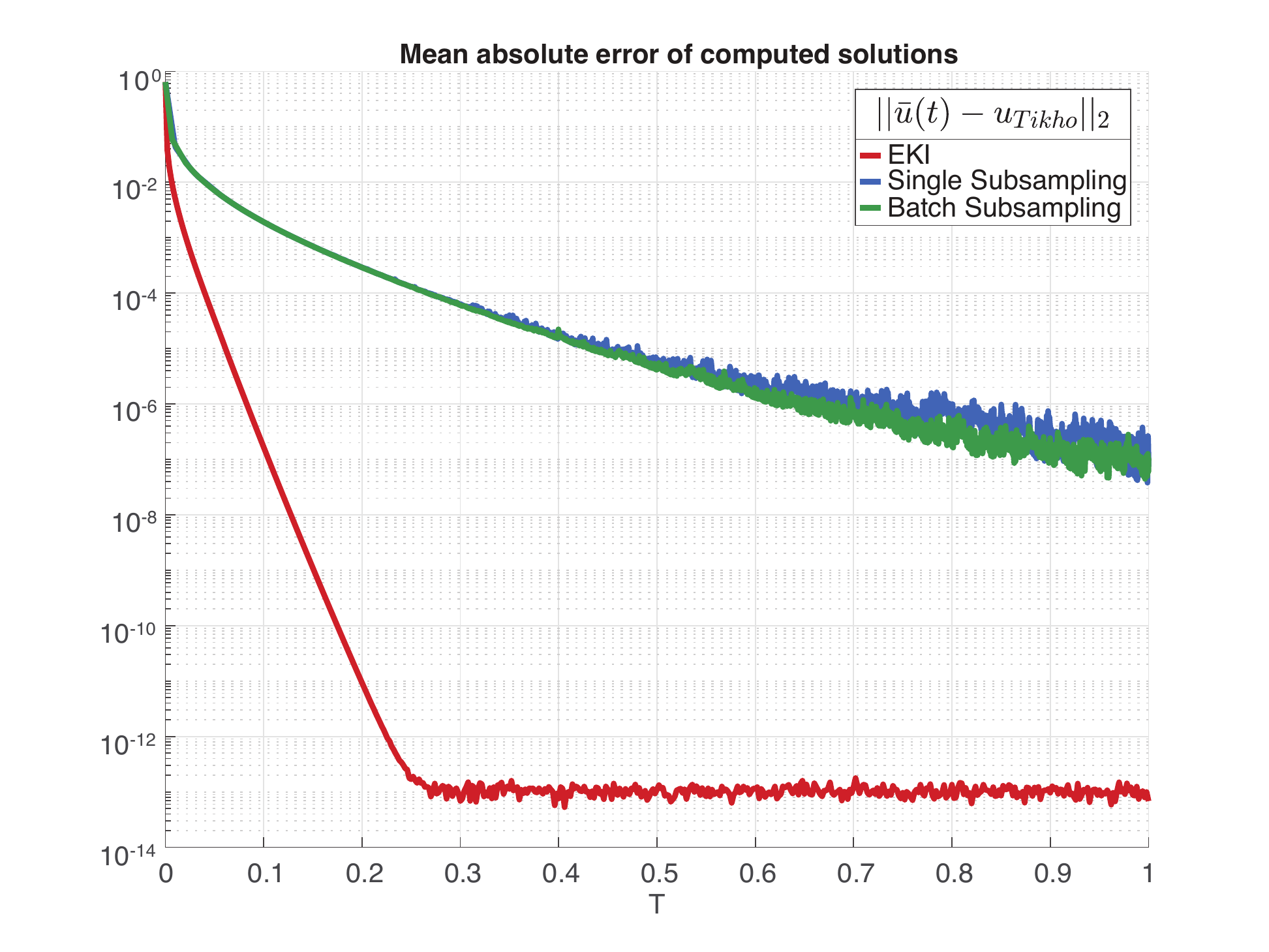}
\hspace*{-0.8cm}%
\includegraphics[scale=0.42, trim = 1.5cm 0cm 0cm 1cm,clip]{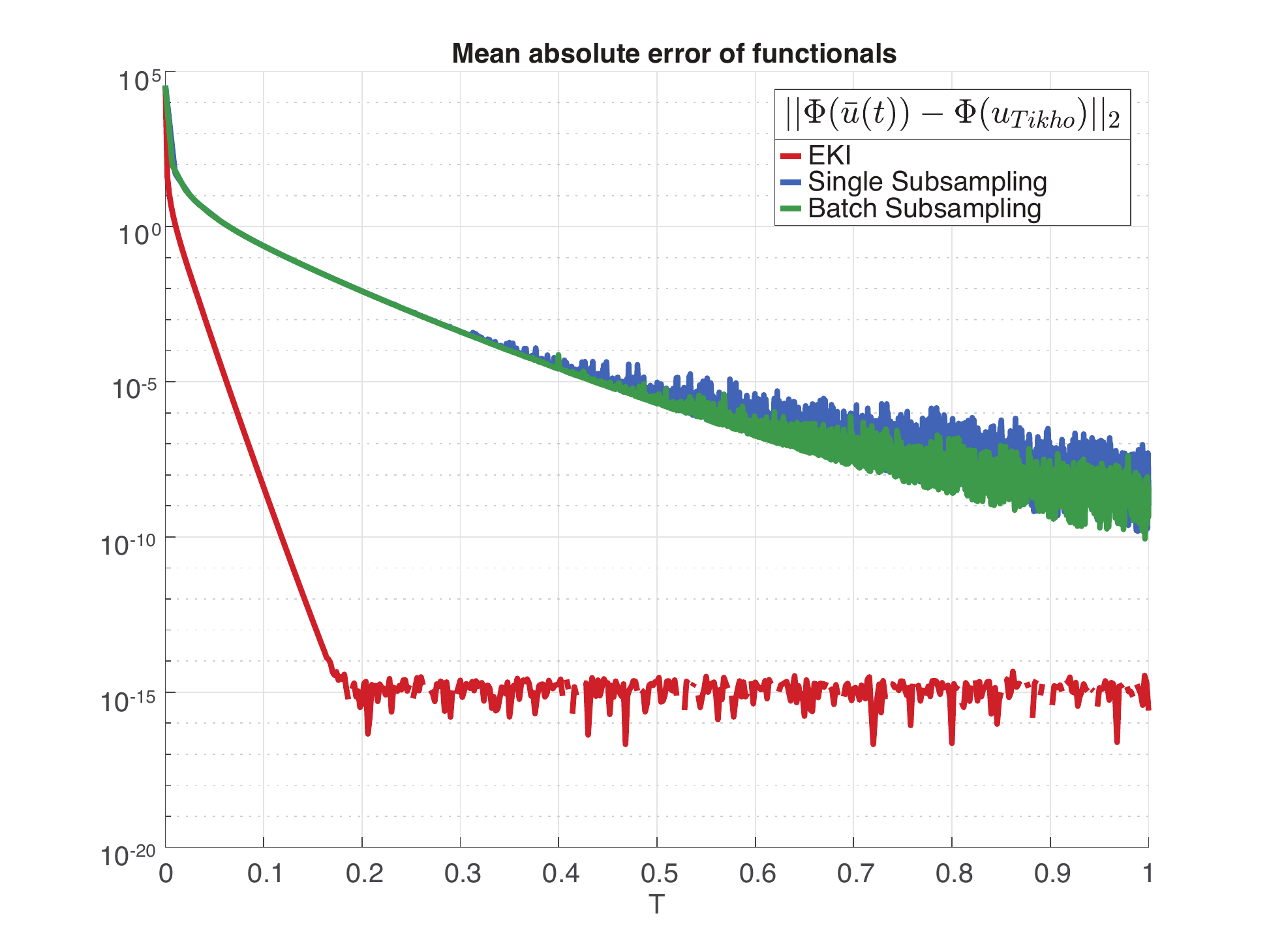}
\caption{Mean absolute errors of computed solutions in the parameter (left) and observation space (right). The red line illustrates the EKI, the blue line single-subsampling and the green line batch-subsampling.}
\label{fig:mean_err_vi_exp_decay_dim}
\end{figure} 

\begin{figure}[H]
\centering
\captionsetup{width=.9\linewidth}
        \includegraphics[scale=0.29]{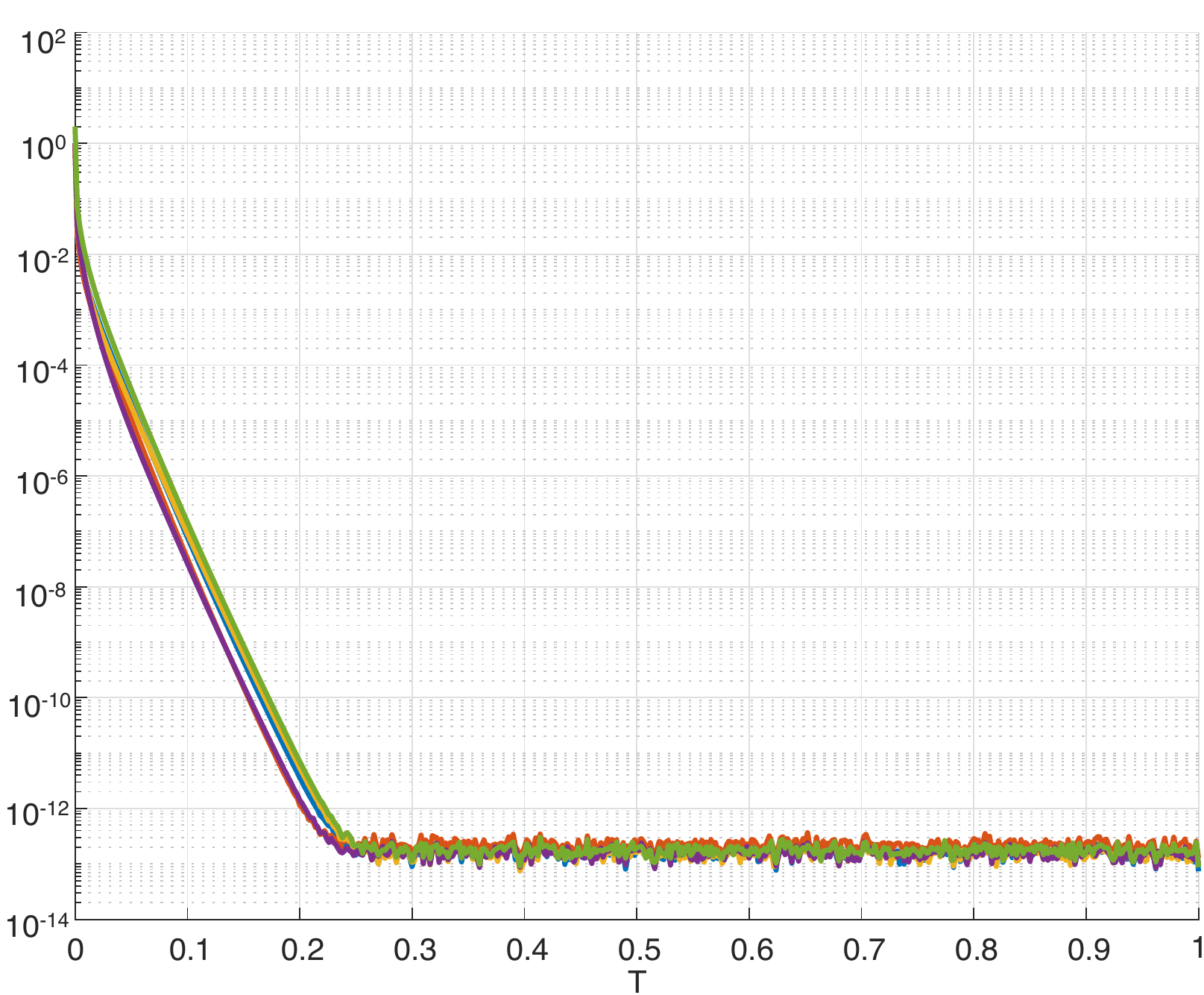}
        \includegraphics[scale=0.29]{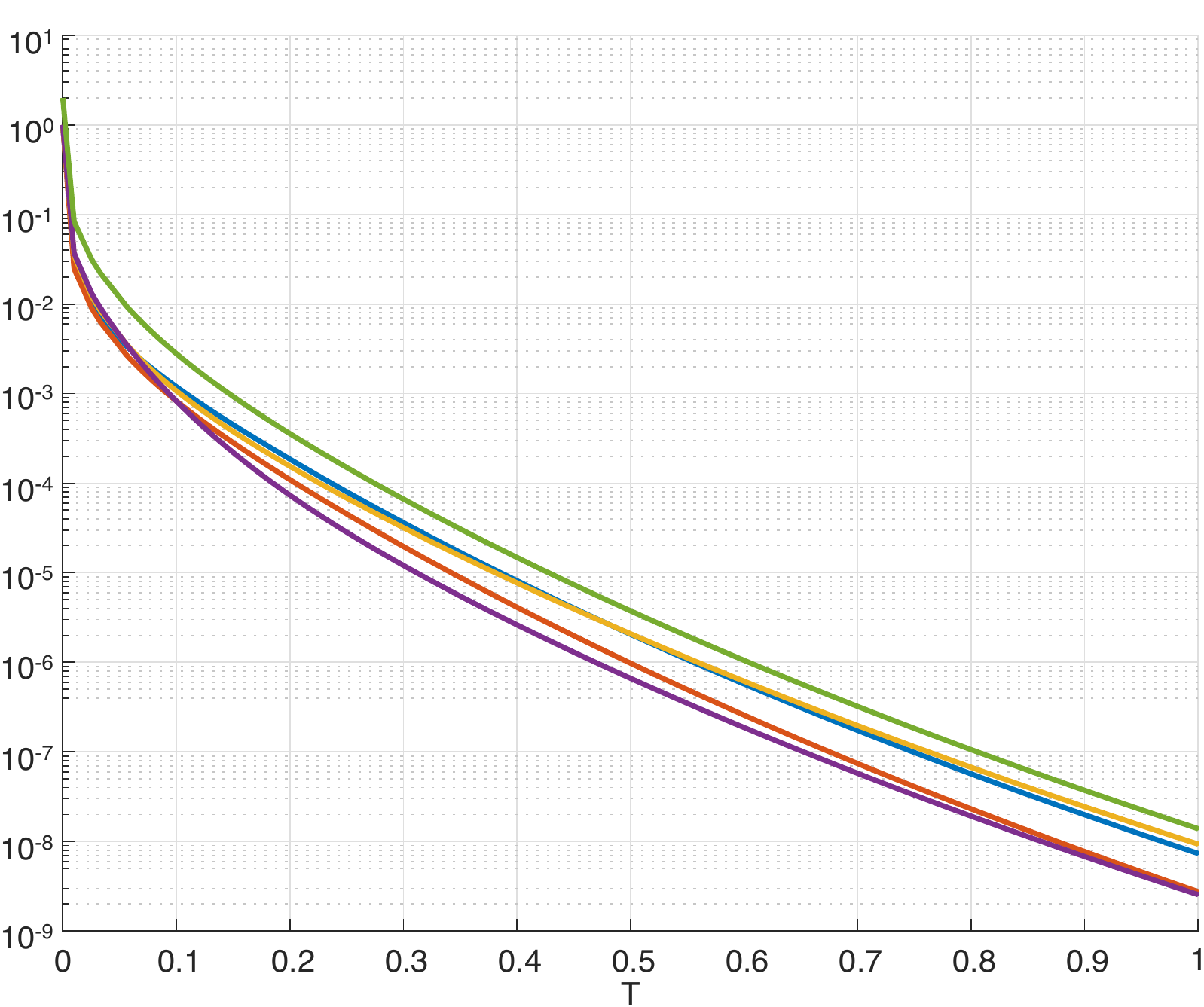}
        \includegraphics[scale=0.29]{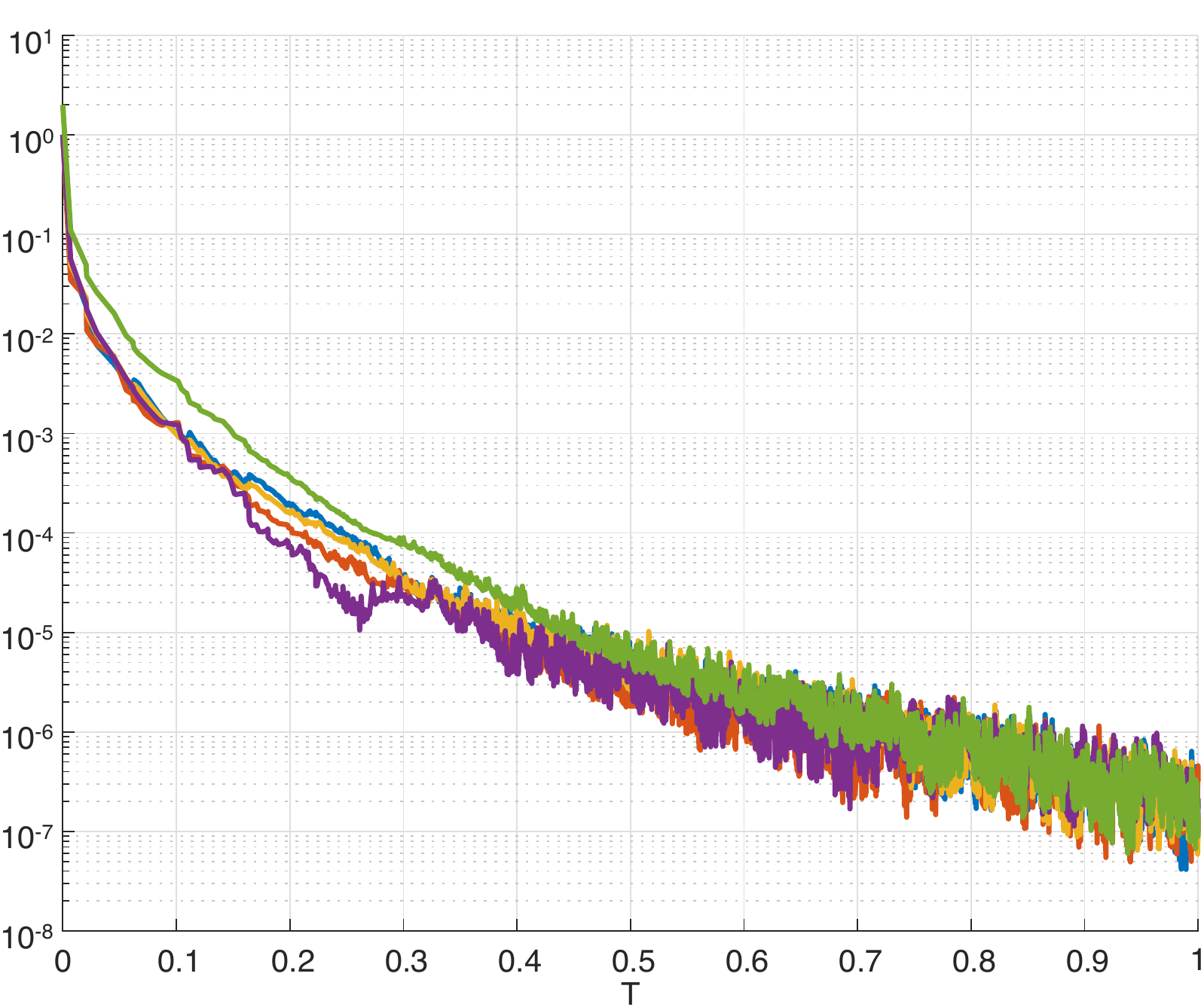}
\caption{Mean ensemble collapse of all $\Nens=5$ particles for the three methods. The different colors represent the ensemble collapse of one ensemble member respectively. Left figure: EKI; middle figure: single-subsampling; right figure: batch-subsampling}.
\label{fig:ek_lin_vi_exp_decay_dim}
\end{figure}

We can see in figure \ref{fig:mean_err_vi_exp_decay_dim} similar results as in \ref{fig:mean_err_vi_exp_decay}. Both subsampling methods converge at a similar rate towards the solution. However, they are both converging at a slower rate than the EKI. Furthermore, this experiment shows less noise in the subsampling approaches as opposed to a constant variance inflation. Figure \ref{fig:ek_lin_vi_exp_decay_dim} shows the ensemble collapse of all methods. We get similar results to the experiment conducted before.

\subsubsection{TEKI without variance inflation}

In this subsection we consider experiments without variance inflation. In theorem \ref{thm:ssTEKI} we proved convergence for single-subsampling. However, we weren't able to obtain a result for batch-subsampling. We illustrate numerical results, indicating that we can also expect convergence for batch-subsampling without variance inflation.\\

We consider again $\beta=10$ as regularisation parameter and compute the solution up until $T=1e6$. This time we use a linear decaying learning rate $\eta(t)=(at+b)^{-1}$, with $a,b=100$. However, due to the decreasing switching times of the subsets, the algorithm becomes computationally slow. We therefore only use the linear decaying learning rate up until $T=10^1$. Afterwards we consider $1e5$ equidistant switching times. Up until $T=1$ we obtain approx $6000$ data switches. Finally, we illustrate our results in log-log plots, since we expect a similar convergence rate as the EKI, which is algebraic. Through Log-log plots, it is easier to observe this convergence rate as opposed to semi-log y plots. Again we conduct $N=32$ experiments and illustrate the mean absolute error.

\begin{figure}[h]
\centering
\captionsetup{width=.9\linewidth}
\includegraphics[scale=0.42, trim = 1cm 0cm 0cm 1.1cm,clip]{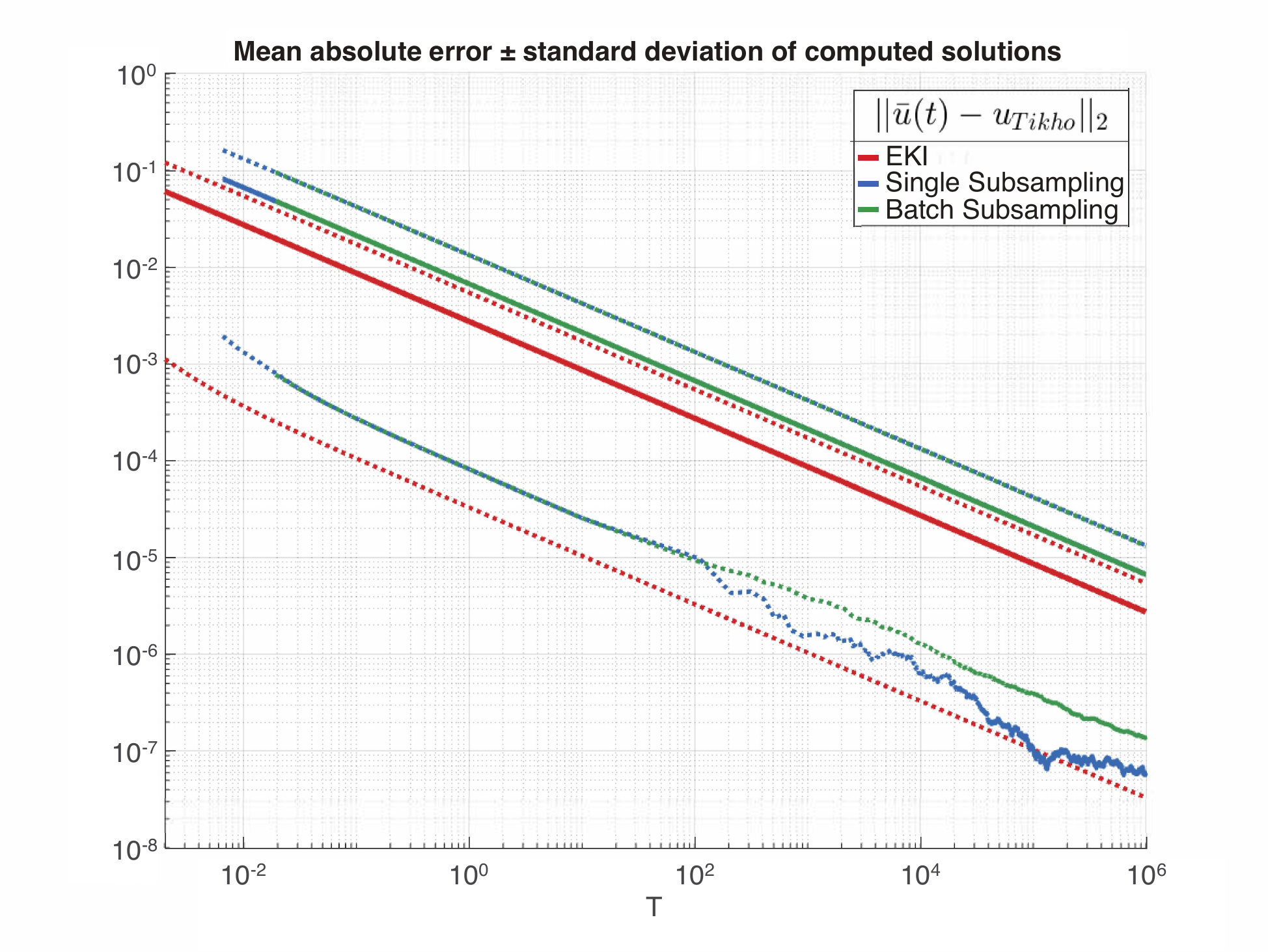}
\hspace*{-0.8cm}%
\includegraphics[scale=0.42, trim = 1.5cm 0cm 0cm 1cm,clip]{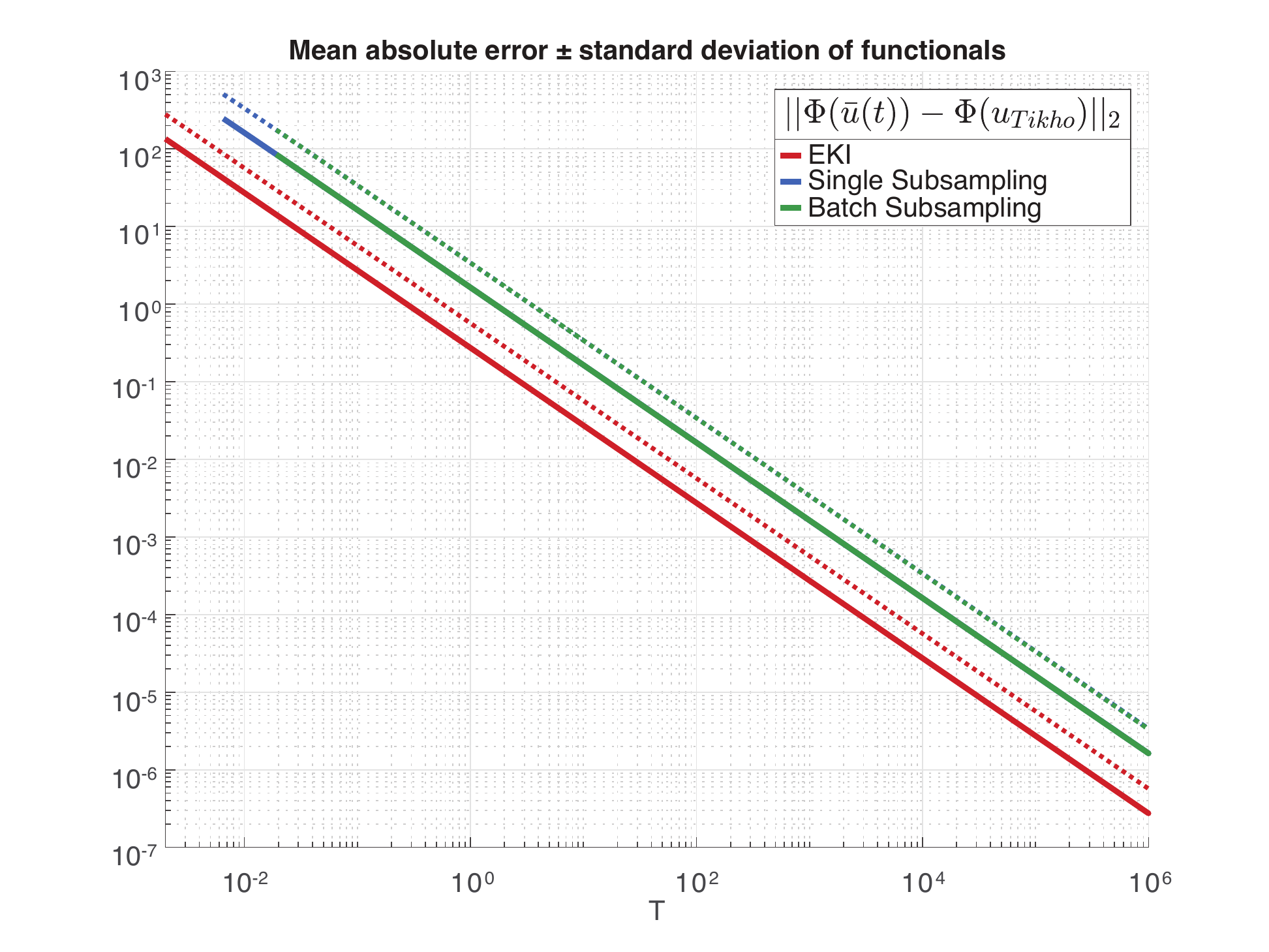}
\caption{Mean absolute errors $\pm$ standard deviation of computed solutions in the parameter (left) and observation space (right). The red line illustrates the EKI, the blue line single-subsampling and the green line batch-subsampling.}
\label{fig:mean_err_nvi_decay_dim}
\end{figure} 

Figure \ref{fig:mean_err_nvi_decay_dim} depicts the mean absolute error in the parameter and observation space to the Tikhonov solution. In those experiments we also included the mean errors $\pm$ one standard deviation. Both subsampling methods show similar results to the EKI. Our methods are therefore suitable alternatives to the EKI, since we obtain the same convergence rate, however we have less computational costs. Note that the mean minus standard deviation is negative in the right picture and therefore not shown.

\begin{figure}[H]
\centering
\captionsetup{width=.9\linewidth}
        \includegraphics[scale=0.29]{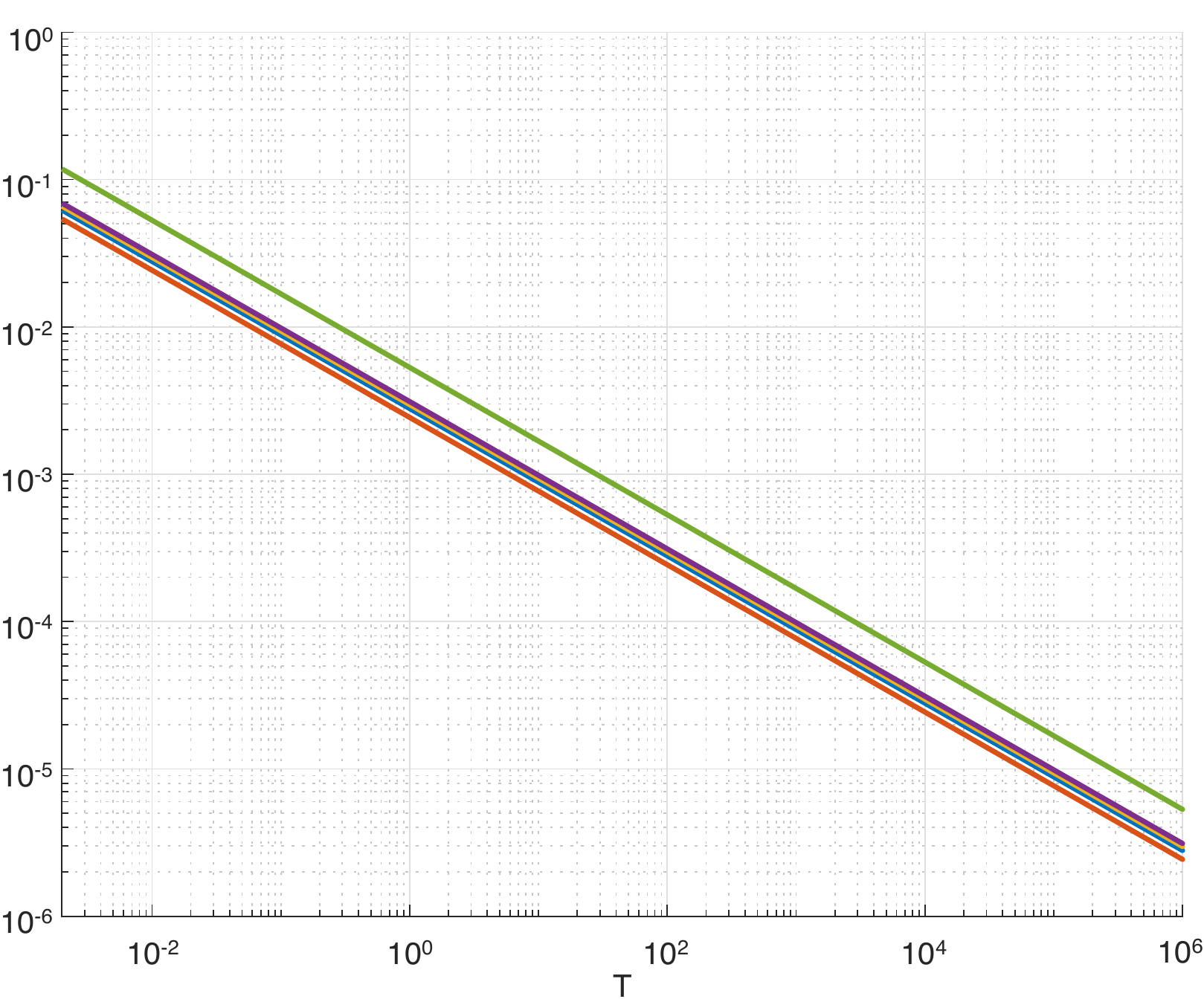}
        \includegraphics[scale=0.29]{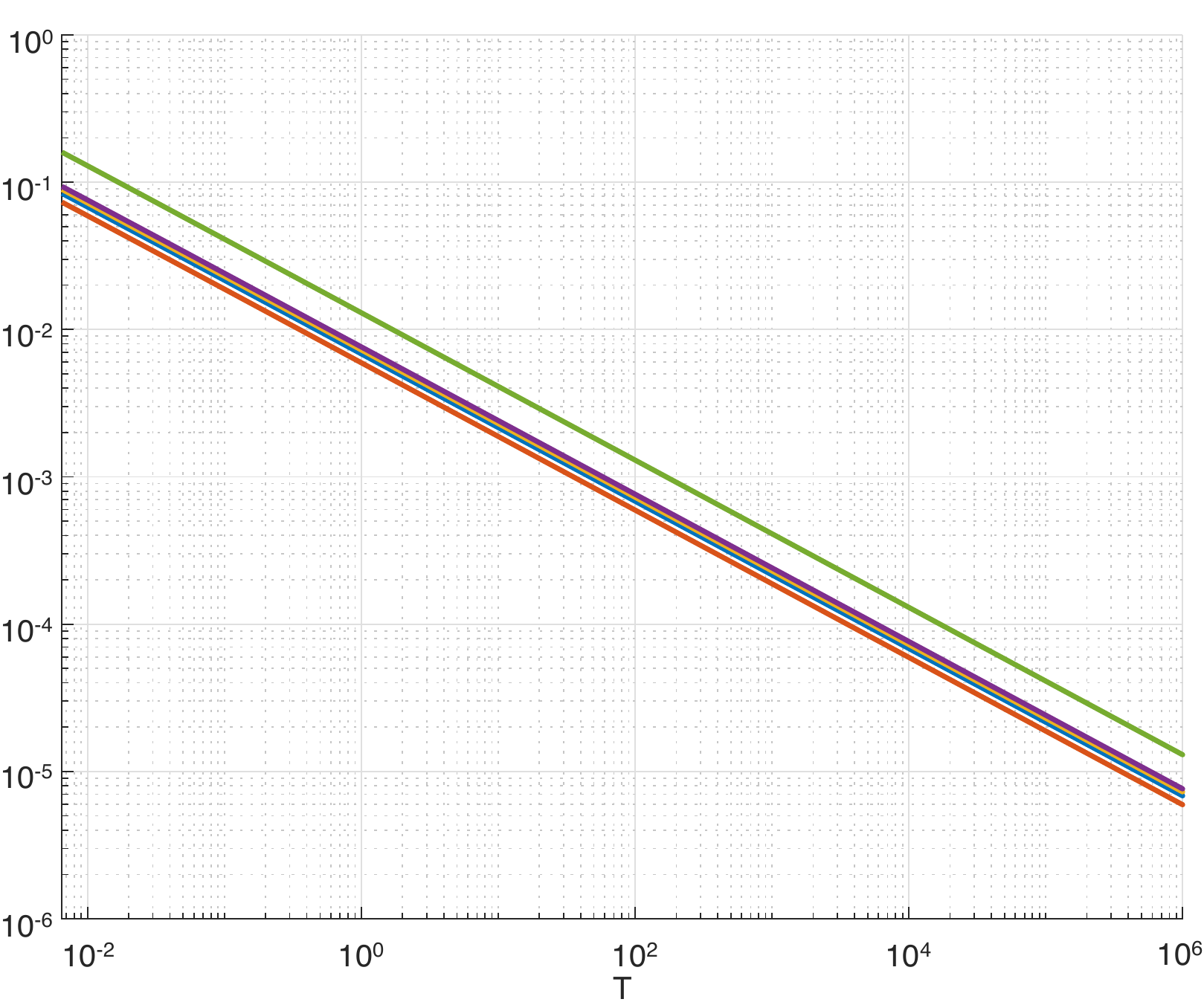}
        \includegraphics[scale=0.29]{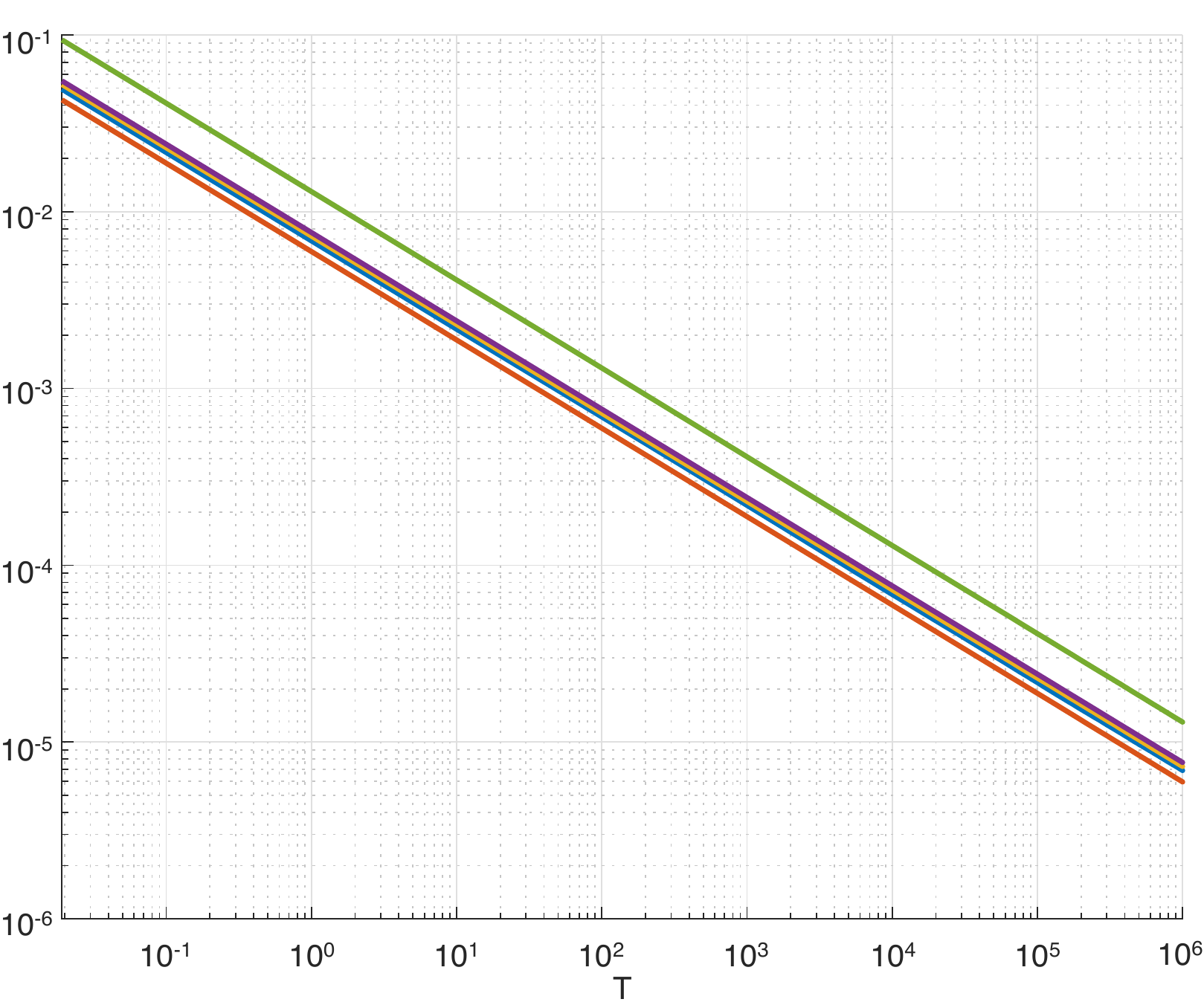}
\caption{Mean ensemble collapse of all $\Nens=5$ particles for the three methods. The different colors represent the ensemble collapse of one ensemble member respectively. Left figure: depicts EKI; middle figure: single-subsampling; right figure: batch-subsampling.}
\label{fig:ek_lin_nvi_decay_dim}
\end{figure}

We can see in figure \ref{fig:ek_lin_nvi_decay_dim} that the ensemble collapse also happens at an algebraic rate for EKI as well as both of our subsampling methods.\\

Furthermore, we illustrate the results of one single run:
We simulate our prior distribution $X$ also by the same KL-expansion, using $18$ terms. We then make $\Nens=20$ independent draws to simulate our initial ensemble. Moreover, we do not use variance inflation and use a regularisation of $\beta=10$, we compute the solution until time $T=1e7$ and use until time $T=10$ the same linear decaying learning rate as in the previous experiment with $a,b=10$. Afterwards we consider again a constant learning rate. We obtain around $600$ switching times until $T=10$ using those parameters.

\begin{figure}[H]
\centering
\captionsetup{width=.9\linewidth}
\includegraphics[scale=0.42, trim = 1cm 0cm 0cm 1.1cm,clip]{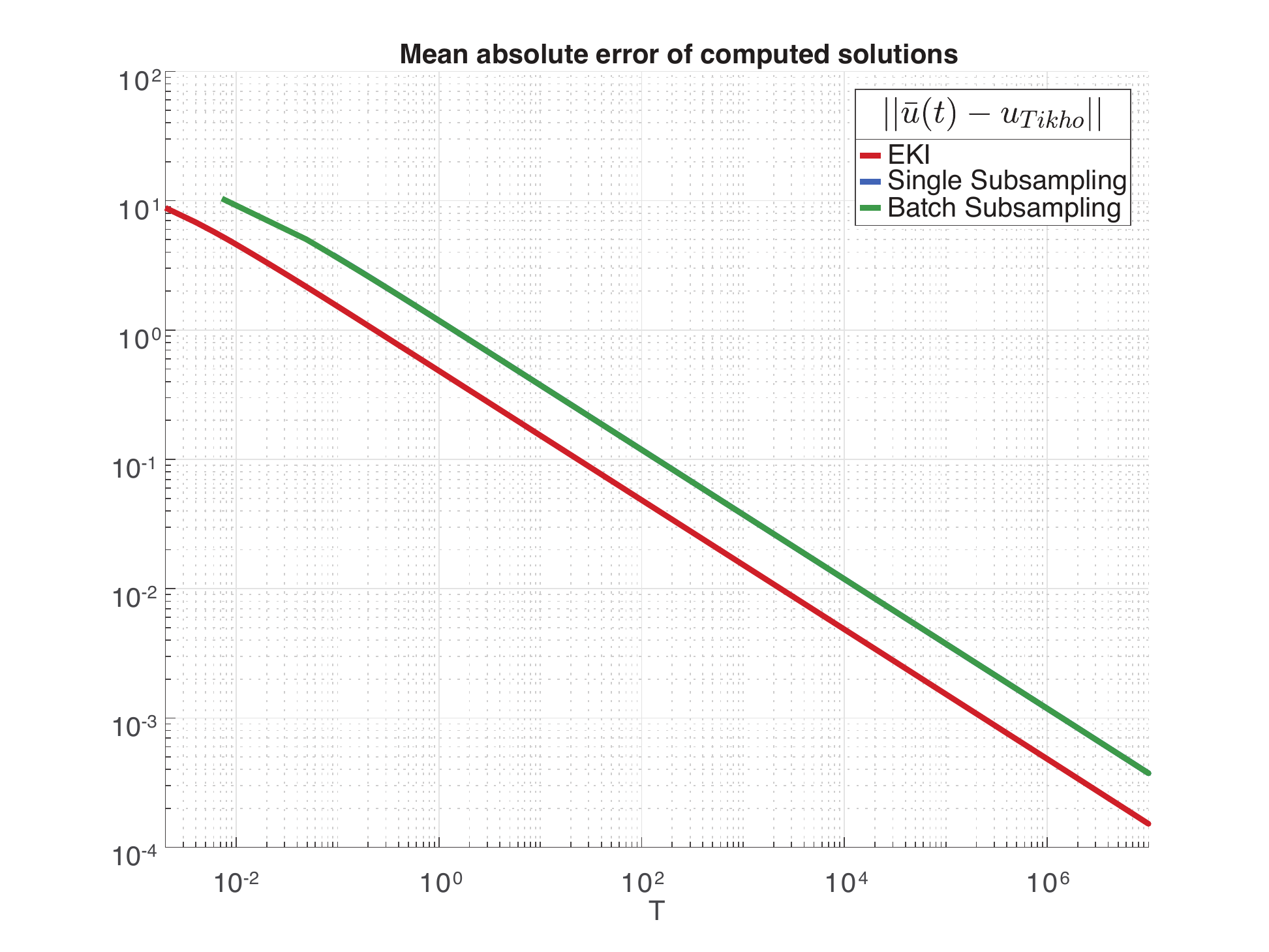}
\hspace*{-0.8cm}%
\includegraphics[scale=0.42, trim = 1.5cm 0cm 0cm 1cm,clip]{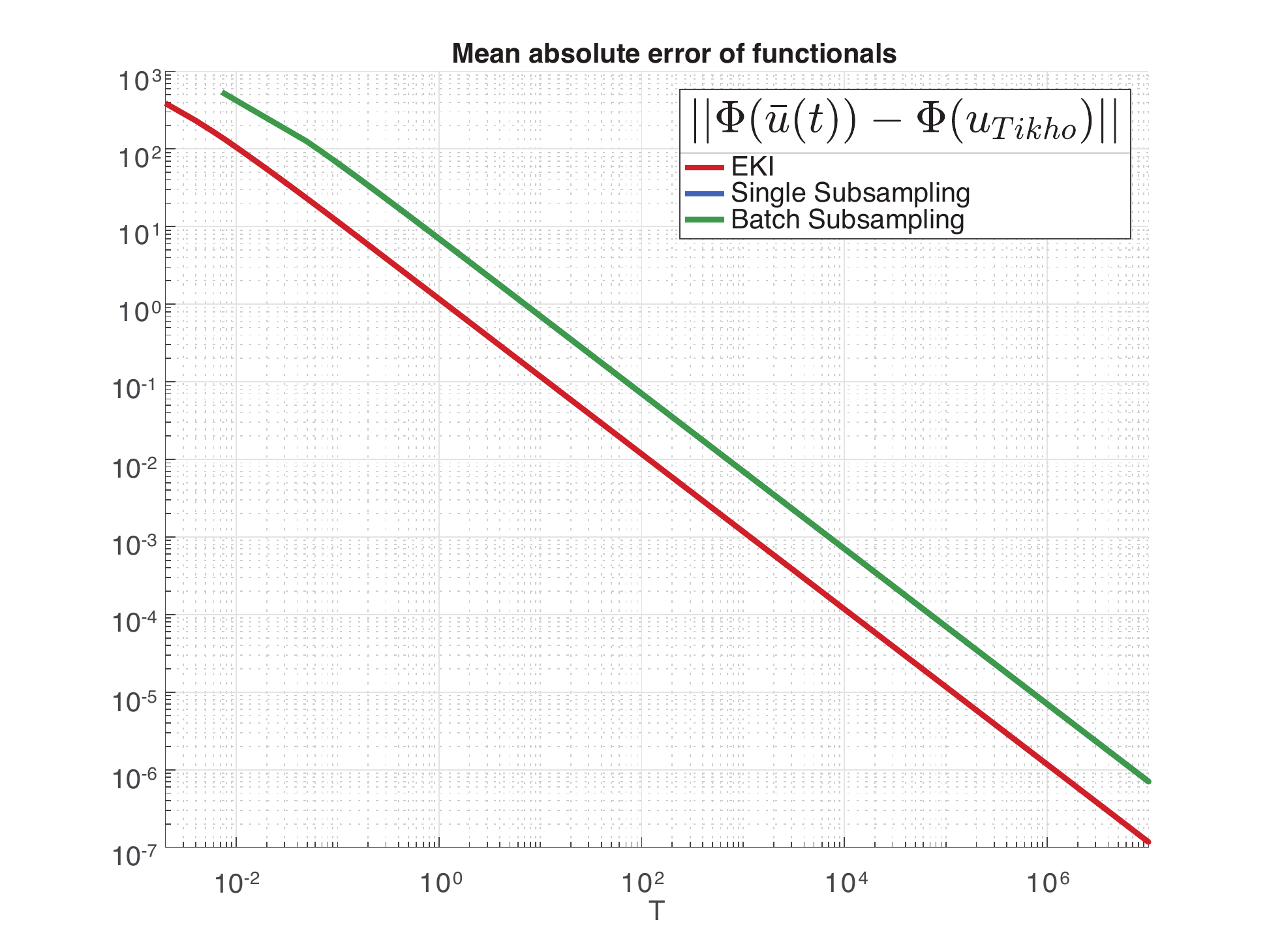}
\caption{Absolute error of computed solutions in the parameter (left) and observation space (right). The red line illustrates the EKI, the blue line single-subsampling and the green line batch-subsampling.}
\label{fig:abs_err_nvi_lin_decay_single_run}
\end{figure} 

In Figure \ref{fig:abs_err_nvi_lin_decay_single_run}, we can see again that the error of all three solutions behaves similarly. Single and batch subsampling compute nearly identical solutions, therefore the error of those two methods are almost identical. The only difference to EKI is the starting time at which the solution is evaluated. The randomly computed starting time of our subsampling approaches is a little higher than the starting point of the EKI. Therefore, there is a slight shift in the error curves.\\
Furthermore, we illustrate how the computed solution behaves over time. We compare the solutions to the Tikhonov solution.

\begin{figure}[H]
\centering
\captionsetup{width=.9\linewidth}
        \includegraphics[scale=0.39, trim = 0.1cm 0cm 1.8cm 0cm,clip]{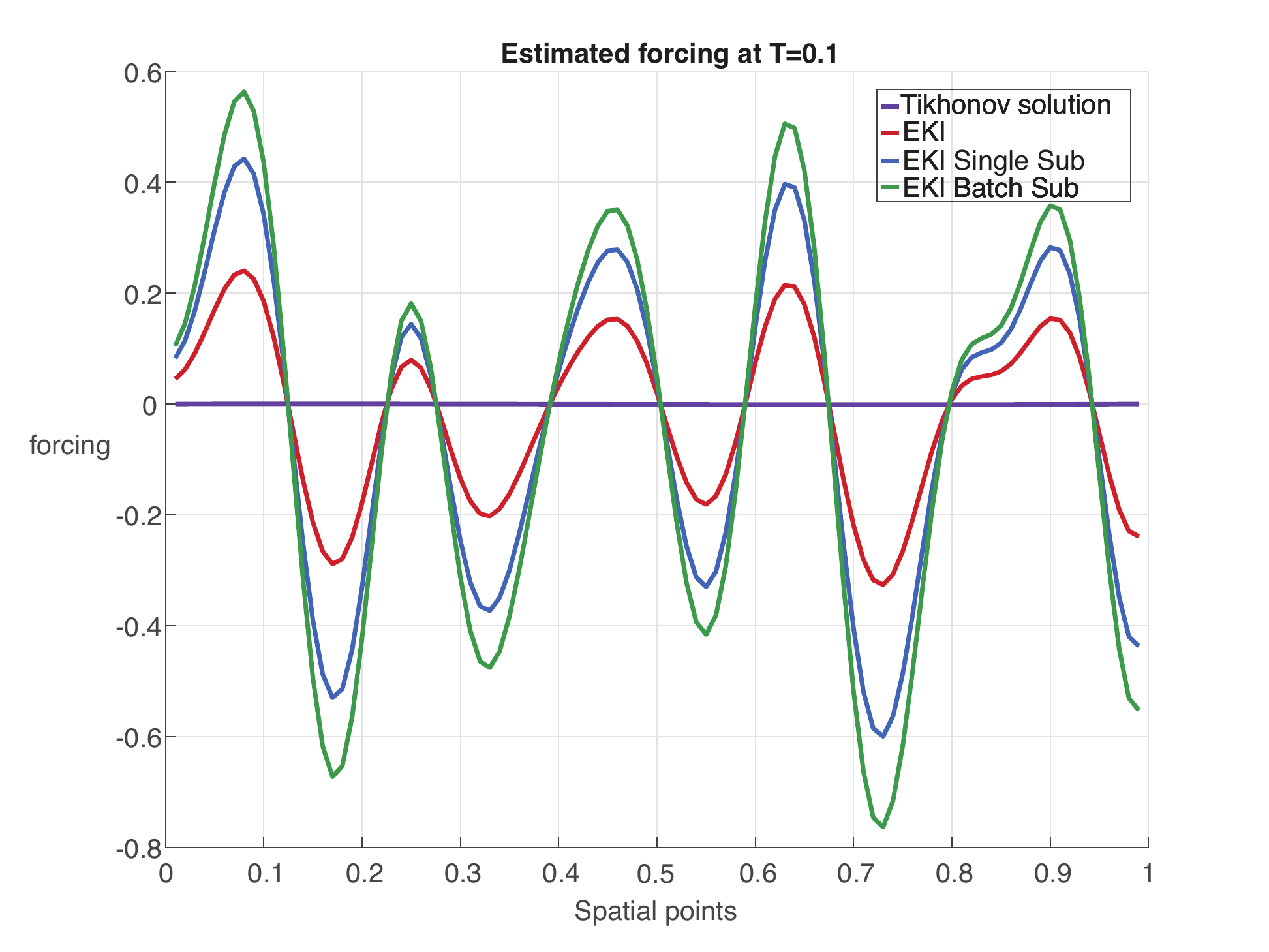}
        \includegraphics[scale=0.39, trim = 1.8cm 0cm 1.8cm 0cm,clip]{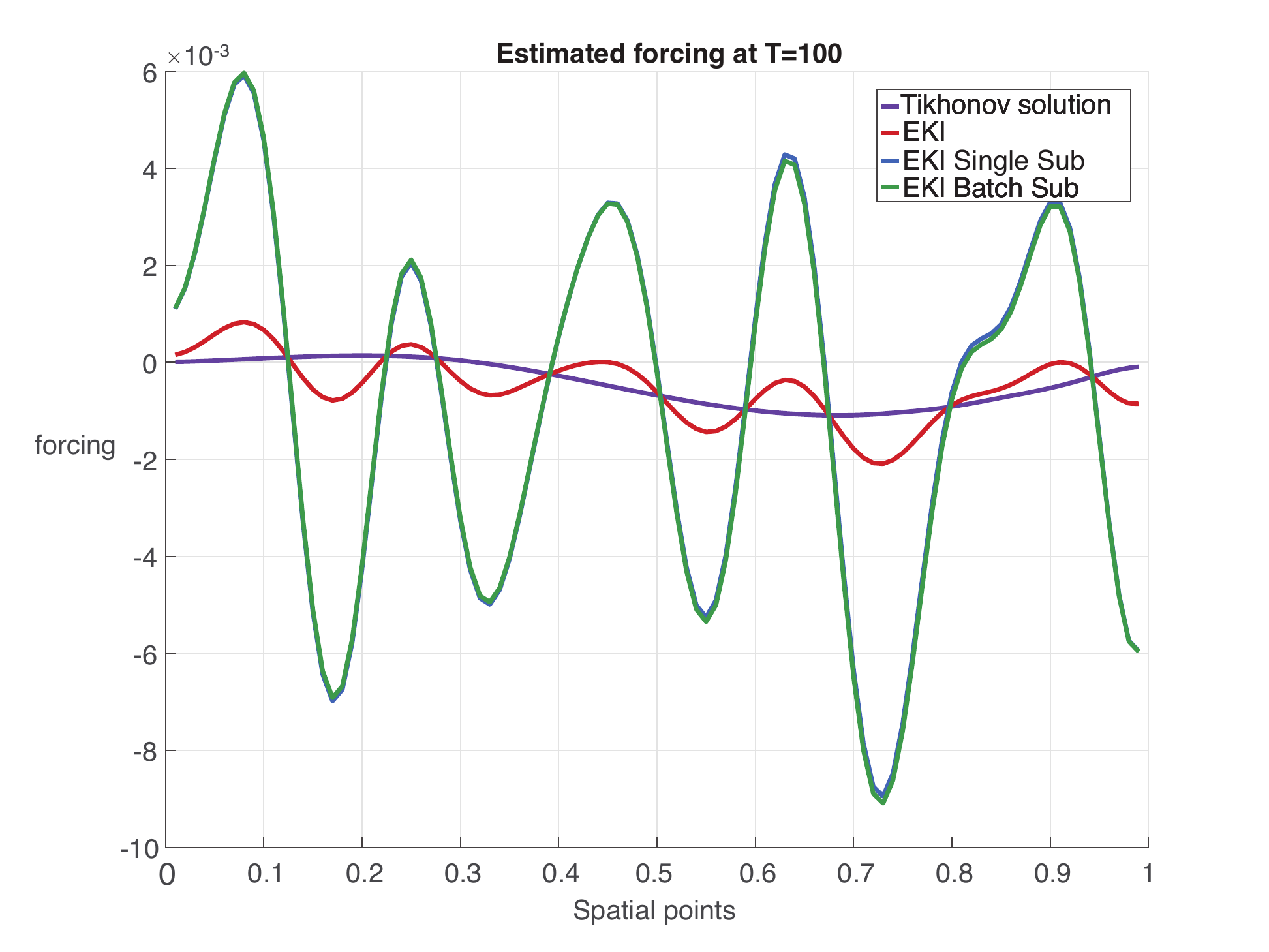}
        \includegraphics[scale=0.39, trim = 0.1cm 0cm 1.8cm 0cm,clip]{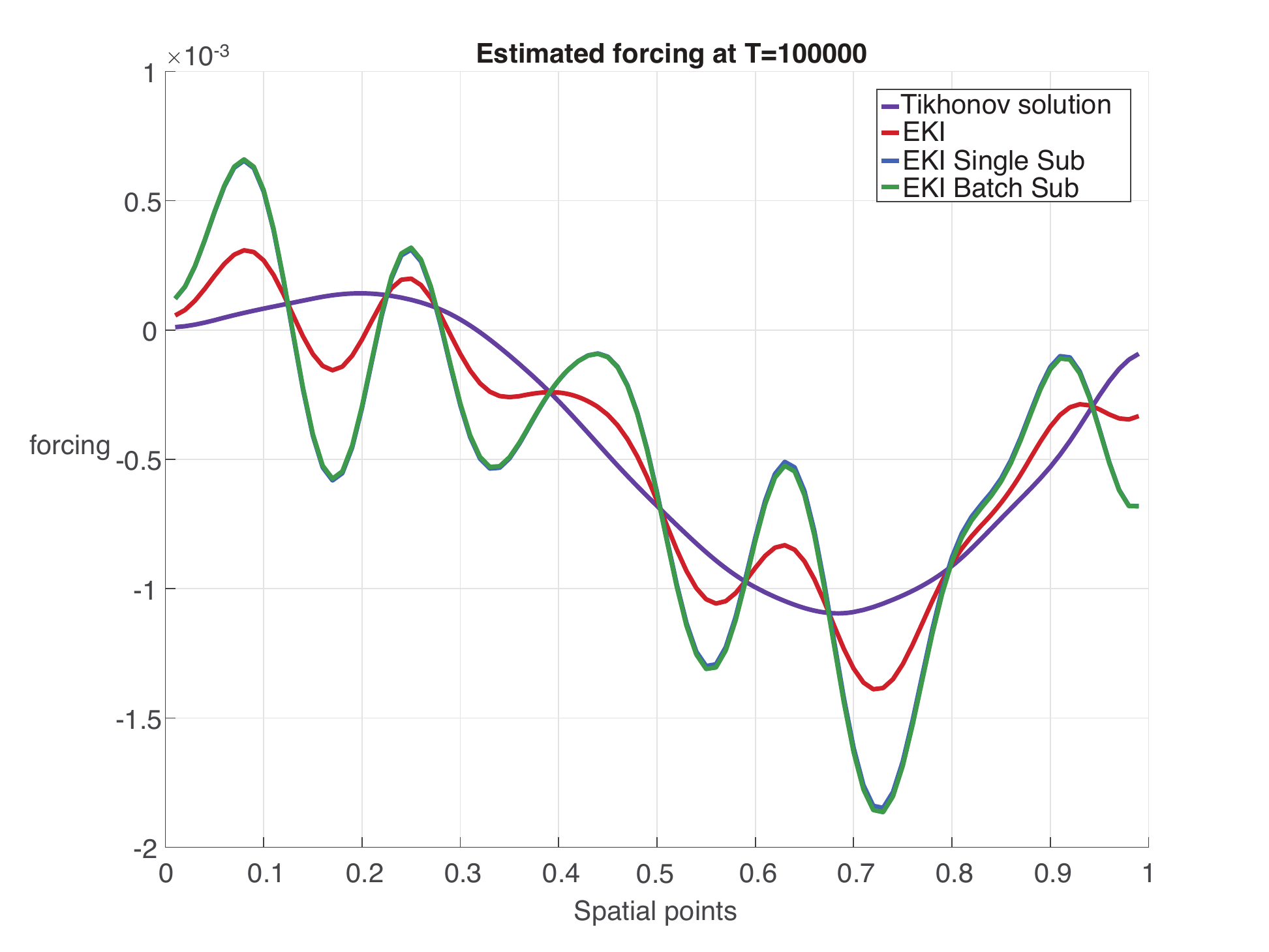}
        \includegraphics[scale=0.39, trim = 1.8cm 0cm 1.8cm 0cm,clip]{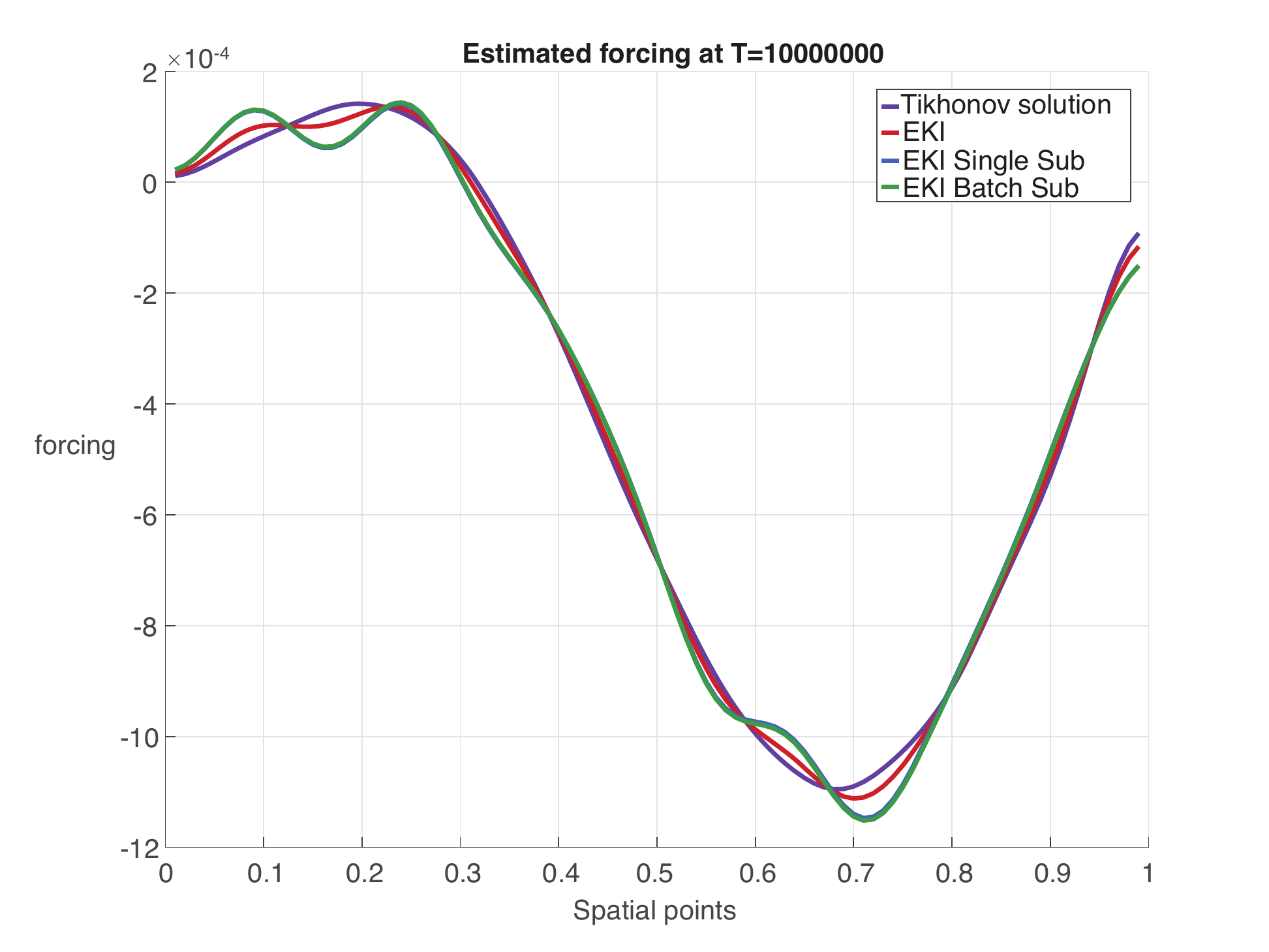}
\caption{Development of computed solution over time in comparison with the Tikhonov solution at times $T=10^{-1}$ (upper left), $T=10^2$(upper right), $T=10^5$ (lower left) and $T=10^7$ (lower right). The red line illustrates the EKI, the blue line single-subsampling, the green line batch-subsampling and purple is the reference solution $\theta^\star$.}
\label{fig:dev_hq_time}
\end{figure}

Figure \ref{fig:dev_hq_time} shows the development of the computed solution over time. Here one can see that there is a slight difference between single and batch-subsamling at $T=0.1$. However, both methods quickly converge to one another as one can see in the computed results at $T=100$. As mentioned above the EKI is slightly faster due to beginning a bit earlier. However at $T=10000000$ all methods approximate the Tikhonov solution quite similarly.

\subsection{Nonlinear 2D Darcy flow} \label{sec:secondexample}

We introduce in this section one experiment with a non-linear forward operator $G$. Even though our theory only covers the linear case, we will illustrate that subsampling also leads to good results in the nonlinear setting. As subsampling strategy we only consider single-subsampling. The example is motivated by \cite{Chada2020} and \cite{Garbuno2020}.\\
Consider the following elliptic PDE.
\begin{equation} \label{pde:2ddarcy}
    \left\{
    \begin{aligned}
                -\nabla \cdot \left(\exp(u)\nabla p\right)  &= f,  \qquad  x\in D\\
            p&=0,   \qquad     x\in\partial D
    \end{aligned}
    \right.,
\end{equation}
where $D=\left(0,1\right)^2$. We seek to recover the unknown diffusion coefficient $u^\dagger\in C^1(D)=X$, given observation of the solution $p\in H_0^1(D)\cap H^2(D):=\mathcal{V}$. Furthermore, we assume that the scalar field $f\in\mathbb{R}$ is known.\\
The observations are given by
$$y=\mathcal{O}(p)+\eta,$$
where $ \mathcal{O}(p):\mathcal{V}\rightarrow \mathbb{R}^K $ is the observation Operator, that considers $K$ randomly chosen points in $X$, i.e. $\mathcal{O}(p)=\left(p(x_1),...,p(x_K)\right)$. Finally, $\eta$ denotes the noise on our data and is assumed to be Gaussian, i.e. a realisation of $\mathcal{N}(0,\Gamma)$, where $\Gamma=0.1^2 \mathrm{Id}_K$\\
Then our inverse problem is given by
$$y=\mathcal{G}(u)+\eta,$$

where $\mathcal{G}=\mathcal{O}\circ G$ and $G:X\rightarrow\mathbb{R}^K$ denotes the solution operator of the PDE \eqref{pde:2ddarcy}. We solve the PDE on a uniform mesh with a grid size of $h=2^{-8}$ using a FEM method with continuous, piecewise linear finite element basis functions.
We model our prior distribution as the random field
$$u(x,\omega)=\sum_{i=1}^s\lambda_i^{1/2}e_i(x)\xi_i(\omega),$$
where we have $\lambda_i=\left(\pi^2(k_j^2+l_j^2)+\tau^2\right)^{-\alpha}$ and $e_i(x)=\cos(\pi x_1 k_j)\cos(\pi x_2 l_j)$ with $\tau=0.01,\alpha=2,s=25,(k_j,l_j)_{j\in\{1,...,s\}}\in\{1,...,s\}^2$. The variables $\xi_i$ are i.i.d standard normal variables. Afterwards we make $\Nens=10$ independent draws for our initial ensemble.\\
The dimension of the parameter space is $d=2^8$ due to the grid size. We take $K=30$ observations and divide them into $\Nsub=5$ many subsets.\\
We use a linear decaying learning $\gamma(t)=(a+b t)^{-1}$, where $a,b=10$. As regularisation factor we consider $\beta=10$ and compute the solution up until time $T=10^5$. Again we note that due to the decrease of the switching times of the data sets the algorithm becomes computationally very slow. Therefore, we only use a linear decaying switching rate up until time $T=10^1$ from there on we consider $10^5$ equidistant switching times. 

Note that we again work in a subspace that is smaller than $X$. Therefore, we need to compare the computed solution with the respective one given the subspace.\\
The particles stay in the affine space $u_0^\perp +\mathcal E$ for all $t \ge 0$. Therefore, the reference solution is given by $u_0^\perp +u_{\mathcal E,j}^{\dagger}$. We formulate it as a constrained optimisation problem
\begin{equation*}
    \min_{u\in\mathcal E+\theta_0^\perp} \frac12  \|G (u) -   y\|_\Gamma^2+\frac{\beta}{2}\|u\|^2\,,
\end{equation*}

and use MATLABs \verb+fmincon+ solver to compute it.\\

\begin{figure}
    \centering
    \captionsetup{width=.9\linewidth}
    \includegraphics[width=5.3in,trim={1.8cm 4.8cm 0.8cm 4.4cm},clip]{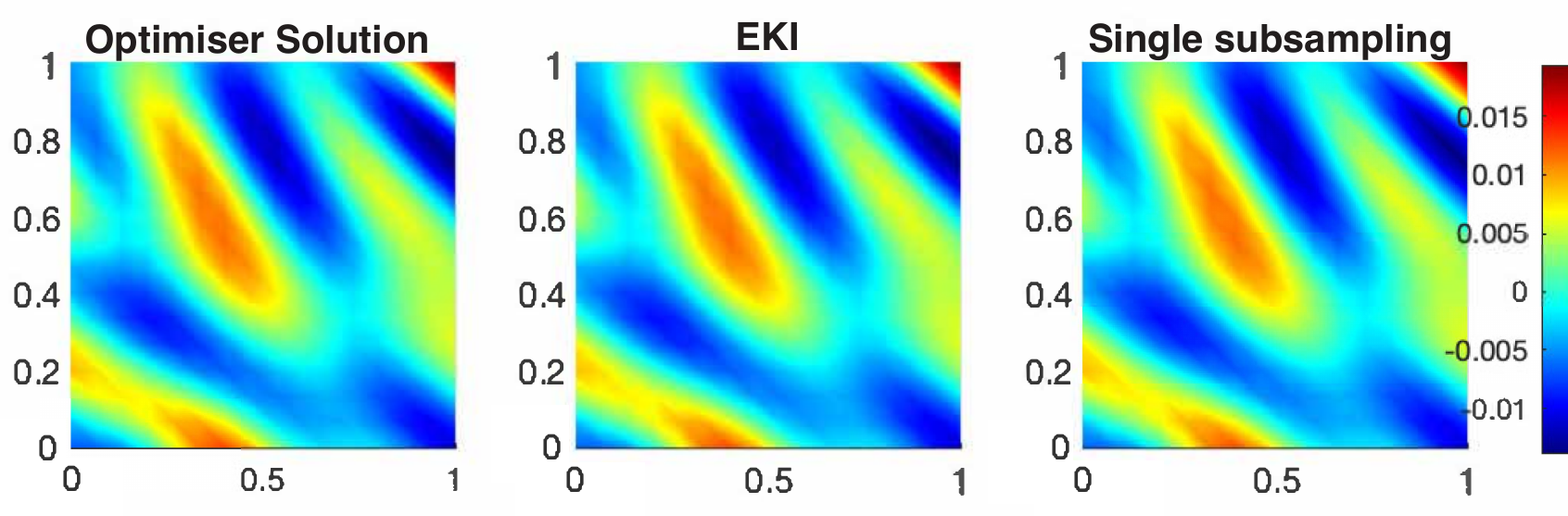}
    \caption{Comparison of computed diffusion coefficient $u$. Left figure: optimiser; middle figure: EKI; right figure: single-subsampling.}
    \label{fig:com_solutions_param}
\end{figure}

Figure \ref{fig:com_solutions_param} shows the computed solutions given by the different algorithms. The left picture is the solution given by MATLABs \verb+fmincon+ solver. In the middle, the solution computed by the normal EKI is depicted and on the right the result of the single-subsampling algorithm shown. One can see that both algorithms compute visually very similar solutions as the optimiser.

\begin{figure}[H]
\centering
\captionsetup{width=.9\linewidth}
\includegraphics[scale=0.42, trim = 1cm 0cm 0cm 1.1cm,clip]{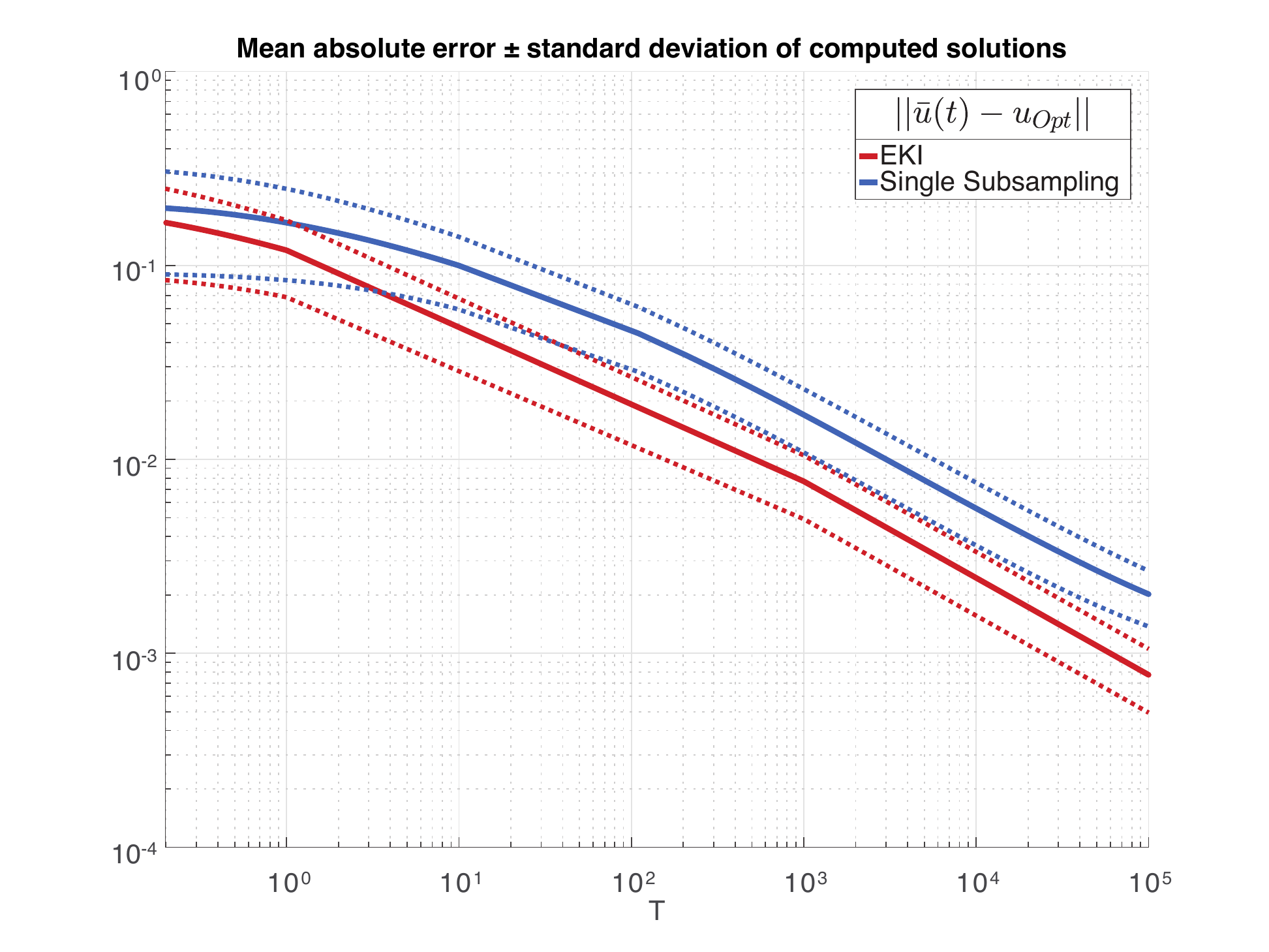}
\hspace*{-0.8cm}%
\includegraphics[scale=0.42, trim = 1.5cm 0cm 0cm 1cm,clip]{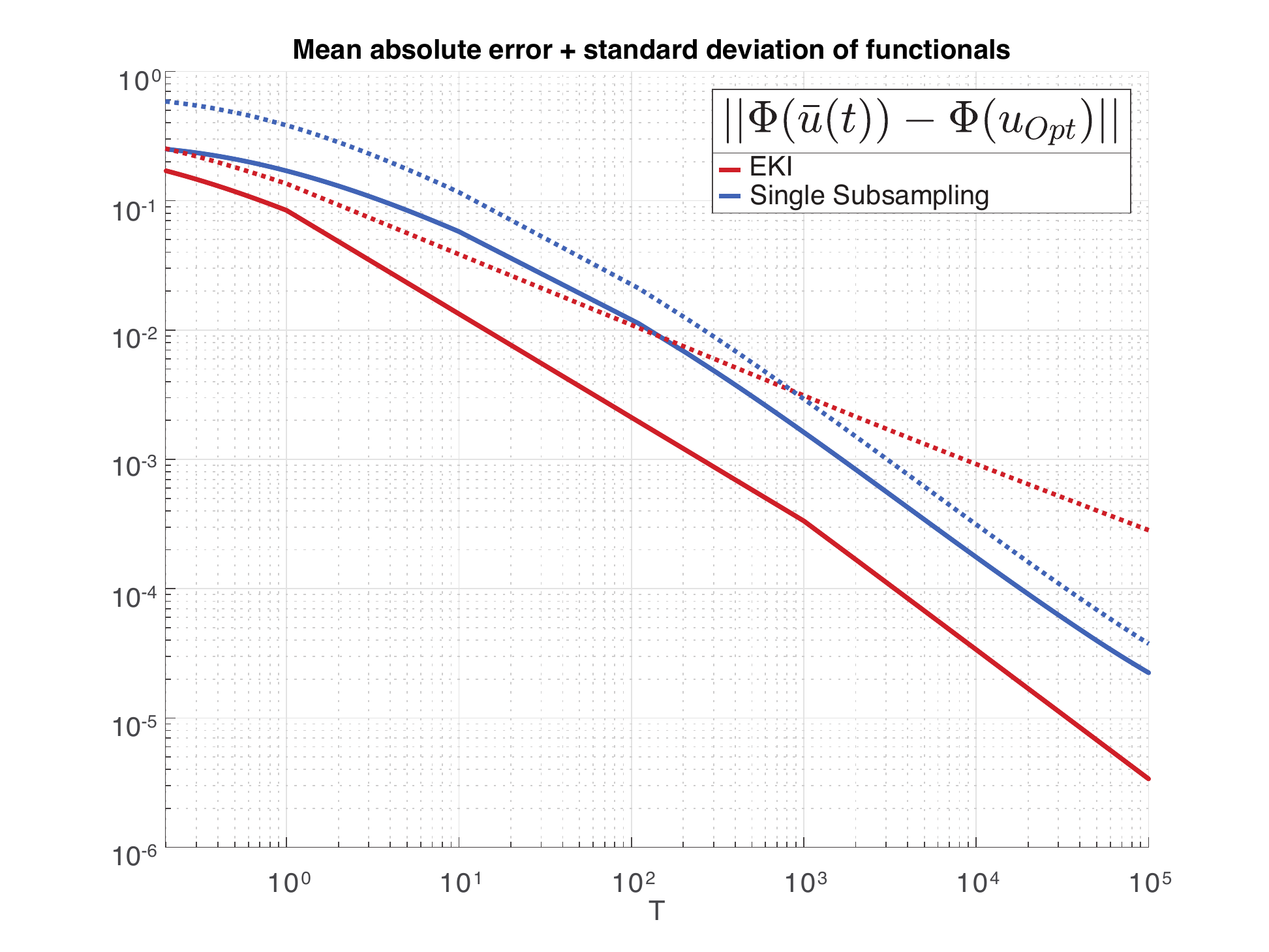}
\caption{Mean absolute errors $\pm$ standard deviation of computed solutions in the parameter (left) and observation space (right). The red line illustrates the EKI and the blue line single-subsampling.}
\label{fig:err_t1e5_beta1}
\end{figure}

Figure \ref{fig:err_t1e5_beta1} depicts the mean errors of $N=32$ runs $\pm$ one standard deviation in the parameter space (left subplot) as well as of the functionals (right subplot) evaluated in the corresponding solutions. We can see that both errors behave similarly and are converging towards zero at an algebraic rate. As in the linear example we note that the mean minus standard deviation is negative in the
right picture and therefore not shown.

\begin{figure}[H]
\centering
\captionsetup{width=.9\linewidth}
\includegraphics[scale=0.42, trim = 0.01cm 0cm 0cm 0.4cm,clip]{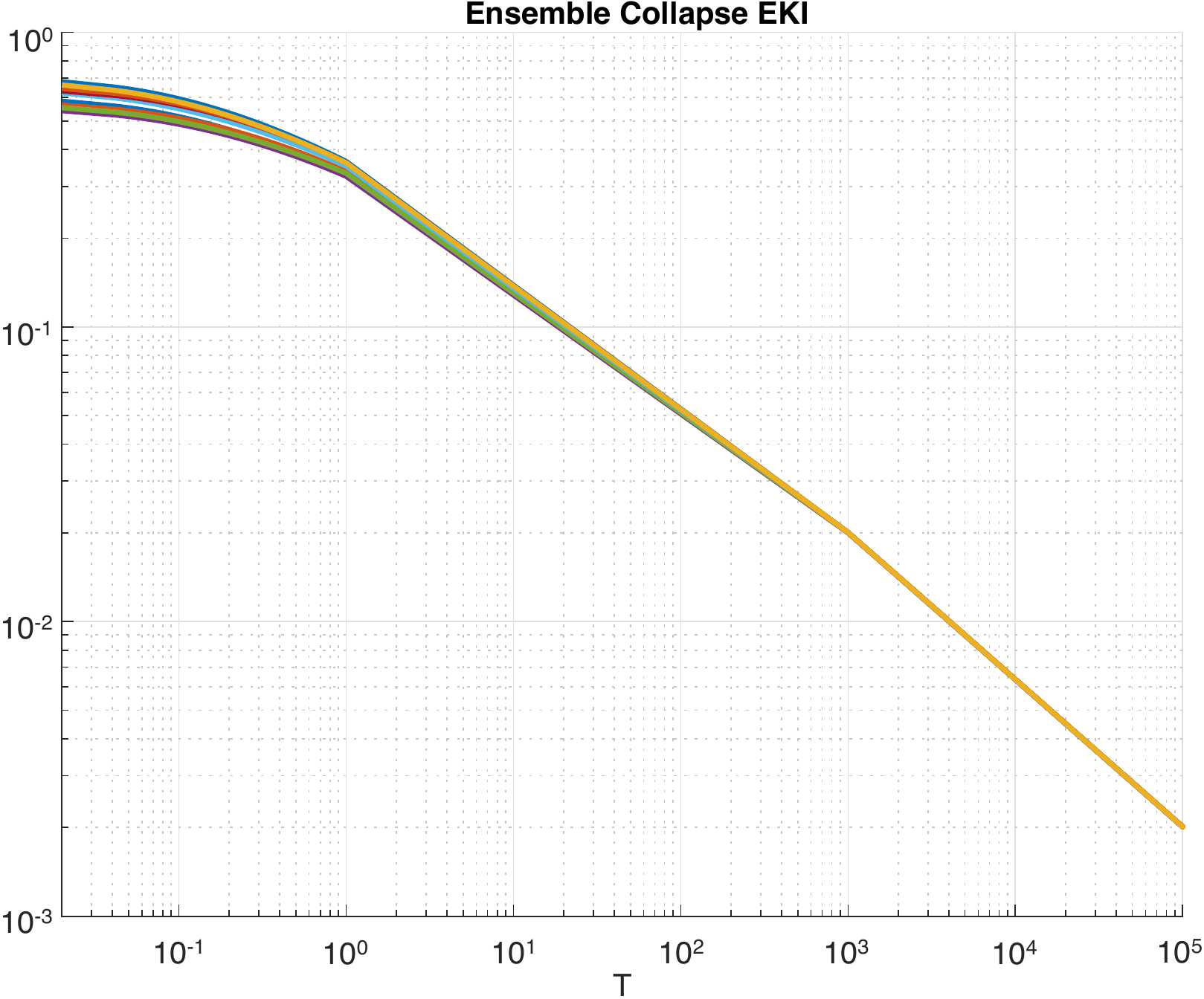}
\includegraphics[scale=0.42, trim = 1.5cm 0.5cm 0cm 1.1cm,clip]{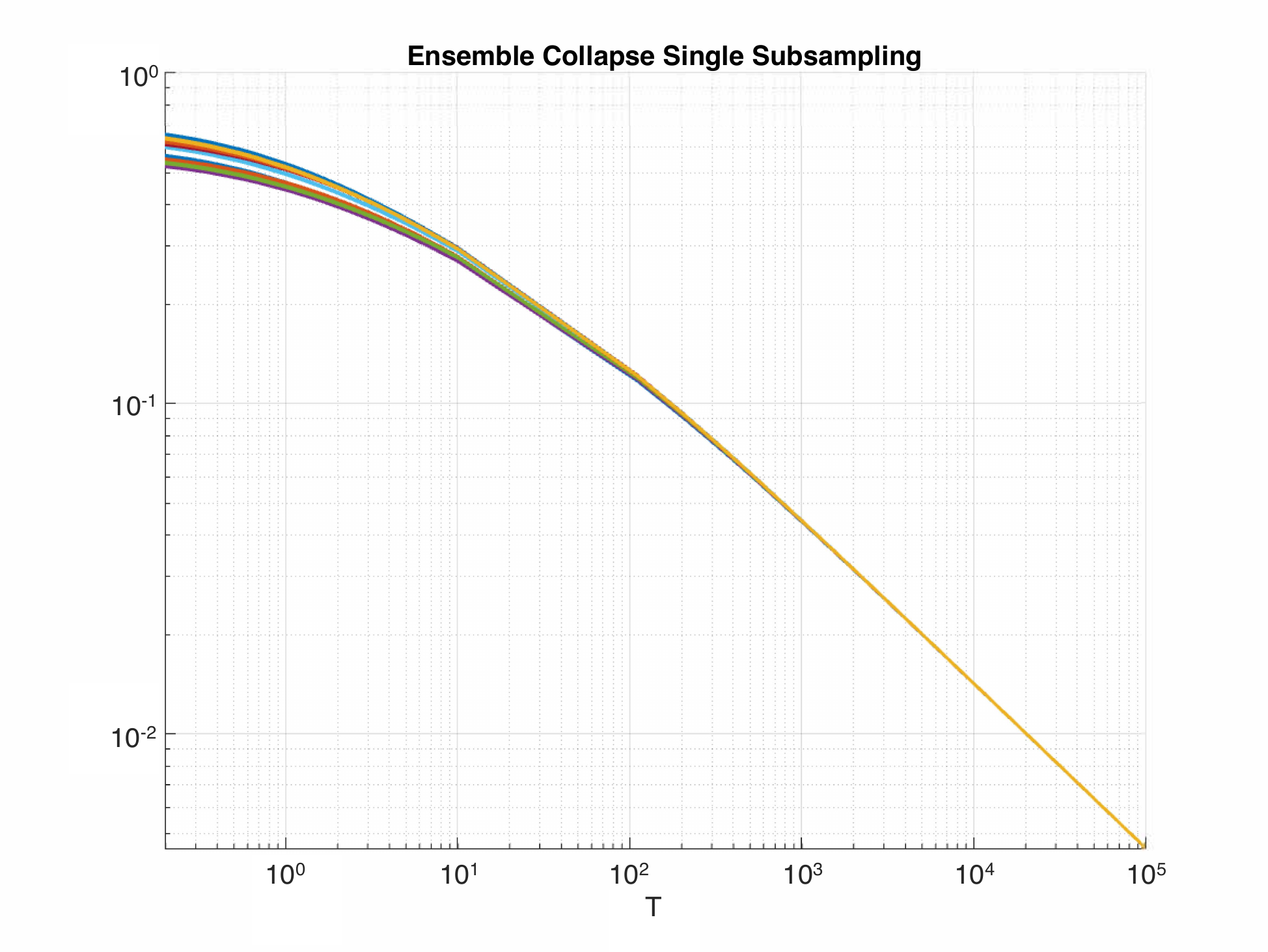}
\caption{Mean ensemble collapse of all $\Nens=10$ particles for the two methods. The different colors represent the ensemble collapse of one ensemble member respectively. Left figure: EKI; right figure: single-subsampling.}
\label{fig:ek_t1e5_beta1}
\end{figure}

Figure \ref{fig:ek_t1e5_beta1} shows the mean ensemble collapse of the $N=32$ runs. Again the left figure shows the results of the EKI, whereas the right for single subsampling. We can see that in both methods the collapse occurs at a similar rate. One should note that single subsampling has a smoother result, this however is only due to the amount of observations where we evaluate our solutions. Due to the required frequent changes in subsampling, we obtain more observations.

\section{Conclusions} \label{sec_conclu}
We have introduced subsampling schemes for EKI to allow the application of the method also in the large data regime. Based on recent results on continuous stochastic gradient processes \cite{Latz2021}, two subsampling approaches, $(i)$ single subsampling, where each particle obtains the same data set when switching the data and $(ii)$ batch-subsampling where data sets may differ for each particle, have been considered in the continuous-time setting. By applying Tikhonov regularisation and variance inflation on both methods $(i)$ and $(ii)$ we were able to show convergence of the schemes to the solution of the original EKI version. For the non-variance inflated variant of $(i)$ convergence results with an algebraic rate could be proven. For batch-subsampling we were only able to show convergence when using a vanishing variance inflation over time. However, our numerical experiments in section \ref{sec:firstexample} also showed similar convergence results for non-variance inflated batch-subsampling. The analysis requires the control of the eigenvalues of the empirical covariance w.r. to the initial ensemble. This will be subject to future work. 
Further, we also considered in section \ref{sec:secondexample} a numerical experiment for a non-linear forward operator. Single-subsampling without variance inflation shows similar convergence results as the original EKI. Analysis of subsampling techniques for non-linear forward operators will be also subject for future work.

\section*{Acknowledgements}
MH is grateful to the DFG RTG1953 “Statistical Modeling of Complex Systems and Processes” for funding of this research. JL thanks the Engineering and Physical Sciences Research Council for their support through grant EP/S026045/1. CS acknowledges support from MATH+ project EF1-19: Machine Learning Enhanced Filtering Methods for Inverse Problems and EF1-20: Uncertainty Quantification and Design of Experiment for Data-Driven Control, funded by the Deutsche Forschungsgemeinschaft (DFG, German Research
Foundation) under Germany's Excellence Strategy – The Berlin Mathematics
Research Center MATH+ (EXC-2046/1, project ID: 390685689). We are also grateful to the state of Baden-Württemberg through bwHPC.

\appendix

\section{Proof of Theorem~\ref{thm_gen_Latz21}}
We now prove Theorem~\ref{thm_gen_Latz21} which we first recall below.
\begin{theorem}[Theorem \ref{thm_gen_Latz21}]\label{thm_gen_Latz21_app} Let Assumption~\ref{Assum_conv}(i)-(ii) hold for a function $h$ and a stochastic approximation process $(\bsi(t), \theta(t))_{t \geq 0}$ with initial values $(i_0, \theta_0) \in I \times X$. Then,
$$
\lim_{t \rightarrow \infty}\mathrm{d}_{\rm W}\left(\delta(\cdot - \theta^*), \mathbb{P}(\theta(t) \in \cdot | \theta_0, i_0)\right) = 0
$$
\end{theorem}
To prove this theorem, we first define an auxiliary process $(\bsi^{(\varepsilon)}(t), \theta^{(\varepsilon)}(t))_{t \geq 0}$ that converges to $(\bsi(t), \theta(t))_{t \geq 0}$ as $\varepsilon \rightarrow 0$, but has a bounded transition rate matrix. 
For this process, we show Wasserstein ergodicity. 
Let $\varepsilon > 0$ and $B(t, \varepsilon) := A(-\log(\varepsilon + \exp(-t)))$, $t\geq 0$ where $(A(t))_{t \geq 0}$ is the transition rate matrix given in \eqref{eq_transition_rate_mat}. We then define $(\bsi^{(\varepsilon)}(t))_{t \geq 0}$ to be the stochastic process with transition rate matrix $B(t, \varepsilon)$ and $(\theta^{(\varepsilon)}(t))_{t \geq 0}$ to be the related stochastic approximation process. Moreover, we denote by $K_{t|t_0}^{(\varepsilon)} := \mathbb{P}(\theta^{(\varepsilon)}(t) \in \cdot | \theta^{(\varepsilon)}(t_0) = \cdot)$ the Markov kernel associated with $(\theta^{(\varepsilon)}(t))_{t \geq 0}$, we ignore the underlying dependency on $(\bsi^{(\varepsilon)}(t))_{t \geq 0}$. Similarly, we write $K_{t|t_0} := \mathbb{P}(\theta(t) \in \cdot | \theta(t_0) = \cdot)$.

In a first auxiliary result, we show that $(\theta^{(\varepsilon)}(t))_{t \geq 0}$ is ergodic and converges to a unique stationary measure.
\begin{lemma}\label{aux_lemma} Let Assumption~\ref{Assum_conv}(i)-(ii) hold and let $\varepsilon > 0$.
Then, there is a unique probability measure $\mu_{\varepsilon}$ such that for any initial  distribution $\mu_0 := \mathbb{P}(\theta^{(\varepsilon)}(0) \in \cdot)$ with finite second moment, we have
$$\lim_{t \rightarrow \infty}\mathrm{d}_{\rm W}(\mu_{\varepsilon}, \mathbb{P}(\theta^{(\varepsilon)}(t)) \in \cdot)) = 0.$$
\end{lemma}
\begin{proof}
1. Let $\mu_0, \mu_0^\dagger$ be to different initial distributions. Moreover, let $(\theta^{(\varepsilon)}(t))_{t \geq 0}$, $(\theta_{\dagger}^{(\varepsilon)}(t))_{t \geq 0}$ be two realisations of the stochastic approximation process with $\mu_0 = \mathbb{P}(\theta^{(\varepsilon)}(0) \in \cdot)$ and $\mu_0^\dagger = \mathbb{P}(\theta_{\dagger}^{(\varepsilon)}(0) \in \cdot)$. Furthermore, we assume that the processes are coupled through the associated index processes. Indeed, we assume that $(\bsi^{(\varepsilon)}(t))_{t \geq 0}$ and $(\bsi^{(\varepsilon)}_{\dagger}(t))_{t \geq 0}$ are almost surely identical.
Then, 
$$\mathrm{d}_{\rm W}(\mu_{0}K_{t|0}^{(\varepsilon)}, \mu_{0}^\dagger K_{t|0}^{(\varepsilon)})^2 \leq \mathbb{E}[\|\theta^{(\varepsilon)}(t) -  \theta^{(\varepsilon)}_{\dagger}(t)\|^2] \qquad (t \geq 0).$$
Assumption~\ref{Assum_conv}(ii) implies that 
\begin{align*}
    \frac{\mathrm{d}\|\theta^{(\varepsilon)}(t) -  \theta^{(\varepsilon)}_{\dagger}(t)\|^2}{\mathrm{d}t} = 2  \langle \mathbf{F}_{i}(\theta^{(\varepsilon)}(t))-  \mathbf{F}_{i}(\theta_{\dagger}^{(\varepsilon)}(t)), \theta^{(\varepsilon)}(t) - \theta_{\dagger}^{(\varepsilon)}(t) \rangle_X \leq - 2h(t) \|\theta^{(\varepsilon)}(t) - \theta_{\dagger}^{(\varepsilon)}(t)\|^2
\end{align*}
Now, the Gr\"onwall inequality implies that 
$$
\|\theta^{(\varepsilon)}(t) -  \theta^{(\varepsilon)}_{\dagger}(t)\|^2 \leq \exp\left(-2\int_{0}^t h(t) \mathrm{d}t\right) \|\theta^{(\varepsilon)}(0) - \theta_{\dagger}^{(\varepsilon)}(0)\|^2.
$$
Taking expectations on both sides, we obtain:
\begin{equation} \label{eq_contraction}
    \mathrm{d}_{\rm W}(\mu_{0}K_{t|0}^{(\varepsilon)}, \mu_{0}^\dagger K_{t|0}^{(\varepsilon)})^2 \leq \exp\left(-2\int_{0}^t h(t) \mathrm{d}t\right) \mathbb{E}\left[ \|\theta^{(\varepsilon)}(0) - \theta_{\dagger}^{(\varepsilon)}(0)\|^2
\right]
\end{equation}

which converges to 0 as $t \rightarrow \infty$ by Assumption~\ref{Assum_conv}(ii). 

2. To show the assertion of the theorem, we now choose some $\delta > 0$ and consider the process  $\theta_{\dagger}^{(\varepsilon)}(t):= \theta^{(\varepsilon)}(t+\delta)$. The contraction property \eqref{eq_contraction} implies that
$$
\mathrm{d}_{\rm W}(\mu_{0}K_{t|0}^{(\varepsilon)}, \mu_{0} K_{t+\delta|0}^{(\varepsilon)}) \rightarrow 0, \qquad (t \rightarrow \infty).$$
Thus, the sequence $(\mu_{0}K_{\delta n|0}^{(\varepsilon)})_{n=0}^\infty$ is a Cauchy sequence on the Wasserstein space associated to $\mathrm{d}_{\rm W}$. Thus, due to the completeness of Wasserstein spaces, see \cite{Bolley2008}, we have that $\mu_{0}K_{t|0}$ converges to some probability distribution $\mu_\varepsilon$ as $t \rightarrow \infty$ in $\mathrm{d}_{\rm W}$. Again due to the contraction given in \eqref{eq_contraction}, $\mu_\varepsilon$ does not depend on the initial distribution $\mu_0$ and is, thus, unique.
\end{proof}
  In the second auxiliary result, we show that  $(\theta^{(\varepsilon)}(t))_{t \geq 0}$ converges to $(\theta(t))_{t \geq 0}$, as $\varepsilon \downarrow 0$, and that $\pi^\varepsilon$ converges to $\delta(\cdot - \theta^*)$. This result is a small extension of Proposition 4 in \cite{Latz2021} and the proof proceeds identically.
\begin{proposition} \label{aux_prop} Let Assumption~\ref{Assum_conv}(i)-(ii) hold. Then, we have
\begin{enumerate}
    \item  
    $\mathrm{d}_{\rm W}(\mu_0K_{t|t_0}^{(\varepsilon)},\mu_0K_{t|t_0}^{(\varepsilon)}) \leq \alpha'(\varepsilon)$ $(\varepsilon > 0)$ for any initial distribution $\mu_0$,
    \item $\mathrm{d}_{\rm W}(\mu_\varepsilon,\delta(\cdot - \theta^\varepsilon) \leq \alpha''(\varepsilon)$ $(\varepsilon > 0)$,
\end{enumerate}
where $\alpha', \alpha'': [0,\infty) \rightarrow [0,\infty)$ are continuous and equal to 0 at 0.
\end{proposition}
The proof of Theorem~\ref{thm_gen_Latz21} now consists in a simple rearrangement of the auxiliary results above.
\begin{proof}[Proof of Theorem \ref{thm_gen_Latz21}]
By the triangular inequality, we have 
\begin{align*}
    \mathrm{d}_{\rm W}&\left(\delta(\cdot - \theta^*), \mathbb{P}(\theta(t) \in \cdot | \theta_0, i_0)\right)  \\ &\leq  \mathrm{d}_{\rm W}\left(\delta(\cdot - \theta^*), \mu_\varepsilon\right) + \mathrm{d}_{\rm W}(\mu_{\varepsilon}, \mathbb{P}(\theta^{(\varepsilon)}(t)) \in \cdot)) + \mathrm{d}_{\rm W}(\mathbb{P}(\theta^{(\varepsilon)}(t)) \in \cdot), \mathbb{P}(\theta(t)) \in \cdot)) 
\end{align*}
where the last term in the sum is identical to $\mathrm{d}_{\rm W}(\mu_0K_{t|t_0}^{(\varepsilon)},\mu_0K_{t|t_0}^{(\varepsilon)})$. By Lemma~\ref{aux_lemma} and Proposition~\ref{aux_prop}, we have $\mathrm{d}_{\rm W}\left(\delta(\cdot - \theta^*),\mathbb{P}(\theta(t) \in \cdot | \theta_0, i_0)\right)  \rightarrow 0$, as $t \rightarrow \infty$.
\end{proof}

\begin{lemma}[Lemma \ref{thm:Expconv}]\label{thm:Expconv_app}
Given the $\Nsub$ subsamples $\left(\tilde y_{(1)}^\dagger, \ldots, \tilde y_{(\Nsub)}^\dagger\right)$, assume that the centered initial ensemble is a generator of the full space $X$, i.e. $\spann{\{e_0^{(j)},j \in J\}}=X$. We further assume that $\sum_{j=1}^\Nens (\theta^\dagger_j-\bar \theta^\dagger)(\theta^\dagger_j-\bar \theta^\dagger)^\top$ has full rank $d$. Then the particles converge to the true solution $\theta^\dagger$ exponentially fast, i.e. $\theta^{(j)}\to \theta^\dagger_i$ with $\theta^\dagger_i$ denoting the minimiser of $\Phi_i(\theta)=\frac12 \|\tilde A_i \theta - \tilde y_i\|^2$.
\end{lemma}
\begin{proof} 
Note that we do not switch the data. Therefore, the subset that each particle obtains does not depend on time.\\
The gradients $g^{(j)}=\nabla \Phi_{(\bsi;j)}(\theta)=\tilde A_{(\bsi;j)}^\top (\tilde A_{(\bsi;j)} \theta^{(j)}- \tilde y_{(\bsi;j)})$ satisfy
\begin{align}\nonumber
\frac{\mathrm{d} \tilde A_{(\bsi;j)}^\top \tilde A_{(\bsi;j)} u^{(j)}(t)}{\mathrm{d}t} &= - \tilde A_{(\bsi;j)}^\top \tilde A_{(\bsi;j)} \hat{C}_t \tilde A_{(\bsi;j)} ^\top (\tilde A_{(\bsi;j)} \theta^{(j)}(t)- \tilde y_{(\bsi;j)})\,,
\end{align}

We will prove in the following that $g^{(j)}\to 0$ exponentially fast as $t\to\infty$. This is a sufficient and necessary optimality condition as $A_{\bsi(;j)}^\top \tilde A_{(\bsi;j)}$ is positive definite due to regularisation. We obtain
\begin{align}
\frac12\frac{\mathrm{d} \|g^{(j)}\|^2_{\tilde A_{(\bsi;j)}^\top \tilde A_{(\bsi;j)}}}{\mathrm{d}t} &= - \frac 1\Nens \sum_{k=1}^\Nens \langle e^{(k)},g^{(j)}\rangle^2\le 0, \label{eqn:subspaceineq}
\end{align}
i.e. the gradients are monotonically decreasing. To prove convergence, we will now show that $\frac12\frac{\mathrm{d}}{\mathrm{d}t} \|g^{(j)}\|^2_{\tilde A_{(\bsi;j)}^\top \tilde A_{(\bsi;j)}} <0$.  Note that, if $g^{(j)}\neq 0$ and $\{e^{(k)}\}_{k=1}^\Nens$ is still a generating set of $X$ at time $t$, then there exists at least one $k\in\{1,\ldots,\Nens\}$ such that $\langle e^{(k)},g^{(j)}\rangle\neq 0$.
The quantity $e^{(j)}$ satisfies 
\begin{align*}
\frac{\mathrm{d} e^{(j)}(t)}{\mathrm{d}t} &= - \widehat{C}_t v^{(j)} =-\frac 1\Nens \sum_{k=1}^\Nens \langle v^{(j)} , e^{(k)}\rangle e^{(k)}\,,
\end{align*} 
with $v^{(j)}=g^{(j)}-\bar g$. Thus the dynamical behavior of empirical covariance is given by
\begin{equation*}
    \frac{\mathrm{d} }{\mathrm{d}t} \hat C=D\hat C+\hat C D^\top
    \end{equation*}
with $D=-\frac1\Nens \sum_{j=1}^\Nens v^{(j)} \otimes e^{(j)}$.
 
Therefore, the rank of the empirical covariance stays constant over time (cp. \cite{Reid}) and the members $\{e^{(j)}\}_{j=1}^\Nens$ still form a generating set of $X$ at time $t$.\\
Then, there exists at least one $k$ in \eqref{eqn:subspaceineq} such that $\langle e^{(k)},g^{(j)}\rangle^2\neq 0$, i.e. the gradients converge to $0$.\\
This implies the convergence $\theta^{(j)}\to \theta^\dagger_j$ due to the strong convexity. 
By assumption, the limit of the empirical covariance has full rank $d$, 
since $\theta_j\to\theta^\dagger_j$, i.e. the minimal eigenvalue $\lambda_{\min}(\hat{C}_t)$ of the empirical covariance is bounded from below uniformly in time.
Thus, we have
\begin{align*}
\frac12\frac{\mathrm{d} \|g^{(j)}\|^2_{\tilde A_{(\bsi;j)}^\top \tilde A_{(\bsi;j)}}}{\mathrm{d}t} &= -  \langle g^{(j)}, C(t)g^{(j)}\rangle \le -\frac{\lambda_{\min}}{\lambda_{\max}((\tilde A_{\bsi;}^\top \tilde A_{\bsi})^{-1})}\|g^{(j)}\|^2_{\tilde A_{(\bsi;j)}^\top \tilde A_{(\bsi;j)}}\,,
\end{align*}
where $\lambda_{\min}>0$ denotes the lower bound on the minimal eigenvalue of the empirical covariance.\\

\end{proof}

\begin{lemma} \label{lemma:convex_vi_ss_app}
The particles $\theta^{(j)}$ converge at exponential speed to the unique solution of the regularised data misfit, i.e. $\rho^{(j)}(t)\to 0$. Hence, there exists a (unique) $\theta_j^\dagger\in S=span\{\theta^{(1)}_0,...,\theta^{(J)}_0\}$ such that $\theta^{(j)}\to \theta_j^\dagger$.\\

Furthermore, let $\theta_1$ and $\theta_2$ be two coupled process with different initial values $\theta_1(0),\theta_2(0).$ Then there exists a measurable function $h: [0, \infty) \rightarrow \mathbb{R}$ such that the following holds

\begin{align*}
-\langle \theta_1 -  \theta_2, (\widehat{C}_t^1+\alpha_{vi}C_{vi}) D_\theta \Phi_{{i'}(t,\varepsilon;\cdot)}( \theta_1) - (\widehat{C}_t^1+\alpha_{vi}C_{vi}) D_\theta \Phi_{{i'}(t,\varepsilon;\cdot)}( \theta_2) \rangle \leq -h(t) \| \theta_1 - \theta_2\|^2,
\end{align*}
 for $t$ large enough. We have
 \begin{enumerate}
     \item Single-Subsampling with variance inflation:
     
     \hspace{1cm}
     $h(t)=\alpha_{vi}\lambda_{\min}(C_{vi})\min_{i\in\{1,...,\Nsub\}}\lambda_{\min}(\tilde A_{\bsi}^T \tilde A_{\bsi}).$
     \item Batch-Subsampling with variance inflation:
     
     \hspace{1cm} $h(t)=\alpha_{vi}\lambda_{\min}(C_{vi})\min_{i\in\{1,...,\Nsub\}}\lambda_{\min}(A_{\bsi}^TA_{\bsi})$
     \item Single-Subsampling without variance inflation:
     
     \hspace{1cm} $h(t)=\lambda_{\min}(\widehat C_t)\min_{i\in\{1,...,\Nsub\}}\lambda_{\min}(\tilde A_{\bsi}^T \tilde A_{\bsi}).$
     \item Batch-Subsampling with diminishing variance inflation:
     
     \hspace{1cm} $h(t)=\frac{\alpha_{vi}}{t}\lambda_{\min}(C_{vi})\min_{i \in\{1,...,\Nsub\}}\lambda_{\min}(A_{\bsi}^TA_{\bsi}).$
 \end{enumerate}

\end{lemma}

\begin{proof}

    $1.$ We first consider single subsampling with variance inflation

The gradients $\tilde A_{(\bsi;j)}^\top \rho^{(j)}$ for a constant (w.r. to time and particle) data stream satisfy the following differential equation for all subsets $\bsi$.
\begin{align*}
\frac{\mathrm{d} \tilde A_{\bsi}^\top \rho^{(j)}(t)}{\mathrm{d}t} &= - \tilde A_{\bsi}^\top \tilde A_{\bsi}[\widehat{C}_t+\alpha_{\rm vi} C_{\rm vi}] \tilde A_{\bsi}^\top \rho^{(j)}(t).
\end{align*}
The norm of the gradients thus satisfies
\begin{align*}
\frac12 \frac{\mathrm{d} \|\tilde A_{\bsi}^\top \rho^{(j)}(t)\|_{\tilde A_{\bsi}^\top \tilde A_{\bsi}}^2}{\mathrm{d}t} &\le - \alpha_{\rm vi}\frac{\lambda_{\min}(C_{\rm vi})}{\lambda_{\max}(\tilde A_{\bsi}^\top \tilde A_{\bsi})} \|\rho^{(j)}(t)\|_{\tilde A_{\bsi}^\top \tilde A_{\bsi}}^2\,,
\end{align*}
which implies the exponential convergence of the mapped residuals and with the injectivity of the modified forward operator the exponential convergence in the parameter space to the (unique) solution of the regularised data misfit.\\

Therefore, we have $\theta^{(j)}\to \theta^\dagger_j$.\\
Then $\theta^\dagger_j$ is an equilibrium point of
$$\mathbf{F}_i(\theta^{(j)},t)=(\widehat C_t+\alpha_{vi}C_{vi}) D_\theta \Phi^{\mbox{\rm reg}}_{\bsi}(\theta^{(j)}(t))\qquad (j \in J)\,.$$ Hence, there exists a function $\kappa(t)\geq 0 $ for all $t>0$ that converges exponentially fast to $0$ for $t \rightarrow \infty$ such that:

$$\|\theta^{(j)}_t- P_Y\theta^\dagger_j\|\leq \kappa(t) \quad \forall t\geq0.$$





Due to linearity w.r. to the initial values, we obtain
\begin{align*}
    \|\theta(t,\theta_0)^{(j)}-\theta(t,\theta_1)^{(j)}\|&=\|\theta(t,\theta_0)^{(j)}-P_Y\theta^\dagger_j-(\theta(t,\theta_1)^{(j)}-P_Y\theta^\dagger_j)\|\leq 2 \kappa(t).
\end{align*}

The next step is to consider equation $(ii)$ from Assumption \ref{Assum_conv}. Note that $\theta_1$ and $\theta_2$ are column vectors consisting of the stacked particle vectors.

Therefore, we define the following matrices to represent the gradient flow for the stacked vector. We set:
$\tilde A=diag\{\tilde A_1,\tilde A_2,...,\tilde A_\Nsub\}, \tilde A^T=diag\{\tilde A_1^T,\tilde A_2^T,...,\tilde A_\Nsub^T\},\mathbf{C}_t=diag\{\widehat C_t+\alpha_{vi}C_{vi},\widehat C_t+\alpha_{vi}C_{vi},...,\widehat C_t+\alpha_{vi}C_{vi}\}$.\\
Then the dynamics are given by:
$$\frac{\mathrm{d}\theta}{\mathrm{d}t}=-\hat {C}\tilde A^T(\tilde A\theta-y).$$

We want to show:

$$
-\langle \theta_1 -  \theta_2, \mathbf{C}_t^1 D_\theta \Phi_{{i'}(t,\varepsilon;\cdot)}( \theta_1) - \mathbf{C}_t^2 D_\theta \Phi_{{i'}(t,\varepsilon;\cdot)}( \theta_2) \rangle \leq -h_{i'}(t) \| \theta_1 - \theta_2\|^2,
$$

We can split the left hand side into two parts:

\begin{align*}
    &-\langle \theta_1 -  \theta_2, \mathbf{C}^1_t D_\theta \Phi_{{i'}(t,\varepsilon;\cdot)}( \theta_1) - \mathbf{C}^2_t D_\theta \Phi_{{i'}(t,\varepsilon;\cdot)}( \theta_2) \rangle \\ =&-\langle \theta_1 -  \theta_2, \mathbf{C}^1_t D_\theta \Phi_{{i'}(t,\varepsilon;\cdot)}( \theta_1)-\mathbf{C}^1_t D_\theta \Phi_{{i'}(t,\varepsilon;\cdot)}( \theta_2)+\mathbf{C}^1_t D_\theta \Phi_{{i'}(t,\varepsilon;\cdot)}( \theta_2) - \mathbf{C}^2_t D_\theta \Phi_{{i'}(t,\varepsilon;\cdot)}( \theta_2) \rangle\\
    =&-\langle \theta_1 -  \theta_2, \mathbf{C}^1_t \left[D_\theta \Phi_{{i'}(t,\varepsilon;\cdot)}( \theta_1)-D_\theta \Phi_{{i'}(t,\varepsilon;\cdot)}( \theta_2)\right]\rangle\\
    &-\langle \theta_1 -  \theta_2, \left[\mathbf{C}^1_t-\mathbf{C}^2_t\right] D_\theta \Phi_{{i'}(t,\varepsilon;\cdot)}( \theta_2)\rangle
\end{align*}
Substituting the corresponding gradient flows into the equations, we obtain

\begin{align*}
    &-\langle \theta_1 -  \theta_2, \mathbf{C}^1_t \left[D_\theta \Phi_{{i'}(t,\varepsilon;\cdot)}( \theta_1)-D_\theta \Phi_{{i'}(t,\varepsilon;\cdot)}( \theta_2)\right]\rangle\\
    &-\langle \theta_1 -  \theta_2, \left[\mathbf{C}^1_t-\mathbf{C}^2_t\right] D_\theta \Phi_{{i'}(t,\varepsilon;\cdot)}( \theta_2)\rangle\\
    =&-\langle \theta_1 -  \theta_2, \mathbf{C}^1_t \tilde A^T \tilde A (\theta_1 -  \theta_2)\rangle\\
    &-\langle \theta_1 -  \theta_2, \left[\mathbf{C}^1_t-\mathbf{C}^2_t\right] \tilde A^T (\tilde A\theta_2-\tilde y)\rangle\\
\end{align*}

We consider both terms separately. For the first part we obtain

\begin{align*}
-\langle \theta_1 -  \theta_2, \mathbf{C}^1_t \tilde A^T \tilde A (\theta_1 -  \theta_2)\rangle &\leq -\lambda_{\min}(\mathbf{C}^1_t)\lambda_{\min}(\tilde A^T\tilde A)\|\theta_1 -  \theta_2\|\notag\\
&=-\lambda_{\min}(\widehat{C}^1_t+\alpha_{vi}C_{vi})\lambda_{\min}(\tilde A^T\tilde A)\|\theta_1 -  \theta_2\|\notag\\
&\leq -\alpha_{vi}\lambda_{\min}(C_{vi})\left(\min_{i\in\{1,...,\Nsub\}}\lambda_{\min}(\tilde A_{\bsi}^T\tilde A_{\bsi})\right)\|\theta_1 -  \theta_2\|^2,
\end{align*}

where we used the positive definiteness of $\widehat{C}^1_t$ for every $t\geq 0$ in the third step. For the second term we obtain

\begin{align*}
    &-\langle \theta_1 -  \theta_2, \left[\mathbf{C}^1_t-\mathbf{C}^2_t\right] \tilde A^T (\tilde A\theta_2-\tilde y)\rangle\\
    \leq & |\langle \theta_1 -  \theta_2, \left[\mathbf{C}^1_t-\mathbf{C}^2_t\right] \tilde A^T (\tilde A\theta_2-\tilde y)\rangle|\\
    \leq & \|\theta_1 -  \theta_2\| \|\mathbf{C}^1_t-\mathbf{C}^2_t\| \|\tilde A^T (\tilde A\theta_2-\tilde y)\|\\
\end{align*}

We can compare the rates of convergence. Considering the results from above we have
$$\|\theta_1 -  \theta_2\|^2 \in \mathcal{O}(\kappa(t)^2),$$

and also 

$$\|\tilde A^T (\tilde A\theta_2-\tilde y)\|\in \mathcal{O}(\kappa(t)).$$

Finally, we have for the covariance matrices
\begin{align*}
    \|\mathbf{C}^1_t-\mathbf{C}^2_t\|&=\|\frac{1}{J}\sum_{j=1}^J (u^{(j)}_1-\bar{u}_1)(u^{(j)}_1-\bar{u}_1)^T-(u^{(j)}_2-\bar{u}_2)(u^{(j)}_2-\bar{u}_2)^T\|\\
    \leq& \frac{1}{J} \sum_{j=1}^J \|u^{(j)}_1(u^{(j)}_1)^T-u^{(j)}_2(u^{(j)}_2)^T\| \\
    +& \|u^{(j)}_1(\bar{u}_1)^T-u^{(j)}_2(\bar{u}_2)^T\| \\
    +& \|\bar{u}_1 (u^{(j)}_1)^T-\bar{u}_2 (u^{(j)}_2)^T\| \\
    +& \|\bar{u}_1 (\bar{u}_1)^T-\bar{u}_2 (\bar{u}_2)^T\|
\end{align*}

Since we know that all particles converge with rate $\kappa(t)$, the mean values also converge with the same rate. By using triangle inequality we obtain the following

$$\|\mathbf{C}^1_t-\mathbf{C}^2_t\|\in \mathcal{O}(\kappa(t)).$$
Hence

$$\|\theta_1 -  \theta_2\| \|\mathbf{C}^1_t-\mathbf{C}^2_t\| \|\tilde A^T (\tilde A\theta_2-\tilde y)\|\in \mathcal{O}(\kappa(t)^3),$$

showing that the second term converges faster and we can therefore neglect it.\\
$2.$ In batch subsampling with variance inflation the only difference is that the forward operator and data both depend on the particle, i.e. the mapped residuals satisfy
\begin{align*}
    \frac{\mathrm{d} \rho^{(j)}}{\mathrm{d} t}=-A_{\bsi(t;j)}\left[\widehat C_t+\alpha_{vi}C_{vi}\right]A_{\bsi(t;j)}^T\rho^{(j)}(t)
\end{align*}

with $\rho^{(j)}=A_{(\bsi;j)}\theta^{(j)}-\tilde y_{(\bsi;j)}$. The exponential convergence of each particle to the minimiser of the functional $\frac{1}{2}\|\tilde A_{(\bsi;j)}\theta^{(j)}-\tilde y_{(\bsi;j)}\|^2+\frac{\alpha}{2}\|\theta\|_{C_0}$ follows again from standard arguments with the Lyapunov function $\|\tilde A_{(\bsi;j)}^T\rho^{(j)}\|^2$. Hence, the convexity analysis does not change.\\
$3.$
If we do not use variance inflation the gradients $\tilde A_{(\bsi;j)}^\top \rho^{(j)}$ satisfy the following differential equation
\begin{align*}
\frac{\mathrm{d} \tilde A_{(\bsi;j)}^\top \rho^{(j)}(t)}{\mathrm{d}t} &= - \tilde A_{(\bsi;j)}^\top \tilde A_{(\bsi;j)}\widehat{C}_t \tilde A_{(\bsi;j)}^\top \rho^{(j)}(t).
\end{align*}
Again, by basic Lyapunov theory we obtain convergence at an algebraic speed, which however is enough to do the same analysis as above.
Similar to above we obtain\\ $h(t)=\lambda_{\min}(\widehat C_t)\left(\min_{i\in\{1,...,\Nsub\}}\lambda_{\min}(A_{\bsi}^TA_{\bsi})\right)$

 $4.$ Batch subsampling with diminishing variance inflation: Theorem \ref{thm:Expconv} gives us the exponential convergence to the respective solution. Then the convexity analysis is similar to the above calculations and we obtain $h(t)=\frac{\alpha_{vi}}{1+t}\lambda_{\min}(C_{vi})\left(\min_{i\in\{1,...,\Nsub\}}\lambda_{\min}(A_{\bsi}^TA_{\bsi})\right)$.

\end{proof}

\subsection{Subsampling without variance inflation}
\begin{lemma} \label{lemma:eig_ss_app}
For the regularised single-subsampling EKI flow, given by the solution of \eqref{EKI_subsampl_reg}, the following lower bound for the smallest eigenvalue $\lambda_{min}(t)$ of the empirical covariance $\hat{C}(t)$ holds.
\begin{align*}
    \lambda_{min}(t)\geq \left(2ct+\frac{1}{\lambda_{min}(0)}\right)^{-1}.
\end{align*}
\end{lemma}

\begin{proof}

The proof follows the ideas used in \cite[Theorem~3.5]{Tong2020}.\\
W.l.o.g we assume that our initial ensemble is a generator of $\mathbb{R}^d$. Otherwise we consider the dynamics of the corresponding coordinates, which correspond to the minimisation problem \eqref{eqn:constrTEKI}.
The particles satisfy \eqref{EKI_subsampl}. Substituting the covariance matrix

$$\hat{C}_t=\frac{1}{\Nens}\sum_{j=1}^\Nens \left(\theta^{(j)}(t)-\bar{\theta}_t\right)\left(\theta^{(j)}(t)-\bar{\theta}_t\right)^T=\frac{1}{\Nens}\sum_{j=1}^\Nens e^{(j)}_t \left(e^{(j)}_t\right)^T.$$

into the dynamics of the particles gives

\begin{align*}
    \frac{\mathrm{d} \theta^{(j)}(t)}{\mathrm{d}t} &=-\frac{1}{\Nens}\sum_{k=1}^\Nens e^{(k)}_t \left(e^{(k)}_t\right)^T \tilde A_{i(t)} ^T\left(\tilde A_{i(t)} \theta^{(j)}(t)-\tilde y_{i(t)}\right) \\
    &=-\frac{1}{\Nens}\sum_{k=1}^\Nens e^{(k)}_t (\tilde A_{i(t)}  e^{(k)}_t)^T\left(\tilde A_{i(t)} \theta^{(j)}(t)-\tilde y_{i(t)}\right)\\
    &=-\frac{1}{\Nens}\sum_{k=1}^\Nens D_{kj}e^{(k)}_t,
\end{align*}

where we set $D_{kj}:=\langle  \tilde A_{i(t)}  e^{(k)}_t,\left(\tilde A_{i(t)} \theta^{(j)}(t)-\tilde y_{i(t)}\right)\rangle$.\\
Next we consider the dynamics of the weighted particles, i.e.

\begin{align*}
    \frac{\mathrm{d} \bar{\theta}(t)}{\mathrm{d}t}&=\frac{1}{\Nens}\sum_{j=1}^\Nens\frac{\mathrm{d} \theta^{(j)}(t)}{\mathrm{d}t}\\
    &=-\frac{1}{\Nens}\sum_{j=1}^\Nens \frac{1}{\Nens}\sum_{k=1}^\Nens D_{kj}e^{(k)}_t\\
    &=-\frac{1}{\Nens}\sum_{j=1}^\Nens \frac{1}{\Nens}\sum_{k=1}^\Nens \langle  \tilde A_{i(t)}  e^{(k)}_t,\left(\tilde A_{i(t)}  \theta^{(j)}(t)-y_{i(t)}\right)\rangle e^{(k)}_t\\
    &=-\frac{1}{\Nens}\sum_{k=1}^\Nens F_k e^{(k)}_t,
\end{align*}

where we set $F_k:=\langle \tilde A_{i(t)} e^{(k)}_t,\left(\tilde A_{i(t)}  \bar{\theta}-\tilde y_{i(t)}\right)\rangle$.\\
The difference of the scalars $D_{jk}$ and $F_k$ is given by
\begin{align*}
    &\langle  \tilde A_{i(t)}  e^{(k)}_t,\left(\tilde A_{i(t)}  \theta^{(j)}(t)-\tilde y_{i(t)}\right)\rangle-\langle  \tilde A_{i(t)}  e^{(k)}_t,\left(\tilde A_{i(t)}  \bar{\theta(t)}-\tilde y_{i(t)}\right)\rangle\\
    &=\langle  \tilde A_{i(t)}  e^{(k)}_t,\tilde A_{i(t)} e^{(j)}_t\rangle:=E_{kj}.
\end{align*}
Obviously we have $E_{kj}=E_{jk}$. With this we can quantify the dynamics of the centered particles, i.e. 
\begin{align*}
    \frac{\mathrm{d} e^{(j)}(t)}{\mathrm{d}t}&=\frac{\mathrm{d} \theta^{(j)}(t)-\bar{\theta}(t)}{\mathrm{d}t}\\
    &=-\frac{1}{\Nens}\sum_{k=1}^\Nens D_{kj}e^{(k)}_t+\frac{1}{\Nens}\sum_{k=1}^\Nens F_k e^{(k)}_t\\
    &=-\frac{1}{\Nens}\sum_{k=1}^\Nens (D_{kj}-F_k)e^{(k)}_t\\
    &=-\frac{1}{\Nens}\sum_{k=1}^\Nens E_{kj}e^{(k)}_t.
\end{align*}

Finally, we obtain for the dynamics of the empirical covariance $\hat{C}_t$
\begin{align*}
    \frac{\mathrm{d} \hat{C}(t)}{\mathrm{d}t}&=\frac{1}{\Nens}\sum_{j=1}^\Nens \frac{\mathrm{d}}{\mathrm{d}t} \left(e^{(j)}_t\left(e^{(j)}_t\right)^T\right)\notag\\
    &=\frac{1}{\Nens}\sum_{j=1}^\Nens\left(-\frac{1}{\Nens}\sum_{k=1}^\Nens E_{kj}e^{(k)}_t\right)\left(e^{(j)}_t\right)^T+\frac{1}{\Nens}\sum_{j=1}^\Nens e^{(j)}_t\left(-\frac{1}{\Nens}\sum_{k=1}^\Nens E_{kj}e^{(k)}_t\right)^T\notag\\
    &=-\frac{2}{\Nens^2}\sum_{j,k=1}^\Nens E_{kj} e^{(k)}_t \left(e^{(j)}_t\right)^T. 
\end{align*}

Now let $\lambda_{min}(t)$ be the smallest eigenvalue of $\hat{C}(t)$ with unit-norm eigenvector $v(t)$. Then we have

$$0=\frac{\mathrm{d}}{\mathrm{d}t}\|v(t)\|_X^2=2\langle v(t),\frac{\mathrm{d}}{\mathrm{d}t} v(t) \rangle.$$

The dynamics of the smallest eigenvalue are then given by:

\begin{align*}
    \frac{\mathrm{d} \lambda_{min}(t)}{\mathrm{d}t}&=\frac{\mathrm{d}}{\mathrm{d}t}\langle v(t),\hat{C}(t) v(t)\rangle\\
    &=\langle v(t),\frac{\mathrm{d}}{\mathrm{d}t}(\hat{C}(t)) v(t) \rangle+ \langle \frac{\mathrm{d}}{\mathrm{d}t} v(t), \hat{C}(t) v(t) \rangle\\
    &=\langle v(t),\frac{\mathrm{d}}{\mathrm{d}t}(\hat{C}(t)) v(t) \rangle+ \lambda_{min} \langle \frac{\mathrm{d}}{\mathrm{d}t} v(t),v(t) \rangle\\
    &=\langle v(t),\frac{\mathrm{d}}{\mathrm{d}t}(\hat{C}(t)) v(t) \rangle\\
    &=-\frac{2}{\Nens^2}\sum_{j,k=1}^\Nens E_{kj} \langle v(t), e^{(k)}_t \left(e^{(j)}_t\right)^T v \rangle.
\end{align*}

Note that the following holds
$$\langle v(t), e^{(k)}_t \left(e^{(j)}_t\right)^T v \rangle=\langle \left(e^{(k)}_t\right)^T v(t), \left(e^{(j)}_t\right)^T v \rangle=\langle e^{(k)}_t,v \rangle \langle e^{(j)}_t,v \rangle.$$

Substituting this and $E_{kj}$ into the latter equation, gives us

\begin{align*}
    \frac{\mathrm{d} \lambda_{min}(t)}{\mathrm{d}t}&= -\frac{2}{\Nens^2}\sum_{j,k=1}^\Nens \langle  \tilde A_{i(t)} e^{(k)}_t,\tilde A_{i(t)} e^{(j)}_t\rangle \langle e^{(k)}_t,v \rangle \langle e^{(j)}_t,v \rangle\\
    &=-\frac{2}{\Nens^2}\sum_{j,k=1}^\Nens \langle  \tilde A_{i(t)} e^{(k)}_t\langle e^{(k)}_t,v \rangle,\tilde A_{i(t)} e^{(j)}_t\langle e^{(j)}_t,v \rangle\rangle\\
    &=-2 \langle \tilde A_{i(t)}  \hat{C}(t) v(t), \tilde A_{i(t)}  \hat{C}(t) v(t)  \rangle\\
    &=-2 \| \tilde A_{i(t)} \hat{C}(t) v(t)\|^2\\
    &\geq -2 \| \tilde A_{i(t)} \|^2 \lambda_{min}^2(t).
\end{align*}

Setting $c=\max_{i\in \{1,...,\Nsub\}}\| A_{\bsi(t)} \|^2$, we obtain for this ODE the following lower bound for the solution

\begin{equation}
    \lambda_{min}(t)\geq \left(2ct+\frac{1}{\lambda_{min}(0)}\right)^{-1}.\notag
\end{equation}

\end{proof}

\bibliographystyle{siam}
\bibliography{main}
\end{document}